\documentclass[11pt]{article}
\usepackage{amsmath}
\usepackage{amssymb} 
 \usepackage{amsthm}

\usepackage{xcolor}
\usepackage{fullpage}
\usepackage{bm}
\usepackage{xfrac}
\usepackage{tikz}
\usepackage{empheq}
\usepackage{textcomp}
\usepackage[normalem]{ulem}
\usepackage[utf8]{inputenc}
\usepackage[english]{babel}
 
\newcommand{\stkout}[1]{\ifmmode\text{\sout{\ensuremath{#1}}}\else\sout{#1}\fi}
\definecolor{myblue}{rgb}{0.9,0.9,0.98}
\newcommand*\mybluebox[1]{%
	\colorbox{myblue}{\hspace{1em}#1\hspace{1em}}}
\usepackage[shortlabels]{enumitem}
\setlist[enumerate]{leftmargin=10 mm,nosep}

\usepackage[title]{appendix}
\usepackage{fancybox}
\usepackage[doc, wmcom,jvcom]{optional}
\usepackage{xcolor}
\usepackage[colorlinks=true,
linkcolor=refkey,
urlcolor=lblue,
citecolor=red]{hyperref}
\usepackage{float}
\usepackage{soul}
\usepackage{graphicx}
\usepackage{mathpazo}

\definecolor{labelkey}{rgb}{0,0.08,0.45}
\definecolor{refkey}{rgb}{0,0.6,0.0}
\definecolor{Brown}{rgb}{0.45,0.0,0.05}
\definecolor{lime}{rgb}{0.00,0.8,0.0}
\definecolor{lblue}{rgb}{0.5,0.5,0.99}
\definecolor{OliveGreen}{rgb}{0,0.6,0}
\definecolor{orange}{RGB}{255, 140, 0}

\colorlet{hlcyan}{cyan!30}

\usepackage{stmaryrd}
\usepackage{amssymb}
\usepackage[margin=0.95in]{geometry}
\makeatletter
\def\namedlabel#1#2{\begingroup
	\def\@currentlabel{#2}%
	\label{#1}\endgroup
}
\makeatother

\definecolor{mypink}{rgb}{0.87, 0.19, 0.39}
\DeclareMathOperator{\gra}{gr}

\providecommand{\RA}{\Rightarrow}
\providecommand{\LA}{\Leftarrow}

\providecommand{\ball}[2]{\operatorname{ball}(#1;#2)}
\newcommand{\p}{{\bm{p}}}
\newcommand{\q}{{\bm{q}}}
\newcommand*{\tran}{^{\mkern-1.5mu\mathsf{T}}}
\newcommand*{\normM}{_{\mkern-1.5mu{M}}}
\newcommand*{\normX}{_{\mkern-1.5mu{X}}}
\newcommand*{\normY}{_{\mkern-1.5mu{Y}}}
\providecommand{\norm}[1]{\lVert#1\rVert}

\newcommand{{\bB}}{\ensuremath{\mathbf{B}}}

\newcommand{\bzero}{\ensuremath{{\boldsymbol{0}}}}
\newcommand{\minimize}[2]{\ensuremath{\underset{\substack{{#1}}}{\mathrm{minimize}}\;\;#2 }}
\newcommand{\maximize}[2]{\ensuremath{\underset{\substack{{#1}}}{\mathrm{maximize}}\;\;#2 }}
\providecommand{\siff}{\Leftrightarrow}

\newcommand{\knn}{\ensuremath{{k\in{\mathbb N}}}}

\newcommand{\menge}[2]{\big\{{#1}~\big |~{#2}\big\}}

\newcommand{\fenv}[1]%
{\ensuremath{\,\overrightarrow{\operatorname{env}}_{#1}}}
\newcommand{\benv}[1]%
{\ensuremath{\,\overleftarrow{\operatorname{env}}_{#1}}}

\newcommand{\scal}[2]{\left\langle{#1},{#2}  \right\rangle}

\newcommand{\RR}{\ensuremath{\mathbb R}}
\newcommand{\RX}{\ensuremath{\,\left]-\infty,+\infty\right]}}

\providecommand{\ri}{\operatorname{ri}}

\newcommand{\NN}{\ensuremath{\mathbb N}}

\newcommand{\dom}{\ensuremath{\operatorname{dom}}}

\newcommand{\argmin}{\ensuremath{\operatorname{argmin}}}

\newcommand{\prox}{\ensuremath{\operatorname{Prox}}}

\newcommand{\inte}{\ensuremath{\operatorname{int}}}

\newcommand{\ran}{\ensuremath{{\operatorname{ran}}\,}}

\newcommand{\conv}{\ensuremath{\operatorname{conv}\,}}

\newcommand{\spn}{\ensuremath{\operatorname{span}\,}}

\newcommand{\Id}{\ensuremath{\operatorname{Id}}}
\newcommand{\cran}{\ensuremath{\overline{\operatorname{ran}}\,}}

\providecommand{\rras}{\rightrightarrows}
\newcommand{\fix}{\ensuremath{\operatorname{Fix}}}
\providecommand{\fady}{\varnothing}
\usepackage[capitalize,nameinlink]{cleveref}
\crefname{equation}{}{equations}
\crefname{chapter}{Appendix}{chapters}
\crefname{item}{}{items}
\crefname{enumi}{}{}
\crefname{appsec}{Appendix}{Appendices}

\newtheorem{theorem}{Theorem}[section]
\newtheorem{lemma}[theorem]{Lemma}

\newtheorem{corollary}[theorem]{Corollary}
\newtheorem{proposition}[theorem]{Proposition}

\newtheorem{example}[theorem]{Example}

\newtheorem{fact}[theorem]{Fact}
	\newtheorem{remark}[theorem]{Remark}
	
	\theoremstyle{remark}

	\def\endproof{\ensuremath{\hfill \quad \blacksquare}}

 \usepackage{autonum}	
	\renewcommand{\d}{{\bm{d}}}
	\renewcommand{\b}{{\bm{b}}}
	\renewcommand{\c}{{\bm{c}}}
	
    \renewcommand{\a}{{\bm{a}}}

    \renewcommand{\t}{{\bm{t}}}
	\newcommand{\R}{\mathbb{R}}
	
	\renewcommand{\u}{{\bm{u}}}
	\renewcommand{\v}{{\bm{v}}}
	\newcommand{\w}{{\bm{w}}}
	\newcommand{\x}{{\bm{x}}}
	\newcommand{\y}{{\bm{y}}}
	\newcommand{\z}{{\bm{z}}}
    \renewcommand{\r}{{\bm{r}}}
    \newcommand{\s}{{\bm{s}}}
	\newcommand{\bz}{{\bm{0}}}
	\newcommand{\lin}{{\cal{A}}}
	\providecommand{\vx}{\v_R}
	\providecommand{\vy}{\v_D}
	\allowdisplaybreaks
	\newcommand{\be}{\begin{equation}}
		\newcommand{\ee}{\end{equation}}

			\newenvironment{myproof}[1][\proofname]{
					{\emph {#1} }%
				}{\endproof}

	\author{Tao Jiang
 				\thanks{Fundamental AI Research (FAIR)
Meta Superintelligence Labs
Seattle, Washington 98109, USA.
Email: \texttt{taojiang@meta.com}
}\and Walaa M.~Moursi				\thanks{Department of Combinatorics \& Optimization, University of Waterloo, Waterloo, Ontario  N2L~3G1, Canada
					. E-mail:
					\texttt{walaa.moursi@uwaterloo.ca}.}
\and Stephen A.~Vavasis				\thanks{Department of Combinatorics \& Optimization, University of Waterloo, Waterloo, Ontario N2L~3G1, Canada
					. E-mail:
					\texttt{vavasis@uwaterloo.ca}.}}
\title{Range of the displacement operator of PDHG with applications to quadratic and conic programming}

\begin{document}
		\maketitle
		\begin{abstract}
   Primal-dual hybrid gradient (PDHG) is a first-order method for saddle-point problems and convex programming introduced by Chambolle and Pock.  Recently, Applegate et al.\ analyzed the behaviour of PDHG when applied to an infeasible or unbounded instance of linear programming, and in particular, showed that PDHG is able to diagnose these conditions.  
   Their analysis hinges on the notion of the infimal displacement vector in the closure of the range of the displacement mapping of the splitting operator that encodes the PDHG algorithm.
   In this paper, we develop a novel formula for this range using
   monotone operator theory. The analysis is then specialized to conic programming and further to quadratic programming (QP) and second-order cone programming (SOCP).  A consequence of our analysis is that PDHG is able to diagnose infeasible or unbounded instances of QP and of the ellipsoid-separation problem, a subclass of SOCP.
		\end{abstract}

	
		\section{Introduction}
  ``First-order" methods for convex programming use matrix-vector multiplication as their principal operation.
		For huge-scale convex programming problems, first-order methods appear to be the only tractable approach.  In a recent survey, Lu \cite{Lu} found that, among first-order methods, primal-dual hybrid gradient (PDHG) appears to be the best in practice {for linear programming (LP)}.  PDHG was introduced by Chambolle and Pock \cite{CP}.  It has been shown by O’Connor and Vandenberghe \cite{Vdb} that PDHG may be regarded as a particular form of the Douglas-Rachford iteration.
In the following, we assume that
		\begin{empheq}[box=\mybluebox]{equation}
			\label{eq:def:XY}
			\text{$X$ and $Y$ are real Hilbert spaces}
		\end{empheq}
		with corresponding inner products $\scal{\cdot}{\cdot}\normX$ 
		(respectively $\scal{\cdot}{\cdot}\normY$)
		and induced norms $\norm{\cdot}\normX$
		(respectively $\norm{\cdot}\normY$)\footnote{When it is clear from the context, we will drop the subscripts $X$ and $Y$
			associated with the inner products and the norms.}.
   		We also assume that 
		\begin{empheq}[box=\mybluebox]{equation}
			\label{eq:def:fg}
			% g:\RR^n \to \RX, \quad f:\RR^m \to \RX,
			f:X \to \RX, \quad g:Y \to \RX,
		\end{empheq}
		are convex lower semicontinuous and proper,
		and that 
		\begin{empheq}[box=\mybluebox]{equation}
			\label{eq:def:A}
			\text{$\lin:X\to Y$ is linear and continuous}.
		\end{empheq}

		PDHG is an algorithm for general saddle-point problems of the form 
  \be
  \inf_{\x\in X}\sup_{\y\in Y} f(\x)-g^*(\y)+\langle\y,\lin \x\rangle.
  \label{eq:saddle}
  \ee
  Here, $g^*$ denotes the \emph{Fenchel--Legendre conjugate} of $g$.  
  It should be noted that \cref{eq:saddle} 
  is equivalent to $\inf_{\x\in X} f(\x)+g(\lin \x)$ 
  since the inner sup of \cref{eq:saddle} is exactly the formula for conjugation of $g^*$.  We return to this point in~\cref{sec:PDHGFen}. 
{Recall the \emph{proximal mapping} of $f$ at $\x\in X$ }
 is defined by 
  \[\prox_f(\x)=\argmin_{\z\in X}\left\{f(\z)+\frac{1}{2}\Vert\z-\x\Vert^2\right\},\]
and that the operator form for the proximal mapping is
  \be
  \prox_f\equiv (\Id+\partial f)^{-1}.
  \label{eq:proxop}
  \ee
{Let $(\x_0,\y_0)\in X\times Y$. The PDHG iteration 
  updates $(\x_0,\y_0)$ as follows $(\forall \knn)$:}
  \begin{subequations}
			\begin{align}
				\x_{k+1} &:= \prox_{\sigma f}(\x_k-\sigma \lin^*\y_k), \label{eq:PDHGx}\\
				\y_{k+1} &:= \prox_{\tau g^*}\left(\y_k+\tau \lin(2\x_{k+1}-\x_k)\right).
				\label{eq:PDHGy}
			\end{align}
   \label{eq:PDHGupdates}
		\end{subequations}
  Here, $\sigma,\tau>0$ are step-size parameters that must be chosen correctly---refer to \cref{prop:M}.
  Thus, the main work on each iteration consists of multiplication by $\lin$ and $\lin^*$ and two prox operations.  

{Observe that \cref{eq:PDHGupdates} can be written as
$(\x_{k+1},\y_{k+1}):=T(\x_k,\y_k)$ where:}
$$
T(\x,\y)
  =\left( 
  \begin{array}{c}
  \prox_{\sigma f}(\x-\sigma \lin^*\y) \\
	\prox_{\tau g^*}\left(\y+\tau \lin(2\prox_{\sigma f}(\x-\sigma \lin^*\y)-\x)\right)
 \end{array}
\right).
\label{eq:T}
$$
Let $(\x^*,\y^*)\in X\times Y$.
Then,  assuming a constraint qualification,
$(\x^*,\y^*)\in\fix T:=\menge{(\x,\y)\in X\times Y}{(\x,\y)=T(\x,\y)}$  or
equivalently
$(\Id-T)(\x^*,\y^*)=\bzero$, 
if and only if $(\x^*,\y^*)$ is a solution to 
\cref{eq:saddle} (see, e.g., \cref{prop:T:op} below).  
On the other hand, if $\bzero\notin\ran(\Id-T)$
then $\norm{(\x_k,\y_k)}\to \infty$ 
(see, e.g., \cite[Corollary~2.2]{BBR78}).  
This motivates the exploration 
of the set $\cran(\Id -T)$ and the corresponding well-defined 
\emph{infimal displacement vector} (see \cref{e:def:v:projm} below).
In passing we point out that 
the study of the range of the displacement mapping
associated with splitting algorithms; namely Douglas--Rachford and forward-backward algorithms,
was a key ingredient 
in exploring the static structure and the dynamic 
behaviour of these methods in the inconsistent case.
In this regard, we refer the reader to \cite{BM16,BM2015,BM22,B21,LRY,Moursithesis,FB,RLY}.

{In the case of PDHG, recently} Applegate et al.\ \cite{Applegate} showed that in the case of inconsistent linear programming (infeasible or unbounded), the infimal displacement vector characterizes the limiting difference between successive PDHG iterates and also certifies the infeasibility or unboundedness.

Our main results can be summarized as follows:
\begin{enumerate}
\item 
We provide a novel formula
 for $\cran (\Id-T)$
 in terms of the domains 
 of the functions $f,g, f^*, g^*$
  and $\lin$ (see \cref{prop:T:prop} below).
  Along the way,  we obtain a formula for the 
  range of the sum of {a skew symmetric operator (of the form \cref{eq:def:S} below) and a maximally monotone operator with} a specific structure (see \cref{thm:ranform:mm} below).
\item 
When specializing\footnote{Let $S\subseteq X$. Here and elsewhere we use
$\iota_S$ to denote the \emph{indicator function} of $S$
defined as: $\iota_S(\x)=0$ if $\x\in S$;
 and $\iota_S(\x)=+\infty$ if $\x\in X\smallsetminus S$.
It is well known that if $S$ is closed, convex, and nonempty, 
then $\iota_S$ is a proper l.s.c.\ convex  function.} 
$g=\iota_K(\cdot-b)$,
where $K$ is a nonempty closed convex cone of $Y$,
we obtain powerful properties 
for the infimal displacement vector  
in $\cran (\Id-T)$ (see \cref{lem:qp:v:K} below).
\item
We present a comprehensive analysis of the behaviour of
PDHG when applied to QP. More specifically, we prove that the infimal displacement
vector $\v$ 
in $\cran(\Id-T)$ provides certificates of 
inconsistency (see \cref{lem:qp:infeasibility} below).
In \cref{t:QPconvergence} below
we prove that the sequence 
$((\x_k,\y_k)+k\v)_\knn$ converges as $k\rightarrow\infty$.
% , where $\v$ is the infimal displacement vector 
% in $\cran(\Id-T)$.

\item 
We analyze the ellipsoid separation problem, another instance of conic programming. For this problem, we also derive a
convergence result (see \cref{t:ellips} below), and we establish again that PDHG in the inconsistent case returns a certificate (see \cref{thm:esep:cert} below).
\end{enumerate}

% \begin{center}
%      [Reorganizing (iii)-(v) by type of result]
% \end{center}

% \begin{enumerate}
% \setcounter{enumi}{2}
% \item 
% We establish new convergence results for the PDHG sequence when applied to certain conic problems in the presence of inconsistency.  For QP
% we prove (\cref{t:QPconvergence}) that the sequence 
% $(\x_k,\y_k)+k\v$ converges as $k\rightarrow\infty$, where $\v$ is the infimal displacement vector 
% in $\cran(\Id-T)$. For the ellipsoid separation problem, we develop a new convergence result in \cref{t:ellips}.

% \item 
% We establish that the infimal displacement vector of PDHG, in the inconsistent case, encodes a certificate of infeasibility for QP (\cref{lem:qp:infeasibility}).  An analogous result is developed for the ellipsoid separation in 
% \cref{thm:esep:cert}.
% \end{enumerate}

{\bf Organization.} 
The rest of this paper is organized as follows.
In \cref{Appp:1} we present a formula for the range of the sum of 
two maximally monotone operators  that have 
particular structures.
In \cref{sec:PDHGrangeIdminusT} we develop formulas for $\Id-T$, the displacement operator of PDHG, and its range.
  In \cref{sec:conicpdhg} and the following sections, 
  we consider the several specializations of PDHG to convex optimization problems.
  \cref{sec:qpstatic} presents
our analysis of $\ran(\Id-T)$ and the infimal displacement vector in the case of QP. 
Furthermore, when the problem is inconsistent, $\v$ is nonzero and certifies inconsistency. 
We provide \textcolor{black}{a numerical example} that illustrates our conclusions.
Another special case of the general conic programming problem is 
presented 
in \cref{sec:standconic}. An application of this setting to the 
ellipsoid separation problem is detailed.
We also prove that PDHG can diagnose infeasible instances of the ellipsoid separation problem.

		\section{On the range of the sum of monotone operators}
		\label{Appp:1}
Let $B\colon X\rras X$. Recall that 
  $B$ is \emph{monotone} if  $(\forall (\x,\u)\in \gra B)$ $(\forall (\y,\v)\in \gra B)$ 
  $\scal{\x-\y}{\u-\v}\ge 0$ and $B$ is \emph{maximally monotone} if it is monotone and its
  graph does not admit any proper extension (in terms of set inclusion).
		In this section, we derive a formula for the range of the sum of two maximally monotone operators of the form \cref{e:mm:sumSB} below. This will play a critical role in our 
		analysis later.
		In the following, we assume that 
		\begin{empheq}[box=\mybluebox]{equation}
			\text{$B_1\colon X\rras X$ and $B_2\colon Y\rras Y$ are maximally monotone.}
		\end{empheq}
  
For the remainder of the paper, we set
  		\begin{empheq}[box=\mybluebox]{equation}
			\label{eq:def:S}
			S:=
			\begin{pmatrix}
				% 0&\lin\tran\\
				% -A&0
				0&\lin^*\\
				-\lin&0
			\end{pmatrix}.
		\end{empheq}

		We now define the maximally monotone operator (see, e.g., \cite[Proposition~20.23]{BC2017})
		\begin{empheq}[box=\mybluebox]{equation}
			\label{eq:def:bB}
			{\bf B}\colon X\times Y\rras X\times Y\colon
			(\x,\y) \mapsto B_1\x\times B^{-1}_2\y.
			%\RR^n\to \RX: f(x)+h(Ax)
		\end{empheq}
		Observe that by, e.g., \cite[Proposition~2.7(i)--(iii)]{LBC11} we
     have $\bB$ is maximally monotone,
		$S\colon X\times Y\to X\times Y$ is maximally monotone 
        with $S^*=-S$ and 
		\begin{equation}
			\label{e:mm:sumSB}
			\text{${\bB}+S$ is maximally monotone.}
		\end{equation}	

		Let $C\colon X\rras X$ be monotone. Recall that $C$ is
		$3^*$ monotone if  $(\forall (\s,\r) \in \dom C\times \ran C)$
		\begin{equation}
			\label{e:3*mono:mm}
			\inf_{(\u,\v)\in \gra C}\;\;
			\scal{\u-\s}{\v-\r}>-\infty.
		\end{equation}
		\begin{fact}
			\label{fact:subd3*}
			$\partial f$ is $3^*$ monotone.	
		\end{fact}	
		\begin{proof}
			See, e.g., \cite[Example~25.13]{BC2017}. 
		\end{proof}	
  \begin{lemma}
  \label{lem:Snot3*}
Suppose that $\lin\neq 0$.
Then $S$ is \emph{not} $3^*$ monotone.
\end{lemma}
\begin{proof}
Combine \cite[Proposition~25.12]{BC2017}
 and 
 \cite[Example~4.5]{BBBRW07}.
\end{proof}
		
		The following result provides a formula for the closure of the range of
		${\bf B}+S$.
		\begin{theorem}
			\label{thm:ranform:mm}
			Suppose that $B_1$ and $B_2$ are $3^*$ monotone.
			Then 
			\begin{equation}
				\label{e:ran:form}
				\cran{({\bf B}+S)}
				=\overline{(\ran B_1+\lin^* (\ran B_2))} \times 
				\overline{(\dom B_2 -\lin (\dom B_1))}.
			\end{equation}
		\end{theorem}
        \begin{proof}
See Appendix~\ref{App:1}.
        \end{proof}

		\begin{remark}
			Suppose that $B_1$ and $B_2$ are $3^*$ monotone. Then ${\bB}$ is $3^*$ monotone.
			Some comments are in order.
			\begin{enumerate}
				\item 
				Suppose that $\dom {\bB}=X\times Y$.
				This is equivalent to $\dom B_1=X$ and $\ran B_2=\dom B_2^{-1}=Y $.
				The formula in \cref{e:ran:form} boils down to
				
				\begin{equation}
					\label{e:ran:formBH}
					\cran{({\bf B}+S)}
					=\overline{(\ran B_1+\ran \lin^*)} \times 
					\overline{(\dom B_2 -\ran \lin )}=\overline{\ran {\bB}+\ran S}.
				\end{equation}
				The above formula is alternatively obtained using the celebrated Brezis--Haraux theorem,
				see, e.g., \cite[Theorem~25.24(ii)]{BC2017}.
				
				\item The assumption that 
				$\dom {\bB}=X\times Y$ is critical to prove  that  the formula in \cref{e:ran:form} 
				reduces to the formula in 
				\cref{e:ran:formBH} as we illustrate in 
				\cref{ex:notBHfriend} below. 
				
				\item 
				The assumption that both $B_1$ and $B_2$ are $3^*$ monotone
				is critical to obtain the conclusion of \cref{thm:ranform:mm} 
				as we illustrate in \cref{ex:both3*} below.
			\end{enumerate}		
		\end{remark}	
		
		\begin{example}
			\label{ex:notBHfriend}
			Suppose that $X=Y=\RR$, that $\lin=-\Id$,
			that\footnote{Let $C$
   be a nonempty closed convex subset of $X$.
   Here and elsewhere we use $N_C$ to denote 
   the \emph{normal cone operator}
   associated with $C$ defined by $N_C(x):=\menge{u\in X}{\sup\scal{u}{C-x}=0}$, if $x\in C$;
   and $N_C(x):=+\infty$, otherwise.}
   $B_1=N_{\left[0,+\infty\right[}$
			and that $B_2=N_{\left]-\infty,0\right]}$.
			Then 
			\begin{equation}
				\cran({\bB}+S)=	\left]-\infty,0\right]  \times \RR \subsetneqq \RR^2 =\overline{\ran {\bB}+\ran S}.
			\end{equation}
		\end{example}

		\begin{proof}
        See  Appendix~\ref{App:1}.
			% Observe that $B_1$ and $B_2$ are subdifferential operators, hence $3^*$ monotone 
			% by \cref{fact:subd3*}. 
			% Moreover,
			% 	\text{ ${\bB}=N_{\RR^2_+}$ and $S=\begin{psmallmatrix}
			% 			0&-1
			% 			\\
			% 			1&0
			% 		\end{psmallmatrix}	$}	
			% On the one hand, clearly
			% 	$\ran {\bB}=\RR^2_{-}$ and $\ran S=\RR^2$\; hence $\ran {\bB}+\ran S=\RR^2=\overline{\ran {\bB}+\ran S}$.
			% On the other hand,
			% 	$\dom B_1=\ran B_2=\left[0,+\infty\right[$ and  $\dom B_2=\ran B_1=\left]-\infty,0\right]$,
			% and \cref{e:ran:form} yields
			% $\cran({\bB}+S)=	\left]-\infty,0\right]  \times \RR \subsetneqq \RR^2 =\overline{\ran {\bB}+\ran S}$,
			% as claimed.	
		\end{proof}

		\begin{example}
			\label{ex:both3*}
			Suppose that $X=\RR$, that $Y=\RR^2$, that 
			$B_1\equiv 0$, that 
			$B_2=\begin{psmallmatrix}
				0&-1
				\\
				1&0
			\end{psmallmatrix}$,
			and that $\lin=\begin{psmallmatrix}
				1
				\\
				0
			\end{psmallmatrix}$.
			Then 
			the following hold:
			\begin{enumerate}
				\item 
				\label{ex:both3*:i}
				$\cran({\bB}+S)=\spn\{(0,1,0)\tran, (-1,0,1)\tran\}  \subsetneqq \RR^3$.
				\item 	
				\label{ex:both3*:ii}
				$\overline{(\ran B_1+\lin^* (\ran B_2))} \times 
				\overline{(\dom B_2 -\lin (\dom B_1))}=\RR^3$.
				\item 
				\label{ex:both3*:iii}
				$	\cran({\bB}+S) \subsetneqq \overline{(\ran B_1+\lin^* (\ran B_2))} \times 
				\overline{(\dom B_2 -\lin (\dom B_1))}.$
			\end{enumerate}
		\end{example}	
		\begin{proof}
        See  Appendix~\ref{App:1}.
% 			It follows from \cref{lem:Snot3*} that
% 			$B_2$ is \emph{not} $3^*$ monotone. 
%    Moreover, it is easy to verify that
% $B_2^{-1}=-B_2$,
% 			that $\dom B_1=\RR$, 
% 			that $\ran B_1=\{0\}$,
% 			and that $\dom B_2=\ran B_2=\RR^2$.
% 			\cref{ex:both3*:i}: We have 
% 			\begin{equation}
% 				{\bB}=
% 				\begin{pmatrix}
% 					0&0&0
% 					\\
% 					0&0&1
% 					\\
% 					0&-1&0
% 				\end{pmatrix}		
% 				\;
% 				\text{and}
% 				\;
% 				S=
% 				\begin{pmatrix}
% 					0&1&0
% 					\\
% 					-1&0&0
% 					\\
% 					0&0&0
% 				\end{pmatrix}		
% 				\; \text{hence}\;
% 				{\bB}+S=
% 				\begin{pmatrix}
% 					0&1&0
% 					\\
% 					-1&0&1	\\
% 					0&-1&0
% 				\end{pmatrix}.
% 			\end{equation}	
% 			This proves \cref{ex:both3*:i}.
% 			\cref{ex:both3*:ii}:
% 			Indeed, we have 
% 			$\ran B_1+\lin^* (\ran B_2)=\{0\}+\lin^*(\RR^2)=\RR$
% 			and 
% 			$\dom B_2 -\lin (\dom B_1)=\RR^2$.
% 			\cref{ex:both3*:iii}:
% 			Combine \cref{ex:both3*:i} and \cref{ex:both3*:ii}.
		\end{proof}

		\section{The PDHG splitting operator, the range of its displacement map, and the infimal displacement vector}
  \label{sec:PDHGrangeIdminusT}
	In this section, we apply the theory from \cref{Appp:1} 
 to develop formulas $\Id-T$, the displacement operator of PDHG, and its range.  These formulas appear in \cref{prop:T:prop}.  Then in \cref{subsec:infimal} we define the infimal displacement vector, which lies in the closure of this range, and state and prove some of its important properties.  
Let $\sigma>0$ and let $\tau>0$. For the remainder of the paper, we set
 \begin{empheq}[box=\mybluebox]{equation}
			\label{eq:def:M}
			M:=
			\begin{pmatrix}
				\tfrac{1}{\sigma}\Id\normX&-\lin^*\\
				-\lin&\tfrac{1}{\tau}\Id\normY
			\end{pmatrix},
		\end{empheq}
		and we set 
		\begin{empheq}[box=\mybluebox]{equation}
			\label{eq:def:F}
			F\colon 
			X\times Y\to \RX
   \colon (\x,\y)\mapsto f(\x)+g^*(\y).
		\end{empheq}
		Then (see, e.g., \cite[Proposition~16.9]{BC2017})
		\begin{equation}
			\partial F(\x,\y)=\partial f(\x) \times \partial g^*(\y).	
		\end{equation}	
 
\subsection{PDHG and Fenchel--Rockafellar duality}
\label{sec:PDHGFen} 
 Consider the primal problem
		\begin{equation}
			\label{e:FRp}
			\minimize{\x\in X}\quad f(\x)+g(\lin \x),	
		\end{equation}	
  which is equivalent to \cref{eq:saddle},
		and its Fenchel--Rockafellar dual
		\begin{equation}
			\label{e:FRd}
			\minimize{\y\in Y}\quad f^*(-\lin^* \y)+g^*(\y).	
		\end{equation}

  Under appropriate constraint qualifications (see, e.g., \cite{b2010} and \cite{bgw2009}
  or \cite[Proposition~4.1(iii)]{LBC11})
		\cref{e:FRp} boils down to solving the primal inclusion:
		\begin{equation}
			\label{e:p:minc}
			\text{find  $\x\in X$ such that }	0\in \partial f(\x)+\lin^* \partial g(\lin \x)
		\end{equation}	
		while \cref{e:FRd} boils down to solving the dual inclusion: 
  \begin{equation}
			\label{e:d:minc}
			\text{find  $\y\in Y$ such that $(\exists \x\in X)$ }	-\lin^* \y\in \partial f (\x)\text{  and   } \y\in  \partial g(\lin \x).
		\end{equation}	
		
Following \cite{bh2013} , we say that $(\overline{\x},\overline{\y})\in X\times Y$
		is a primal-dual solution to both \cref{e:p:minc} and \cref{e:d:minc} if 
\begin{equation}
			\label{e:pd:msol}
			-\lin^* \overline{\y}\in \partial f(\overline{\x})\;\text{  and   }\; \overline{\y}\in  \partial g(\lin \overline{\x}).
\end{equation}	
  One checks that the existence of a solution to \cref{e:pd:msol} implies the existence of a solution to  both 
	\cref{e:p:minc} and 	\cref{e:d:minc}.  Conversely, a solution to either \cref{e:p:minc} or \cref{e:d:minc} implies the existence of a solution to \cref{e:pd:msol}.
		
		The following result is part of the literature, see, e.g., \cite{bh2013}, \cite{condat13} and \cite{vu13}.
		We include a proof for the sake of completeness.
		Recall that $S$ was defined in \cref{eq:def:S},  $M$  in \cref{eq:def:M}, and $F$ in \cref{eq:def:F}. Recall $\prox_f=(\Id+\partial f)^{-1}$,
		and $\partial f^*=(\partial f)^{-1}$.	
		\begin{proposition}
			\label{prop:T:op}
			Let $(\x,\y)\in X\times Y$ .
			We set 	
			\begin{equation}
				\label{e:pdhg:gen:update}
				\begin{pmatrix}
					\x^+\\
					\y^+
				\end{pmatrix}=		
				%\begin{pmatrix}
				%	\prox_{g}(x-\lin\tran y)\\
				%	y+A(2x^+-x)-b.
				%\end{pmatrix}.
				\begin{pmatrix}
					\prox_{\sigma f}(\x-\sigma \lin^* \y)\\
					\prox_{\tau g^*}(\y+\tau \lin(2\x^+-\x))
				\end{pmatrix}.
			\end{equation}
			Let $(\overline{\x},\overline{\y})\in X\times Y$. 
   We set  $T=(M+{\partial F}+S)^{-1}M$.
			Then the following hold.
			\begin{enumerate}
				\item 
				\label{prop:T:op:i}
				We have
				\begin{equation}
					\label{e:def:Tgen}
					\begin{pmatrix}
						\x^+\\
						\y^+
					\end{pmatrix}=		
					T\begin{pmatrix}
						\x\\
						\y
					\end{pmatrix}.
				\end{equation}
				\item 
				\label{prop:T:op:ii}
				Recalling \cref{e:pd:msol} we have 
				$(\overline{\x},\overline{\y})$ is a primal-dual solution of
				\cref{e:p:minc}-\cref{e:d:minc} if and only if $(\overline{\x},\overline{\y})\in \fix T$.
			\end{enumerate}
		\end{proposition}	
		\begin{proof}
See Appendix~\ref{App:2}.
		\end{proof}

 \subsection{The range of $\Id-T$}
 In this subsection we apply the results of \cref{Appp:1} to obtain a formula for $\cran(\Id-T)$, starting from the formula for $T$ given by \cref{prop:T:op}\cref{prop:T:op:i}.
 A key operator in our analysis is $M^{-1}(\partial F + S)$.  The significance of this operator becomes apparent below in \cref{prop:T:prop}, and hence we set the stage with some preliminary results about $M$ and $\partial F+ S$.
We start by observing that 
		\cref{e:mm:sumSB} applied with $\bB$ replaced by $\partial F$ implies that
		\begin{equation}
			\label{e:mm:sumSF} 
			\text{$\partial F+S$ is maximally monotone on $X\times Y$.}
		\end{equation}		
  For the remainder of the paper, we impose the assumption that $\sigma,\tau$ are chosen to satisfy
  \begin{empheq}[box=\mybluebox]{equation}
  \sigma\tau\Vert\lin\Vert^2<1.
  \end{empheq}
		The following lemma is straightforward to verify.
		We include the proof for the sake of completeness.
		
		\begin{lemma}
			\label{prop:M}
			$M$ is self-adjoint. Moreover,
			% under the assumption $\sigma \tau \norm{\lin}^2<1$
			we have: 
			\begin{enumerate}
				\item
				\label{prop:M:0}
				$M$ and $M^{-1}$ are strongly monotone\footnote{Let $B\colon X\rras X$
    be monotone and let $\beta>0$. Then $B$ is \emph{$\beta$-strongly monotone} if $B-\beta \Id$ is monotone.}.
				\item
				\label{prop:M:i}
				$M$ and $M^{-1}$ are surjective.
				\item
				\label{prop:M:ii}
				$M$ and $M^{-1}$ are injective.
                \item
				\label{prop:M:i:ii}
				$M$ and $M^{-1}$ are bijective.
				\item
				\label{prop:M:iii}
				$M$ and $M^{-1}$ are maximally monotone.	
			\end{enumerate}		
		\end{lemma}	
		\begin{proof}
        See Appendix~\ref{App:2}.
		\end{proof}

		Recalling that $M$ is positive definite, in the following we let 
		$Z$ be the Hilbert space obtained 
		by endowing $X\times Y$ with the inner product and induced norm 
		\begin{empheq}[box=\mybluebox]{equation}
			\label{e:hilbertZ}
			\text{$\scal{\cdot}{\cdot}\normM:Z\times Z\to \RR:(\u,\v)\mapsto\scal{\u}{M\v}$\;\;\; and\;\;\;$\norm{\u}\normM=\sqrt{\scal{\u}{M\u}}$}
		\end{empheq}
		respectively.
		
		\begin{lemma}
			\label{lem:maxmonosum:op}
			$M^{-1}({\partial F}+S)\colon Z\rras Z$ is maximally monotone.	
		\end{lemma}	
		\begin{proof}
			Combine \cref{e:mm:sumSF}  with
			\cite[Proposition~20.24]{BC2017} in view of 
			\cref{prop:M}\cref{prop:M:0}.
		\end{proof}	
  
Let $D\subseteq X$
 and let $L\colon X\to X$ be linear  and continuous.
 It is easy to verify that 
\begin{equation}
\label{eq:lin:cl}
\overline{L (D)}=\overline{L (\overline D)}.
\end{equation}
		The following result is of central importance in our work.
		\begin{theorem}
			\label{thm:ranform}
			We have 
				\begin{equation}
					\cran{(\partial F+S)}
					=\overline{(\dom f^*+\lin^* (\dom g^*))} \times 
					\overline{(\dom g -\lin(\dom f))}.
				\end{equation}
			% \end{enumerate}		
		\end{theorem}
        \begin{proof}
See Appendix~\ref{App:2}.
\end{proof}

We are now ready to prove the main result in this section.
		\begin{theorem}
			\label{prop:T:prop}
			Let
			$T=(M+\partial F+S)^{-1}M$, that is, the PDHG operator introduced in \cref{prop:T:op}.
			Then the following hold:
			\begin{enumerate}
				\item
				\label{prop:T:prop:i}
				$T=(\Id+M^{-1}(\partial F+S))^{-1}$.
				\item  
				\label{prop:T:prop:ii}
				$T\colon Z\to Z$ is firmly nonexpansive\footnote{Let $T\colon X\to X$.
    Then $T$ is \emph{firmly nonexpansive} if $(\forall (x,y)\in X\times X)$
		$\norm{Tx-Ty}^2+\norm{(\Id-T)x-(\Id-T)y}^2\le \norm{x-y}^2$.}.   
				\item 
				\label{prop:T:prop:iii}
				$\Id-T=(\Id+(\partial F+S)^{-1}M)^{-1}$.
				\item 
				\label{prop:T:prop:iv}
				$\ran (\Id-T)=M^{-1}(\ran (\partial F+S))$.
				\item 
				\label{prop:T:prop:v}
				$\cran (\Id-T)=M^{-1}(\cran (\partial F+S))$.
				\item 
				\label{prop:T:prop:vi}
				$\cran (\Id-T)=M^{-1}\big( \overline{(\dom f^*+\lin^* (\dom g^*))} \times 
				\overline{(\dom g -\lin(\dom f))}\big).$
			\end{enumerate}
		\end{theorem}
\begin{proof}
See Appendix~\ref{App:2}.
\end{proof}

\begin{remark}
\textcolor{black}{The work by O'Connor and Vandenberghe \cite{Vdb} shows that PDHG iterates can be mapped to Douglas–Rachford iterates and vice versa. However, the mapping is at the level of iterates and does not provide a closed-form correspondence between the underlying operators.
As in the present paper and in the Douglas–Rachford case, the proof techniques rely on the operator form to derive the expression for the range.  Due to the nature of the correspondence, being defined via iterates rather than a direct operator identity, it does not appear that the analysis from the Douglas–Rachford setting can be directly or more easily transferred to the PDHG context.
Nonetheless, the recent 2025 work in \cite{Bui25} introduces a unifying framework that
recovers the range results for some of the splitting operators as special cases.
We point out that the author cited the \texttt{arXiv} version of this paper.}
\end{remark}
\subsection{The infimal displacement vector associated with $T$}
\label{subsec:infimal}
We point out that 
\cref{prop:T:prop}
is a powerful and instrumental tool 
in analyzing the behaviour of PDHG in
the inconsistent case in view \cref{prop:T:op}\cref{prop:T:op:ii}.
Indeed, if 
$(\bzero,\bzero)\not \in \ran(\Id-T)$  then $T$ does not have a fixed point.
More concretely, in this section, we study the minimal norm element in $\cran(\Id-T)$.
As we shall see below, this allows us to diagnose
if
the inconsistency comes from the primal problem or from the dual problem or from both.
		In this section, we assume that
		\begin{empheq}[box=\mybluebox]{equation}
			\text{$T\colon Z\to Z$ is defined as in \cref{e:def:Tgen}.}
		\end{empheq}
		Because $T$ is firmly nonexpansive (see \cref{prop:T:prop}\cref{prop:T:prop:ii}) we learn 
		by, e.g., \cite[Example~20.29]{BC2017}
		that
		\begin{equation}
			\label{e:Id-tmm}
			\text{$\Id -T\colon Z\to Z$ is maximally monotone.}	
		\end{equation}	
  It, therefore, follows from \cref{e:Id-tmm} and, e.g., \cite[Corollary~21.14]{BC2017} that $\cran(\Id-T)$ is convex.
Recalling \cref{e:hilbertZ} we now define the \textit{infimal displacement vector} $\v$
		\begin{empheq}[box=\mybluebox]{equation}
			\label{e:def:v:projm}
		\v=(\vx,\vy)
			=\argmin_{\w\in \cran(\Id-T)}\norm{\w}\normM.
		\end{empheq}
 The vector $\v$ has a beautiful interpretation. Indeed, in some sense, $\norm{\v}$ can be viewed as a measurement of how far our problem is from being consistent.
We now have the following useful facts.
		\begin{fact}
			\label{fact:fejer}
			Suppose $\v\in \ran(\Id-T) $,
			equivalently,  $\fix (\v+T)\neq \fady$. 
			Let $\z_0\in Z$. $(\forall \knn)$
update via:
   \begin{equation}
				\z_{k+1}=T\z_k.
			\end{equation}	
			Then 
			$(\z_k+k\v)_\knn$ is Fej\'er 	monotone\footnote{Let $(\x_k)_\knn$ be a sequence in $X$ and let $C$ be a nonempty subset of
				$X$. Then $(\x_k)_\knn$ is Fej\'er 	monotone with respect to $C$
				if
				$
				(\forall \c\in C) (\forall \knn) \quad \norm{\x_{k+1}-\c}\le \norm{\x_k-\c}.	
				$}
			with respect to
			$\fix (\v+T)$, hence $(\z_k+k\v)_\knn$ is bounded.
		\end{fact}
		\begin{proof}
			See \cite[Proposition~2.5]{BM2015}.
		\end{proof}	
  		\begin{fact}
			\label{fact:fixvT}
			Suppose that  $\z_0\in \fix(\v+T)$.
			Then $\z_0-\RR_{+}\v\subseteq \fix(\v+T)$.
		\end{fact}	
		\begin{proof}
			See \cite[Proposition~2.5(i)]{BM2015}.	
		\end{proof}	
		
		\begin{fact}
			\label{fact:T:v}
			Let $\z_0\in Z$. Update via:
			$\z_{k+1}=T\z_k$. Then the following hold:
			\begin{enumerate}
				\item
				\label{fact:T:v:i}
				{\rm(Pazy)} $\frac{\z_k}{k}\to -\v$.
				\item
				\label{fact:T:v:ii}
				{\rm(Baillon--Bruck--Reich)} $\z_k-\z_{k+1}\to \v$.
			\end{enumerate}	
		\end{fact}
		\begin{proof}
			\cref{fact:T:v:i}:
			See \cite{Pazy1970}.
			\cref{fact:T:v:ii}:
			See \cite{BBR78} or \cite{Br-Reich77}.
		\end{proof}
\begin{remark}
Some comments are in order.
\begin{enumerate}
\item \cref{fact:T:v}\cref{fact:T:v:i}\&\cref{fact:T:v:ii} reveal that when $\v\neq \bzero $ 
  PDHG is able to diagnose inconsistent problems.   
  Furthermore, as we shall see in \cref{lem:qp:infeasibility} and again in \cref{thm:esep}, $\v$ may also carry a certificate of inconsistency.  
 \item 
 On the other hand, if $\bzero\in\cran(\Id-T)\smallsetminus\ran(\Id-T)$
then \cref{fact:T:v}\cref{fact:T:v:i}\&\cref{fact:T:v:ii} do not tell whether the problem is inconsistent.
Alternatively, one could use 
\cite[Corollary~2.2]{BBR78} 
 and monitor the sequence $(\norm{\z_k})_\knn$.
\item In passing, we point out that in the applications we study in upcoming sections, we prove that
$\v\in \ran(\Id-T)$. More precisely, in
the QP case (see \cref{sec:qpstatic} below)
 we show that $\ran(\Id-T)$ is closed 
 and the inclusion follows, while in the case of the ellipsoid separation problem in \cref{sec:standconic}
 below
 we prove the inclusion directly.
\end{enumerate}
\end{remark}		
		We conclude this section with the following proposition, which further characterizes and establishes properties of $\v$ used later.
\begin{proposition}
    \label{p:vobang}
		Recalling \cref{e:def:v:projm},	let $\overline{\w}\in \cran(\Id-T)$.
			Then 
			\begin{equation}
      \label{eq:vobang}
				\overline{\w}=\v
				\; \siff\; (\forall \y\in \cran (\partial F+S))
				\scal{\overline{\w}}{M\overline{\w}-\y}\le 0.
        \end{equation}
        In particular
        \begin{enumerate}
\item
\label{eq:vobang:i}
$(\forall \r\in (\overline{\dom f^*+\lin^* (\dom g^*)})) 
$
we have
$\scal{\vx}{\tfrac{1}{\sigma}\vx-\lin^*\vy-\r}\le 0$.
 \item
\label{eq:vobang:ii}
$(\forall \d\in 
(\overline{\dom g -\lin(\dom f)}))$
          we have
$\scal{\vy}{\tfrac{1}{\tau}\vy-\lin\vx-\d}\le 0$.
\end{enumerate}
		\end{proposition}
		\begin{proof}
        See Appendix~\ref{App:2}.
% 			Indeed, let $\y\in \cran (\partial F+S)$.
%    Recalling 
% 			\cref{prop:T:prop}\cref{prop:T:prop:v},
% 			it follows from the Projection Theorem 
% 			see, e.g., \cite[Theorem~3.16]{BC2017},
% 			that 
% \begin{equation}
% 	     \overline{\w}=\v
% \siff\
% 			\scal{\overline{\w}}{\overline{\w}-M^{-1}\y}_M\le 0
%    \siff
% 	\scal{\overline{\w}}{M\overline{\w}-MM^{-1}\y}\le 0 
% \siff
% 	   \scal{\overline{\w}}{M\overline{\w}-\y}\le 0. 
% 	\end{equation}  
%       This verifies \cref{eq:vobang}.
% Now, let $\r\in \overline{(\dom f^*+\lin^* (\dom g^*))}$
%  and let $\d\in \overline{(\dom g -\lin(\dom f))}$.
%  Then $\y:=(\r,\d)\in \cran (\partial F+S)$
%  by \cref{thm:ranform}.
%  This and \cref{eq:vobang} imply that 
%  \begin{equation}
%  \label{e:vobangi}
% \scal{\vx}{\tfrac{1}{\sigma}\vx-\lin^*\vy-\r}+
% \scal{\vy}{\tfrac{1}{\tau}\vy-\lin\vx-\d}\le 0.
% \end{equation}
% Finally, observe that \cref{e:def:v:projm} and 
% \cref{prop:T:prop}\cref{prop:T:prop:v}
% imply
% \begin{equation}
% \label{e:vobangii}
% M\v=(\tfrac{1}{\sigma}\vx-\lin^*\vy,\tfrac{1}{\tau}\vy-\lin\vx)
% \in \cran(\partial F+S).
% \end{equation}
% \cref{eq:vobang:i}:
% Apply \cref{e:vobangi} with $\d$ replaced by $\tfrac{1}{\tau}\vy-\lin\vx$
% in view of \cref{e:vobangii}.% and \cref{thm:ranform}.
% \cref{eq:vobang:ii}:
% Apply \cref{e:vobangi} with $\r$ replaced by $\tfrac{1}{\sigma}\vx-\lin^*\vy$
% in view of \cref{e:vobangii}.% and \cref{thm:ranform}.
\end{proof}

\section{Specialization to conic problems}
  \label{sec:conicpdhg}
		From now on we assume that
		\begin{empheq}[box=\mybluebox]{equation}
			\text{$K$ is a closed convex cone in $Y$ and $\b\in Y$.}
		\end{empheq}	
		
		Consider the problem
		\begin{equation}
			\begin{array}{rl}
				\label{eq:c:problem}
				\minimize{\x\in X} & f(\x) \\
				\mbox{\rm subject to} &\lin\x-\b\in K.
			\end{array}
		\end{equation}
		Problem \cref{eq:c:problem} can be recast as 
		\begin{equation}
			\label{eq:qp:coneK}
			\minimize{\x\in X}\;\;  	f(\x)+\iota_{K}(\lin \x-\b).
		\end{equation}	
		We recover the PDHG framework by setting\footnote{Let $K\subseteq X$. Here and elsewhere 
  we use $K^\ominus $ to denote the \emph{polar cone} of 
  $K$ defined by $K^\ominus:=\menge{\u\in X}{\sup\scal{ K}{\u}\le 0}$.} 
		\begin{equation}
			\label{eq:ff^*:cone}
			\text{$g=\iota_{K}(\cdot-\b)$, hence $g^*=\iota_{K^\ominus}(\cdot)+\scal{\b}{\cdot}$}.
		\end{equation}
		It follows from \cite[Example 23.4~and~Proposition~23.17(ii)]{BC2017}
		that
		\begin{equation}
			\label{e:proxf*}
			\prox_{\tau g^*}=P_{K^\ominus}(\cdot-\tau \b).
		\end{equation}
		By combining \cref{e:pdhg:gen:update} and \cref{e:proxf*},
		the PDHG update to solve 
		\cref{eq:qp:coneK} is
		\begin{equation}
			\label{e:pdhg:CP:update:K}
			\begin{pmatrix}
				\x^+\\
				\y^+
			\end{pmatrix}=	T		
			\begin{pmatrix}
				\x\\
				\y
			\end{pmatrix}
			:=	\begin{pmatrix}
				\prox_{\sigma f}(\x - \sigma \lin^* \y)\\
				P_{K^\ominus}(\y+\tau \lin (2\x^+-\x)-\tau \b)
			\end{pmatrix}.
		\end{equation}
		
		\begin{remark}
			The following are special cases of \cref{eq:c:problem}.
			\begin{enumerate}
				\item
				Let $H\colon X\to X$ be linear, monotone, and self-adjoint, i.e., $H=H^*$
				and let $\c\in X$.
				The quadratic optimization problem
				\begin{equation}
					\label{e:qpform}
					\begin{array}{rl}
						\minimize{\x\in X} & \tfrac{1}{2}\scal{\x}{H\x}+\scal{\c}{\x} \\
						\mbox{\rm subject to} &\lin\x-\b\in K,
					\end{array}
				\end{equation}
				is a special case of \cref{eq:qp:coneK} 
				by setting
				\begin{equation}
					\label{e:def:gQP}
					f(\x)=\tfrac{1}{2}\scal{\x}{H\x}+\scal{\c}{\x}.
				\end{equation}
				It follows from \cite[Example~24.2]{BC2017}
				that
				\begin{equation}
					\label{e:proxs}
					\prox_{\sigma f}=J_{\sigma H}(\cdot-\sigma \c)\text{~where~} J_{\sigma H}:=(\Id+\sigma H)^{-1}.
				\end{equation}
The above formula uses standard notation for the \emph{resolvent}
  of an operator
  defined by $J_A:=(\Id+A)^{-1}$.			
				
				Let $(\x,\y)\in X\times Y$.
				By combining 
    \cref{e:pdhg:CP:update:K} and \cref{e:proxs},
				the PDHG update to solve 
				\cref{e:qpform} is
				\begin{equation}
					\label{e:impt:xup}
					\begin{pmatrix}
						\x^+\\
						\y^+
					\end{pmatrix}=	T
					\begin{pmatrix}
						\x\\
						\y
					\end{pmatrix}
					:=	\begin{pmatrix}
						J_{\sigma H}(\x-\sigma \lin^* \y-\sigma \c)\\
						P_{K^\ominus}(\y+\tau \lin(2\x^+-\x)-\tau \b)
					\end{pmatrix}.
				\end{equation}
				We point out that PDHG for standard QP is recovered by taking 
				$X=\RR^n$, $Y=\RR^m$ and $K=\R^m_-$.  This special case is the topic of \cref{sec:qpstatic}.
				\item 
				Let $C$ be a nonempty closed convex subset of $X$ and let $\c\in X$.
				The problem
				\begin{equation}
					\begin{array}{rl}
						\label{e:esform}
						\minimize{\x\in C} & \scal{\c}{\x} \\
						\mbox{\rm subject to} &\lin\x-\b\in K
					\end{array}
				\end{equation}
				is a special case of \cref{eq:qp:coneK} 
				by setting
				\begin{equation}
					\label{e:def:fcC}
					f(x)=\scal{\c}{\x}+\iota_C(\x).
				\end{equation}
				It follows from \cite[Example~24.2]{BC2017}
				that
				\begin{equation}
					\label{e:proxsC}
					\prox_{\sigma f}=P_C(\cdot-\sigma \c).
				\end{equation}
				Let $(\x,\y)\in X\times Y$.
				By combining \cref{e:pdhg:CP:update:K} and \cref{e:proxsC},
				the PDHG update to solve 
				\cref{e:esform} is
				\begin{equation}
					\label{e:impt:xup:es}
					\begin{pmatrix}
						\x^+\\
						\y^+
					\end{pmatrix}=	T
					\begin{pmatrix}
						\x\\
						\y
					\end{pmatrix}
					:=	\begin{pmatrix}
						P_C(\x-\sigma \lin^* \y-\sigma \c)\\
						P_{K^\ominus}(\y+\tau \lin(2\x^+-\x)-\tau \b)
					\end{pmatrix}.
				\end{equation}
    A further special case of \cref{e:esform} is obtained when $C$ is a cone and $K=\{\bzero\}$, yielding the problem commonly known as standard conic primal form.  We return to standard conic primal form in \cref{sec:standconic}.
			\end{enumerate}	
		\end{remark}

		\subsection{The infimal displacement vector in the conic case}
		\label{sec:infimal2}
  {In the remainder of this section, we develop additional properties of the infimal displacement vector for \cref{eq:c:problem}.}
		\begin{proposition}
			\label{lem:qp:v:K}
			Let $T$ be defined as in \cref{e:pdhg:CP:update:K}, let $\v$ be given by \cref{e:def:v:projm}, and let
			\begin{equation}
				\v=:(\vx,\vy) .
    % =\begin{pmatrix}{\vx}\\{\vy}\end{pmatrix}=\argmin_{w\in \cran(\Id-T)}\norm{w}\normM.
			\end{equation}
			Then the following hold: 
			\begin{enumerate}
				\item
				\label{lem:qp:K:v:iii}
				$-{\vy}=P_{K^\ominus}(-{\vy}-\tau \lin{\vx})$.
				\item
				\label{lem:qp:K:v:iv}
				$-\tau \lin{\vx}=P_{K}(-{\vy}-\tau \lin{\vx})$.
				\item
				\label{lem:qp:K:v:v}
				$ \scal{\lin{\vx}}{{\vy}}=0$.
				\item
				\label{lem:qp:K:v:vv}
				$\norm{\v}\normM^2=\norm{\v_R}^2/\sigma + \norm{\v_D}^2/\tau$.
				\item 
				\label{lem:C0:ii}
				$K=\{\bzero\} \RA \lin {\vx}=\bzero$, i.e., ${\vx}\in \ker \lin$.
			\end{enumerate}
			Suppose further that $f$ is defined as in 
			\cref{e:def:gQP}. Then \cref{e:pdhg:CP:update:K} reduces to \cref{e:impt:xup} and
			we additionally have:
			\begin{enumerate}[resume]
				\item
				\label{lem:qp:K:v:vi}
				$\lin^* {\vy}+H{\vx}=\bzero$.
				\item
				\label{lem:qp:K:v:vii}
				$\scal{{\vx}}{H{\vx}}=0$.
				\item 
				\label{lem:qp:K:v:ix}
				$J_{\sigma H}({\vx})={\vx}+\sigma J_{\sigma H}(\lin^*{\vy})$.
			\end{enumerate}
		\end{proposition}
		
		\begin{proof}
Let $(\x_0, \y_0)\in X\times Y $
  and let $((\x_{k},\y_{k}))_\knn$
  be the sequence obtained via the update 
  $(\forall \knn)$
  $(\x_{k+1},\y_{k+1})=T(\x_{k},\y_{k})$.
		Applying 	\cref{fact:T:v}\cref{fact:T:v:i}\&\cref{fact:T:v:ii}
		with $\z_0=(\x_0, \y_0)$
			% =\begin{psmallmatrix}
			% 	x_0
			% 	\\
			% 	y_0
			% \end{psmallmatrix}
			in view of \cref{prop:T:prop}\cref{prop:T:prop:ii}
			implies
			\begin{equation}
				\label{e:sep:lim:i}
				{\x_k}-\x_{k+1}\to {\vx}	\;\;\text{and}\;\;{\y_k}-\y_{k+1}\to {\vy},
			\end{equation}	
			and 
			% \cref{fact:T:v}\cref{fact:T:v:ii}
			% implies
			\begin{equation}
				\label{e:sep:lim:ii}
				\tfrac{\x_k}{k}\to -{\vx}	\;\;\text{and}\;\;\tfrac{\y_k}{k}\to -{\vy}.	
			\end{equation}

			\cref{lem:qp:K:v:iii}:	
			Because $K^\ominus$ is a cone
			we have
			$P_{K^\ominus}$ is positively homogeneous by \cite[Theorem~5.6(7)]{Deutsch}.
			Moreover,  $P_{K^\ominus}$ is (firmly) nonexpansive
   by, e.g., \cite[Proposition~4.16]{BC2017}, hence continuous.
			It follows from \cref{e:pdhg:CP:update:K}
			applied with $(\x,\y)$ replaced by $(\x_k,\y_k)$
			that 
			\begin{equation}
				\label{eq:PDHGyK:ii} 
				\y_{k+1} =P_{K^\ominus}(\y_k+\tau \lin(2\x_{k+1}-\x_k)-\tau \b).     
			\end{equation}
			Dividing both sides of \cref{eq:PDHGyK:ii} 
			by $k+1$ and taking the limit as $k\to \infty$
			in view of \cref{e:sep:lim:i}, \cref{e:sep:lim:ii} and \cref{eq:def:A}
			yield
			% \begin{subequations}
				\begin{align}	
					-{\vy}\leftarrow\tfrac{\y_{k+1}}{k+1}
					&=\tfrac{1}{k+1}P_{K^\ominus}	(\y_k+\tau \lin(2\x_{k+1}-\x_k)-\tau \b)
					\\
					&=\tfrac{1}{k+1}P_{K^\ominus}	
					\Big(\y_k-\y_{k+1}+\y_{k+1}+\tau \lin \x_{k+1}+ \tau \lin(\x_{k+1}-\x_k)- \tau \b \Big)
					\\
					&=P_{K^\ominus}	
					\Big(\tfrac{\y_k-\y_{k+1}}{k+1}
					+\tfrac{\y_{k+1}}{k+1}+\tau \lin\tfrac{\x_{k+1}}{k+1}
					+ \tau \lin \Big(\tfrac{\x_k-\x_{k+1}}{k+1} \Big)-\tau \tfrac{\b}{k+1}\Big)
					\\
					&\to P_{K^\ominus}	(-{\vy}-\tau \lin{\vx}).
				\end{align}	
			% \end{subequations}	
			This proves the claim.
\cref{lem:qp:K:v:iv}:	
			Indeed, using  
			\cref{lem:qp:K:v:iii}
			and the Moreau decomposition, see, e.g., 
			\cite[Theorem~6.30(i)]{BC2017}
			we have 
			% \begin{subequations}
				\begin{align}
					P_{K}(-{\vy}-\tau \lin{\vx})
					&=-{\vy}-\tau \lin{\vx}-P_{K^\ominus}(-{\vy}- \tau \lin{\vx})
					\\
					&=-{\vy}-\tau \lin{\vx}-(-{\vy})=-\tau \lin{\vx}.
				\end{align}	
			% \end{subequations}	
\cref{lem:qp:K:v:v}:	
			Combine  \cref{lem:qp:v:iii},  \cref{lem:qp:v:iv}
			and, e.g.,   \cite[Theorem~6.30(ii)]{BC2017}.
\cref{lem:qp:K:v:vv}:
			Indeed, recalling  \cref{lem:qp:K:v:v} we have 
			% \begin{subequations}
				\begin{align}
					\norm{\v}^2\normM
					&=\scal{\v}{M\v}=\scal{({\vx},{\vy})}{({\vx}/\sigma-\lin^*{\vy}, {\vy}/\tau-\lin{\vx})}  
					\\
					&=\scal{{\vx}}{{\vx}/\sigma-\lin^*{\vy}}+\scal{{\vy}}{{\vy}/\tau-\lin{\vx}}
					\\
					&=\norm{{\vx}}^2/\sigma-\scal{{\vx}}{\lin^*{\vy}}+\norm{{\vy}}^2/\tau-\scal{{\vy}}{\lin {\vx}}
					\\
					&=\norm{{\vx}}^2/\sigma+\norm{{\vy}}^2/\tau.
				\end{align}   
			% \end{subequations}
\cref{lem:C0:ii}:
			This is a direct consequence of \cref{lem:qp:K:v:iv}.
			\cref{lem:qp:K:v:vi}:	
			It follows from
			\cref{e:impt:xup} applied with $(\x,\y)$ replaced by 
			$(\x_k,\y_k)$
			that 
			$$
				\x_k-\x_{k+1}	- \sigma \lin^*(\y_k-\y_{k+1})= \sigma \lin^* \y_{k+1}+ \sigma H\x_{k+1} + \sigma \c.
			$$	
			Dividing the above equation by $k+1$
			and taking the limit as $k\to \infty$
			in view of \cref{e:sep:lim:i},  
			\cref{e:sep:lim:ii} and \cref{eq:def:A} yield 
			% \begin{subequations}
				\begin{align}
					\bzero\leftarrow &\quad\tfrac{1}{k+1}(	\x_k-\x_{k+1})	- \sigma \lin^*( \tfrac{1}{k+1}(\y_k-\y_{k+1}))
					\\
					&= \sigma \lin^*  \Big(\tfrac{\y_{k+1}}{k+1}\Big)+\sigma H\Big(\tfrac{\x_{k+1}}{k+1}\Big) + \sigma\left(\tfrac{\c}{k+1}\right)
					\to -\sigma \lin^* {\vy}-\sigma H{\vx}.
				\end{align}
			% \end{subequations}
			% This proves the claim.
\cref{lem:qp:K:v:vii}:	
			It follows from  \cref{lem:qp:v:v} and \cref{lem:qp:K:v:vi}
			that $0=\scal{{\vx}}{\lin^* {\vy}+H{\vx}}=\scal{{\vx}}{H{\vx}}
			+\scal{{\vx}}{\lin^* {\vy}}
			=\scal{{\vx}}{H{\vx}}
			+\scal{\lin{\vx}}{ {\vy}}=\scal{{\vx}}{H{\vx}}$.
\cref{lem:qp:K:v:ix}:
			It follows from  \cref{lem:qp:K:v:vi}
			that $(\Id+\sigma H){\vx}={\vx}-\sigma \lin^*{\vy}$. Hence,
			${\vx}=J_{\sigma H}({\vx}-\sigma \lin^*{\vy})=J_{\sigma H}{\vx}-\sigma J_{\sigma H} ( \lin^*{\vy})$
			as claimed.
		\end{proof}	

		\section{Application to QP}
		\label{sec:qpstatic}
		In this section we set
		\begin{empheq}[box=\mybluebox]{equation}
			X=\RR^n,\; Y=\RR^m, K=\RR^m_{-},
		\end{empheq}
		and 
		\begin{empheq}[box=\mybluebox]{equation}
  \label{eq:remind}
			\text{$T$ is defined as in \cref{e:impt:xup}.}
		\end{empheq}

This means that \cref{e:qpform} specializes to
\be
\begin{array}{rl}
\minimize{\x\in \RR^n} & \frac{1}{2}\langle \x,H\x\rangle + \langle \c,\x\rangle \\
\mbox{\rm subject to} & \lin\x\le \b,
\end{array}
\label{eq:qpprim}
\ee
whose Lagrangian dual is well known to be
\be
\begin{array}{rl}
\maximize{\q\in \RR^n,\y\in \RR^m} & -\frac{1}{2}\langle \q, H\q\rangle-\langle \b,\y\rangle \\
\mbox{\rm subject to} & H\q=-(\c+\lin^*\y), \\
& \y\ge\bz.
\end{array}
\label{eq:qpdu}
\ee
%\textcolor{brown}{
%The Lagrange form of  problem \cref{eq:qpprim} 
%\[
%\min_{\x}\max_{\y}
%f(\x)+\langle  \y,\lin \x\rangle -
%\underbrace{(\langle \b,\y\rangle + \iota_{K^*}(\y) )}_{\mbox{\scriptsize $g^*(\cdot)$ in \cref{eq:saddle}}}
%\]
%is seen to be a special case of \cref{eq:saddle}. Here, $K^*$
%denotes the dual cone of $K$. 
In view of \cref{e:proxs}, 
PDHG for QP is efficient in practice only if a fast method is available 
to solve the equations implicit in \cref{e:impt:xup}, 
e.g., if $\Id+\sigma H $ has a sparse Cholesky factorization. 
Unlike, e.g., interior-point methods, the coefficient matrix of the 
linear system to be solved does not vary with iteration counter $k$.

		\subsection{PDHG for QP: Static properties}
  We start with the following key lemma.
		\begin{lemma}
			\label{lem:ran:phd:v}
			The following hold:
			\begin{enumerate}
				\item
				\label{lem:ran:phd:v:i}
				$\partial F+S$ is a \emph{polyhedral multifunction}\footnote{Following
					\cite{Robin},
					we say that 
					a (possibly) set-valued mapping $A\colon X\rras X$ 
					is a polyhedral multifunction if $\gra A$
					is a union of finitely many polyhedral subsets of $X\times X$.}.
     		\item
				\label{lem:ran:phd:v:0:i}
    		$\ran(\partial F+S)$
				is a union of finitely many polyhedral\footnote{Let $C\subseteq X$.
					We say that $C$ is \emph{polyhedral}
					if $C$ is the intersection of finitely
					many halfspaces.} sets.
     \item 
     \label{lem:ran:phd:v:0:ii}
     				$\ran(\partial F+S)$ is closed and convex.
				% In fact, $\ran(\partial F+S)$ is polyhedral.
     	      \item
            \label{lem:ran:phd:v:i:i}
                $%\cran(\partial F+S)=
                \ran(\partial F+S)
                =(\ran H+\lin^*(\RR^m_+)+\c)\times (\RR^m_-+\ran \lin+\b)$.
                   \item
            \label{lem:ran:phd:v:i:ii}
                $\ran(\partial F+S)$ is polyhedral. 
				
				% \item 
				% \label{lem:ran:phd:v:ii}
				% $\ran(\Id-T)$
				% is a union of finitely many polyhedral sets.
    			\item 
				\label{lem:ran:phd:v:iv}
				$\ran(\Id-T)=M^{-1}((\ran H+\lin^*(\RR^m_+)+\c)\times (\RR^m_-+\ran \lin+\b))$.
				\item 
				\label{lem:ran:phd:v:iii}
				$\ran(\Id-T)$ is polyhedral.
				Hence, $\ran(\Id-T)$ is convex and closed.

			\end{enumerate}
		\end{lemma}
		\begin{proof}
			\cref{lem:ran:phd:v:i}:
			It follows from \cite[Example~on~page~207]{Robin}
			that $\partial F+S$ is a polyhedral multifunction.
			That is
			$\gra(\partial F+S)=P_1\cup\ldots\cup P_l$
			and $(\forall i\in \{1,\dots,l\})$ $P_i$
			is polyhedral.
			% Consequently, $\gra (\Id-T)
			% =\gra M^{-1}(\partial F+S)
			% =M^{-1}\gra (\partial F+S)=M^{-1}(P_1)\cup\ldots\cup M^{-1}(P_l)$.
   \cref{lem:ran:phd:v:0:i}:
   Observe that 
			$\ran (\partial F+S)$
			is the canonical projection of 
			$\gra (\partial F+S)$ onto $\RR^n\times \RR^m$.
			It follows from  \cite[Theorem~19.3]{Rock70}
			that $\ran (\partial F+S)$ is  a finite union of 
			polyhedral sets.
   \cref{lem:ran:phd:v:0:ii}:
   The claim of closedness is a consequence of \cref{lem:ran:phd:v:0:i}.
   The convexity of $\ran (\partial F+S)$ follows from combining 
   \cref{e:mm:sumSF},  \cite[Corollary~21.14]{BC2017} and the closedness of $\ran (\partial F+S)$.
\cref{lem:ran:phd:v:i:i}:
			It follows from
			\cref{e:def:gQP} and \cref{eq:ff^*:cone}
			that $\dom f=\RR^n$, $\dom f^*=\ran H+\c$,
			$\dom g=\RR^m_{-}+\b$ and  $\dom g^*=\RR^m_{+}$.
			Now combine this  with \cref{thm:ranform}
   in view of \cref{lem:ran:phd:v:0:ii}.
   \cref{lem:ran:phd:v:i:ii}:
   Combine \cref{lem:ran:phd:v:i:i} and 
   \cite[Theorem~19.3~and~Corollary~19.3.2]{Rock70}.
			% \cref{lem:ran:phd:v:ii}:
			% Observe that 
			% $\ran (\partial F+S)$
			% is the canonical projection of 
			% $\gra (\partial F+S)$ onto $\RR^n\times \RR^m$.
			% It follows from  \cite[Theorem~19.3]{Rock70}
			% that $\ran (\partial F+S)$ is  a finite union of 
			% polyhedral sets, and so is 
			% $\ran (\Id-T)=M^{-1}(\ran (\partial F+S))$.
   \cref{lem:ran:phd:v:iv}:
   Combine \cref{prop:T:prop}\cref{prop:T:prop:iv}
    and \cref{lem:ran:phd:v:i:i}.
\cref{lem:ran:phd:v:iii}:
			It follows from \cref{lem:ran:phd:v:iv}
    and  \cite[Theorem~19.3]{Rock70}, in view of \cref{lem:ran:phd:v:i:i},
			that $\ran(\Id-T)$ is polyhedral, hence it is closed and convex.
		\end{proof}
	\begin{remark}[{\bf when $K$ is a general polyhedral cone}]
 \label{rem:polyhedK}
One easily checks that the proof 
of \cref{lem:ran:phd:v} generalizes seamlessly 
if we replace $\RR^m_{-}$
by a general polyhedral cone $K$.
In this case
we have $  \ran(\partial F+S)
                =(\ran H+\lin^*(K^\ominus)+\c)\times (K+\ran \lin+\b)$
and 
 $  \ran(\Id -T)
                =M^{-1}((\ran H+\lin^*(K^\ominus)+\c)\times (K+\ran \lin+\b))$.
\end{remark}
		Finally, we set
		\begin{empheq}[box=\mybluebox]{equation}
			\label{e:def:v:Pm}
			\v=(\vx,\vy)
			=\argmin_{\w\in \cran(\Id-T)}\norm{\w}\normM\in \ran (\Id-T),
		\end{empheq}
		where the inclusion follows from \cref{lem:ran:phd:v}\cref{lem:ran:phd:v:iii}.
		
		\begin{lemma}
			\label{lem:qp:v}
			The following hold:
			\begin{enumerate}
				% \item
				% \label{lem:qp:v:i}
				% $A{\vx}\in \RR^m_{+}$.
				% \item
				% \label{lem:qp:v:ii}
				%  ${\vy}\in \RR^m_{-}$.
				\item
				\label{lem:qp:v:iii}
				$-{\vy}=P_{\RR^m_{+}}(-{\vy}-\tau \lin{\vx})$.
				\item
				\label{lem:qp:v:iv}
				$-\tau \lin{\vx}=P_{\RR^m_{-}}(-{\vy}-\tau \lin{\vx})$.
				\item
				\label{lem:qp:v:v}
				$ \scal{\lin {\vx}}{{\vy}}=0$.
				\item
				\label{lem:qp:v:vi}
				$\lin^\ast {\vy}+H{\vx}=\bzero$.
				\item
				\label{lem:qp:v:vii}
				$H{\vx}=\bzero$, i.e., ${\vx}\in \ker H$. Consequently, 				$\lin^\ast {\vy}={\bf 0}$.  
				% \item
				% \label{lem:qp:v:viii}
				% $\lin^\ast {\vy}={\bf 0}$.  
				\item 
				\label{lem:qp:v:ix}
				$J_{\sigma H}({\vx})=(\Id+\sigma H)^{-1}( {\vx})={\vx}$.
				\item
				\label{lem:qp:v:x}
				Let $i\in \{1, \ldots, m\}$.
				If $(\lin{\vx})_i>0$ then $({\vy})_i=0$.
				\item
				\label{lem:qp:v:xi}
				Let $i\in \{1, \ldots, m\}$.
				If $({\vy})_i<0$ then $(\lin {\vx})_i=0$.
				
				% \item
				% \label{lem:qp:v:iv}
				%  $\scal{{\vx}}{\lin\tran {\vy}}=0$.
			\end{enumerate}
		\end{lemma}	
		\begin{proof}
			We apply \cref{lem:qp:v:K}
			with $(X,Y, K)$ replaced by $(\RR^n,\RR^m, \RR^m_{-})$.
\cref{lem:qp:v:iii}--\cref{lem:qp:v:vi}:
			This follows from  \cref{lem:qp:v:K}\cref{lem:qp:K:v:iii}--\cref{lem:qp:K:v:vi}.
\cref{lem:qp:v:vii}:
	\textcolor{black}{It follows from \cref{lem:qp:v:K}\cref{lem:qp:K:v:vii}
			and the positive semidefiniteness of $H$
   in view of \cref{lem:qp:v:vi}
   that $H\vx=0$. 
   Now combine this with \cref{lem:qp:v:vi}.}
% \cref{lem:qp:v:viii}:	
% 			Combine \cref{lem:qp:v:vi} and \cref{lem:qp:v:vii}.
			\cref{lem:qp:v:ix}:	
			This is a direct consequence of
			\cref{lem:qp:v:K}\cref{lem:qp:K:v:ix}
			and \cref{lem:qp:v:vii}.
\cref{lem:qp:v:x}\&\cref{lem:qp:v:xi}:
			Combine \cref{lem:qp:v:iii}, \cref{lem:qp:v:iv}
			and \cref{lem:qp:v:v}.
		\end{proof}
  
We conclude this section with \cref{ex:vcalc} 
below \textcolor{black}{which serves as an illustrative example}.
In \cref{ex:preqpex} below 
\textcolor{black}{we define an empty feasible set and prove relevant conclusions}.
\begin{example}
\label{ex:preqpex}
Suppose that $m=4$, $n=2$,
$\c=(1,-2)\tran$, $\b=(-2,1,0,0)\tran$
 and 
$\lin=\begin{psmallmatrix}
    -1&1\\
    1&-1\\
    -1&0\\
    0&-1
\end{psmallmatrix}$.
Let $\u=-0.15\cdot (1,1)\tran$, let $\w=-0.15\cdot (1,1,0,0)\tran$,
and we set $\sigma=\tau=0.3$.
Let $\r\in \lin\tran (\RR^4_+)+\c$
 and let
 $\d\in \RR^4_{-}+\ran \lin+\b$.
Then the following hold:
\begin{enumerate}
\item
\label{ex:preqpex:0}
$\scal{\w}{\lin \u}=0$.
\item
\label{ex:preqpex:i}
$\norm{\lin}=\sqrt{5}$. 
Hence, $\sigma\tau\norm{A}^2=0.45<1$.
\item 
\label{ex:preqpex:ii}
$\tfrac{1}{\sigma}\u\in \lin\tran (\RR^4_+)+\c$.
\item 
\label{ex:preqpex:iii}
$\tfrac{1}{\tau}\w-\lin \u\in \RR^4_{-}+\ran \lin+\b$.
\item 
\label{ex:preqpex:iv}
$\tfrac{1}{\tau}\w\in \RR^4_{-}+\ran \lin+\b$.
\item 
\label{ex:preqpex:v}
$\scal{\u}{\tfrac{1}{\sigma}\u-\r}\le 0$.
\item 
\label{ex:preqpex:vi}
$\scal{\w}{\tfrac{1}{\tau}\w-\d}\le 0$.
\item 
\label{ex:preqpex:vii}
$\scal{\w}{\tfrac{1}{\tau}\w-\lin \u-\d}\le 0$.
\end{enumerate}
\end{example}
\begin{proof}
\cref{ex:preqpex:0}:
This is clear.
\cref{ex:preqpex:i}:
Indeed, a direct calculation yields that
$\lin\tran \lin
=
\begin{psmallmatrix}
    3&-2
    \\
    -2&3
\end{psmallmatrix}$.
Hence $\norm{\lin}^2=\norm{\lin\tran \lin}=\lambda_{\rm{max}} (\lin\tran \lin)=5$.
\cref{ex:preqpex:ii}:
Indeed, let $\x=(1.5,0,0,0)\tran\in \RR^4_+$
 and observe that 
 $\tfrac{1}{\sigma}\u=-0.5\cdot(1,1)\tran
 =\lin\tran \x+\c
 \in \lin\tran(\RR^4_+)+\c$.
\cref{ex:preqpex:iii}:
Indeed, let $\y=(-2.5,-1)\tran$
and let $\z=(0,0,-2.65,-1.15)\tran\in \RR^4_{-}$
 and observe that
 $\tfrac{1}{\tau}\w-\lin \u
 =-(0.5,0.5,0.15,0.15)\tran
 =\z+\lin \y+\b\in \RR^4_{-}+\ran \lin +\b$.
\cref{ex:preqpex:iv}:
Indeed, let $\y=(-2.5,-1)\tran$
and let $\z=(0,0,-2.5,-1)\tran\in \RR^4_{-}$
 and observe that
 $\tfrac{1}{\tau}\w
 =(-0.5,-0.5,0,0)\tran
 =\z+\lin \y+\b\in \RR^4_{-}+\ran \lin +\b$.
\cref{ex:preqpex:v}:
Indeed, let 
$\p=(\alpha,\beta,\gamma,\delta)\tran\in \RR^4_{+}$
be such that $\r=\lin\tran \p+\c$.
Then 
$\r=(-\alpha+\beta-\gamma+1, \alpha-\beta-\delta-2)\tran$.
Thus,
% \begin{subequations}
    \begin{align}
        \scal{\u}{\tfrac{1}{\sigma}\u-\r}
        &
        =-0.15\scal{(1,1)}{(-0.5,-0.5)-(-\alpha+\beta-\gamma+1, \alpha-\beta-\delta-2)}
        \\
        &=-0.15(-0.5+\alpha-\beta+\gamma-1-0.5-\alpha+\beta+\delta+2)
        \\
        &=-0.15(\gamma+\delta)\le 0.
    \end{align}
% \end{subequations}
\cref{ex:preqpex:vi}:
Indeed, let 
$\q=(\alpha,\beta,\gamma,\delta)\tran\in \RR^4_{-}$
 and let $\y=(\rho,\eta)\in \RR^2$
be such that $\d= \q+\lin \y+\b$.
Then $\d=(\alpha-\rho+\eta-2,\beta+\rho-\eta+1,\gamma-\rho,\delta-\eta)\tran$.
Therefore,
% \begin{subequations}
    \begin{align}
        &\quad \scal{\w}{\tfrac{1}{\tau}\w-\d}
        \\
        &
        =-0.15\scal{(1,1,0,0)}{(-0.5,-0.5,0,0)-
        (\alpha-\rho+\eta-2,\beta+\rho-\eta+1,\gamma-\rho,\delta-\eta)}
        \\
        &=-0.15(-0.5-\alpha+\rho-\eta+2
        -0.5-\beta-\rho+\eta-1)
        =0.15(\alpha+\beta)\le 0.
    \end{align}
% \end{subequations}
\cref{ex:preqpex:vii}:
Combine \cref{ex:preqpex:vi} and \cref{ex:preqpex:0}.
\end{proof}

\textcolor{black}{Using the feasible set defined in \cref{ex:preqpex} we now construct 
examples of an inconsistent quadratic program and an inconsistent linear program.}
\begin{example}
    \label{ex:vcalc}
    Recalling \cref{eq:qpprim},
    let $H$ be a $2\times 2$  matrix and 
    let $\lin$, $\b$, $\c$ be given as in
    \cref{ex:preqpex}.
    In view of \cref{ex:preqpex}\cref{ex:preqpex:i},
     we set $\sigma=\tau=0.3$.
    Recalling \cref{e:def:v:Pm}, the following hold.
    \begin{enumerate}
        \item 
        \label{ex:vcalc:i}
        Suppose that 
        $H=0$. Then 
        $\vx=-0.15\cdot(1,1)\tran$ and  $\vy=-0.15\cdot(1,1,0,0)\tran$.
        \item 
        \label{ex:vcalc:ii}
        Suppose that 
        $H=\Id$. Then 
        $\vx=(0,0)\tran$ and  $\vy=-0.15(1,1,0,0)\tran$.
    \end{enumerate}
\end{example}
\begin{proof}
We set  $\u=-0.15\cdot {(1,1)}\tran$ 
and  set $\w=-0.15\cdot(1,1,0,0)\tran$.
\cref{ex:vcalc:i}:
Observe that $\ran H=\{\bzero\}$, hence
\cref{lem:ran:phd:v}\cref{lem:ran:phd:v:i:i}
yields
$\ran(\partial F+S)=(\lin\tran(\RR^4_+)+\c)\times (\RR^4_-+\ran \lin+\b)$.
On the one hand, it follows from 
\cref{ex:preqpex}\cref{ex:preqpex:ii}\&\cref{ex:preqpex:iii}
 and 
 \cref{lem:ran:phd:v}\cref{lem:ran:phd:v:i:i}
that 
$M(\u,\w)=(\tfrac{1}{\sigma}\u, \tfrac{1}{\tau}\w-\lin \u) \in \ran(\partial F+S)=(\lin\tran(\RR^4_+)+\c)\times (\RR^4_-+\ran \lin+\b)$. 
Equivalently,
$(\u,\w)\in M^{-1}(\ran(\partial F+S))=\ran(\Id-T)$
by  \cref{lem:ran:phd:v}\cref{lem:ran:phd:v:iv}.
On the other hand,
\cref{ex:preqpex}\cref{ex:preqpex:v}\&\cref{ex:preqpex:vii}
implies that 
$(\forall (\r,\d)\in \ran (\partial F+S))$
we have $\scal{(\u,\w)}{M(\u,\w)-(\r,\d)}\le 0$.
Altogether in view of 
\cref{eq:vobang} this
yields $(\vx,\vy)=(\u,\w)$.

\cref{ex:vcalc:ii}:
Observe that $\ran H=\RR^2$, hence
\cref{lem:ran:phd:v}\cref{lem:ran:phd:v:i:i}
yields
$\ran(\partial F+S)=\RR^2\times (\RR^4_-+\ran \lin+\b)$.
On the one hand, it follows from 
\cref{ex:preqpex}\cref{ex:preqpex:ii}\&\cref{ex:preqpex:iii}
 and 
 \cref{lem:ran:phd:v}\cref{lem:ran:phd:v:i:i}
that 
$M(\bzero,\w)=(\bzero, \tfrac{1}{\tau}\w) \in \ran(\partial F+S)=\RR^2\times (\RR^4_-+\ran \lin+\b)$. 
Equivalently,
$(\bzero,\w)\in M^{-1}(\ran(\partial F+S))=\ran(\Id-T)$
by  \cref{lem:ran:phd:v}\cref{lem:ran:phd:v:iv}.
On the other hand,
\cref{ex:preqpex}\cref{ex:preqpex:vi}
implies that 
$(\forall (\r,\d)\in \ran (\partial F+S))$
we have $\scal{(\bzero,\w)}{M(\bzero,\w)-(\r,\d)}
=\scal{\w}{\tfrac{1}{\tau}\w-\d}
\le 0$.
Altogether in view of \cref{eq:vobang} this yields $(\vx,\vy)=(\bzero,\w)$.
\end{proof}
            
            \subsection{Detecting inconsistency of QP} 

   We start with the following useful lemma.          
            \begin{lemma}
            \label{lem:incondet}
            Let $\r\in\ran H+\lin^*(\RR^m_+)+\c$
            and let $\d\in\RR^m_-+\ran \lin+\b$. Then the following hold:
\begin{enumerate}
\item
\label{lem:incondet:i}
$\scal{\vx}{\tfrac{1}{\sigma}\vx-\r}\le 0$.
 \item
\label{lem:incondet:ii}
$\scal{\vy}{\tfrac{1}{\tau}\vy-\d}\le 0$.
\end{enumerate}
In particular, we have 
\begin{enumerate}[resume]
\item
\label{lem:incondet:iii}
$\scal{\vx}{\tfrac{1}{\sigma}\vx-\c}\le 0$.
 \item
\label{lem:incondet:iv}
$\scal{\vy}{\tfrac{1}{\tau}\vy-\b}\le 0$.
\end{enumerate}
\end{lemma}
\begin{proof}
\cref{lem:incondet:i}:
Combine
\cref{p:vobang}\cref{eq:vobang:i},
\cref{lem:ran:phd:v}\cref{lem:ran:phd:v:i:i}
and 
\cref{lem:qp:v}\cref{lem:qp:v:v}
to learn that 
$0
\ge \scal{\vx}{\tfrac{1}{\sigma}\vx-\lin^*\vy-\r}
=\scal{\vx}{\tfrac{1}{\sigma}\vx-\r}-\scal{\vx}{\lin^*\vy}
=\scal{\vx}{\tfrac{1}{\sigma}\vx-\r}-\scal{\lin\vx}{\vy}
=\scal{\vx}{\tfrac{1}{\sigma}\vx-\r}$.
\cref{lem:incondet:ii}:
Proceed similarly to \cref{lem:incondet:i}
but use 
\cref{p:vobang}\cref{eq:vobang:ii}
instead of \cref{p:vobang}\cref{eq:vobang:i}. 
\cref{lem:incondet:iii}:
Apply \cref{lem:incondet:i} with $\r$ replaced by $\c$.
\cref{lem:incondet:iv}:
Apply \cref{lem:incondet:ii} with $\d$ replaced by $\b$.
             \end{proof}

                 \begin{theorem}
            \label{lem:qp:infeasibility}
            The following hold:
            \begin{enumerate}
                \item \label{lem:qp:infeasibility_certificate:vx} If $\vx \neq \bz$, then the dual \cref{eq:qpdu} is infeasible, and $\vx$ is an infeasibility certificate. 
                \item \label{lem:qp:infeasibility_certificate:vy} If $\vy \neq \bz$, then the primal \cref{eq:qpprim} is infeasible, and $\vy$ is an infeasibility certificate.
            \end{enumerate}                
            \end{theorem}
          \begin{proof}
            \cref{lem:qp:infeasibility_certificate:vx}: 
            Indeed, on the one hand it follows from \cref{lem:incondet}\cref{lem:incondet:iii} that $\scal{\c}{\vx}\ge \tfrac{1}{\sigma}\norm{\vx}^2>0$.
            On the other hand \cref{lem:qp:v}\cref{lem:qp:v:iv} implies that $\lin \vx\ge\bzero$.  In addition, $H\vx=\bzero$.  Thus, taking the inner product of the dual constraint $H\q=-(\c+\lin^*\y)$ with $\vx$ yields $0=\scal{\c}{\vx}+\langle \lin\vx,\y\rangle$, which contradicts the other constraint $\y\ge\bzero$.
            Altogether, $\vx$ is an infeasibility certificate for the dual \cref{eq:qpdu}.
            \cref{lem:qp:infeasibility_certificate:vy}: Indeed, on the one hand it follows from \cref{lem:incondet}\cref{lem:incondet:iv} that $\scal{\b}{\vy}\ge \tfrac{1}{\tau}\norm{\vy}^2>0$.
            On the other hand \cref{lem:qp:v}\cref{lem:qp:v:vii} implies that $\lin^* \vy=\bzero$.  Finally, $\vy\le\bzero$ by \cref{lem:qp:v}\cref{lem:qp:v:iii}.  Taking the inner product of $\vy$ with both sides of the constraint $\lin\x\le\b$ yields $\langle\v_D,\lin x\rangle \ge \langle \v_D,\b\rangle$, i.e., $0\ge \langle \v_D,\b\rangle$, a contradiction.
            Altogether, $\vy$ is an infeasibility certificate for  \cref{eq:qpprim}.
            \end{proof}

\subsection{PDHG for QP: Dynamic behaviour}
		\label{sec:qpdynamic}
   In this section, we show that in the case of QP, $((\x_k,\y_k)+k\v)_\knn$ converges as $k\rightarrow\infty$.  Our result extends the analogous result for LP due to Applegate et al. Our proof technique is somewhat different from that of \cite{Applegate} in that it builds on our previous characterization of $\ran(\Id-T)$.  This result strengthens both parts of \cref{fact:T:v} for the special case of PDHG applied to QP.

		We start with the following useful lemma.
		
		\begin{lemma}
			\label{lem:intcone}
			Let $C$ be a nonempty closed convex cone of $X$
			such that $\inte C\neq \fady$.
			Let $\w\in \inte C$ and let $M>0$. Then there exists $\overline{\alpha}\ge 0$
			such that 
			$(\forall \alpha\ge \overline{\alpha})$
			$\ball{\bzero}{M}+\alpha \w\subseteq  \inte C$.
		\end{lemma}
		\begin{proof}
			Indeed, observe that $\inte C$
			is a nonempty convex cone of $X$.
			If $\ball{0}{M} \subseteq \inte C$ then $\overline{\alpha}=0$
			and the conclusion follows.
			Otherwise,
			by assumption $(\exists \epsilon >0)$
			such that $\w+\ball{0}{\epsilon}\subseteq \inte C$.
			Let $\overline{\alpha}\ge  \tfrac{M}{\epsilon}+1$ and observe that
			$(\forall \alpha \ge\overline{\alpha} )$
			$\w+\ball{\bzero}{\epsilon-\tfrac{M}{{\alpha}}}\subseteq \inte C$.
			Now,
   		$\ball{\bzero}{M}+\alpha \w 
					\subseteq \ball{\bzero}{M}+\alpha \w +\ball{\bzero}{\alpha\epsilon-M}
					=\alpha \w +\ball{\bzero}{\alpha\epsilon}=\alpha (\w+\ball{\bzero}{\epsilon})
					 \subseteq \alpha (\inte C)=\inte C.$
			% \begin{subequations}
			% 	\begin{align}
			% 		\ball{\bzero}{M}+\alpha \w 
			% 		&\subseteq \ball{\bzero}{M}+\alpha \w +\ball{\bzero}{\alpha\epsilon-M}
			% 		\\
			% 		&=\alpha \w +\ball{\bzero}{\alpha\epsilon}=\alpha (\w+\ball{\bzero}{\epsilon})
			% 		\\
			% 		&\subseteq \alpha (\inte C)=\inte C.
			% 	\end{align}
			% \end{subequations}
			The proof is complete.
		\end{proof}
		
		\begin{corollary}
			\label{cor:intcone}
			Let $C$ be a nonempty closed convex cone of $X$
			such that $\inte C\neq \fady$.
			Let $\w\in \inte C$ and suppose that
			$(\x_k)_\knn$ is a bounded sequence 
			in $X$. 
			Then there exists $\overline{\alpha}\ge0$
			such that 
			$(\forall \alpha\ge \overline{\alpha})$
			the following hold:
			\begin{enumerate}
				\item
				\label{cor:intcone:i}
				$(\x_k+\alpha \w)_\knn$
				lies in $\inte C$. 
				\item 
				\label{cor:intcone:ii}
				$P_{C}(\x_k+\alpha \w)=\x_k+\alpha \w$.
				\item
				\label{cor:intcone:iii}
				$P_{C^\ominus}(\x_k+\alpha \w)=\bzero$. 
			\end{enumerate}
		\end{corollary}	
		\begin{proof}
			\cref{cor:intcone:i}:
			Because $(\x_k)_\knn$
			is bounded there exists 
			$M>0$ such that
			$(\x_k)_\knn$
			lies in $\ball{0}{M}$.
			Now combine this with \cref{lem:intcone}.
			\cref{cor:intcone:ii}\&\cref{cor:intcone:iii}:
			This is a direct consequence of 
			\cref{cor:intcone:i}.
		\end{proof}
			
		Because $\v\in \ran (\Id-T)$ (see \cref{e:def:v:Pm}) we learn that
		\begin{equation}
  \label{e:fixnotempty}
			\fix(\v+T)\neq \fady,	
		\end{equation}	
		where $(\forall \z\in \RR^n\times \RR^m)$ $(\v+T)(\z)=\v+T\z$.
  
We are now ready for the main result in this section.
We point out that \cref{t:QPconvergence} below generalizes \cite[Theorem~5]{Applegate} 
to quadratic programming. 
\begin{theorem}
  \label{t:QPconvergence}
			Let $\z_0=
   % \begin{psmallmatrix}\x_0\\\y_0\end{psmallmatrix}\in 
   (\x_0,\y_0)\in \RR^n\times \RR^m$.
			Update via $(\forall \knn)$
			\begin{equation}
				\z_{k+1}=T\z_k.
			\end{equation}		
			Then $(\exists \alpha\ge 0) $ such that the sequence $(\z_k+k\v)_\knn$
			converges to a point in $\alpha \v+\fix (\v+T)$.
			% to a point in $\fix (v+T)$.
		\end{theorem}	
		
		\begin{proof}
			Observe that the sequence
			\begin{equation}
				\label{e:bdd:decop}
				(\z_k+k\v)_\knn
				=\begin{pmatrix}
					(\x_k+k{\vx})_\knn\\
					(	\y_k+k{\vy})_\knn
				\end{pmatrix}
				\text{~~is bounded~~}
			\end{equation}	
			by \cref{fact:fejer}.
			Furthermore, the sequence
			\begin{equation}
				\label{e:bdd:decoup:pazy}
				(\z_k-\z_{k+1})_\knn
				=\begin{pmatrix}
					(\x_k-\x_{k+1})_\knn\\
					(\y_k-\y_{k+1})_\knn
				\end{pmatrix}
				\text{~~is convergent, hence bounded~~}
			\end{equation}	
			by  \cref{fact:T:v}\cref{fact:T:v:ii}.

			We set $I=\menge{\{i\in 1, \ldots, m\}}{(\lin{\vx})_i>0 \text{~or ~} (\vy)_i<0}$.
			We proceed by verifying the following claims.
			
			\textsc{Claim~1:}	
			There exists $K\in \NN$ such that 
			$(\forall k\ge K)$ $(\forall i\in I)$
			\begin{equation}
				\label{e:proj:simp}
				((\y_k+\tau \lin(2\x_{k+1}-\x_k)-\tau \b)_{i})_+=
				\begin{cases}
					0,&(\lin{\vx})_i>0;
					\\
					(\y_k+\tau \lin (2\x_{k+1}-\x_k)-\tau \b)_{i},&({\vy})_i<0,
				\end{cases}	
			\end{equation}
			and $(\forall k\in \NN)$  $(\forall i\in I)$
			\begin{equation}
				\label{e:proj:simp:2}
				\begin{split}
					&\qquad((\y_k+k{\vy}+\tau \lin (2\x_{k+1}-\x_k+(k+2){\vx})-\tau \b-K\lin {\vx}-K{\vy})_{i})_+
					\\
					&=\begin{cases}
						0,&(\lin {\vx})_i>0;
						\\
						(\y_k+ k{\vy} + \tau \lin (2\x_{k+1}-\x_k)-\tau \b-K{\vy})_{i},&({\vy})_i<0.
					\end{cases}	
				\end{split}
				%	u_i=y_k+A(2(x_{k+1}+(k+1){\vx}+\alpha {\vx})-(x_k+k{\vx}+\alpha {\vx})-b)_i.
			\end{equation}
			We proceed by verifying the following two sub-claims.
			
			\textsc{Claim~1-a:}	
			There exists $ \overline{K}_1\ge 0$ such that 
			$(\forall k\ge \overline{K}_1)$
			 $(\forall i\in I)$ the identity in
			\cref{e:proj:simp} holds.
			
			Indeed, suppose $ i\in I$.
			First, suppose that $(\lin {\vx})_i>0$.
			In this case, \cref{lem:qp:v}\cref{lem:qp:v:x}
			implies that 
			$({\vy})_i=0$. 
			Therefore
			\begin{align}
				&\quad	(\y_k+\tau \lin (2\x_{k+1}-\x_k)-\tau \b)_{i}
				\nonumber
				\\
				&=(\y_k+\tau \lin (\x_{k+1}-\x_k+(\x_{k+1}+(k+1){\vx}))-\tau \b -\tau (k+1)\lin {\vx})_i.
			\end{align}
			
			It follows from \cref{e:bdd:decop}, the continuity of $\lin $, \cref{e:bdd:decoup:pazy}
			and  \cref{e:bdd:decop} again that 
			the sequences $((\y_k)_i)_\knn=((\y_k+k{\vy})_i)_\knn$ and
			$((\tau \lin (\x_{k+1}-\x_k+(\x_{k+1}+(k+1){\vx}))-\tau \b)_i)_\knn$
			are bounded. Hence, their sum is bounded.
			Applying \cref{cor:intcone}\cref{cor:intcone:iii} with 
			$C$ replaced by $\left]-\infty,0\right]$,
			$(\x_k)_\knn$
			replaced by   $((\y_k + k {\vy} + \tau \lin (\x_{k+1}-\x_k+(\x_{k+1}+(k+1){\vx}))-\tau \b)_i)_\knn$
			and $\w$ replaced by $(-\tau \lin {\vx})_i<0$ we learn that 
			there exists $K_i$ such that 
			\begin{equation}
				\label{e:alpha:pos}
				(\forall k\ge K_i)
				\quad
				((\y_k+k {\vy} + \tau \lin (2\x_{k+1}-\x_k)-\tau \b)_{i})_+ = ((\y_k+\tau \lin (2\x_{k+1}-\x_k)-\tau \b)_{i})_+=0.
			\end{equation}	
			Now, suppose that $({\vy})_i<0$.
			In this case, \cref{lem:qp:v}\cref{lem:qp:v:xi}
			implies that 
			$(\lin {\vx})_i=0$. Therefore,
			\begin{equation}
				(\y_k+\tau \lin (2\x_{k+1}-\x_k)-\tau \b)_{i}
				=(\y_k+k{\vy}+\tau \lin (2\x_{k+1}-\x_k)-\tau \b-k{\vy})_i.
			\end{equation}
			It follows from \cref{e:bdd:decop}, the continuity of $\lin $, \cref{e:bdd:decoup:pazy}
			and  \cref{e:bdd:decop} again that 
			the sequences $((\y_k+k{\vy})_i)_\knn$ and
			$((\tau \lin (\x_{k+1}-\x_k+x_{k+1})-\tau \b)_i)_\knn
             =((\tau \lin (\x_{k+1}-\x_k+(x_{k+1}+(k+1){\vx}))-\tau \b)_i)_\knn$
			are bounded. Hence, their sum is bounded.
			Applying
			\cref{cor:intcone}\cref{cor:intcone:ii}
   % \&\cref{cor:intcone:iii} 
   with 
			$C$ replaced by $\left[0,+\infty \right[$,
			$(\x_k)_\knn$
			replaced by   $((\y_k+k{\vy}+\tau \lin (\x_{k+1}-\x_k+\x_{k+1})-\tau \b)_i)_\knn$
			and $\w$ replaced by $(-{\vy})_i>0$ 
			we learn that 
			there exists $K_i$ such that 
			\begin{equation}
				\label{e:alpha:neg}
				(\forall k\ge K_i)
				\quad
				((\y_k+\tau \lin (2\x_{k+1}-\x_k)-\tau \b)_{i})_+
    =(\y_k+\tau \lin (2\x_{k+1}-\x_k)-\tau \b)_{i}.
			\end{equation}	
			We set 
			\begin{equation}
				\overline{K}_1=\max_{i\in I}\{K_i\}.	
			\end{equation}	
			This verifies \textsc{Claim~1-a}.
			
			\textsc{Claim~1-b:}	
			There exists $ \overline{K}_2\ge 0$ such that 
			$(\forall \hat{K}\ge \overline{K}_2)$
			$(\forall k\in \NN)$ $(\forall i\in I)$
			the identity in \cref{e:proj:simp:2}, with $K$ replaced by $\hat{K}$, holds.
			
			Indeed, observe that $(\forall \knn)$
			\begin{equation}
				\label{e:bddd}
				\begin{split}
					\y_k+k{\vy}+\tau \lin (2\x_{k+1}-\x_k+(k+2){\vx})-\tau \b
					\\
					=\y_k+k{\vy}+\tau \lin (2(\x_{k+1}+(k+1){\vx})-(\x_k+k{\vx}))-\tau \b.
				\end{split} 
			\end{equation}
			Hence, $(\forall i\in I)$
			$( \y_k+k{\vy}+\tau \lin (2\x_{k+1}-\x_k+(k+2){\vx})-\tau \b)_i$
			is bounded by  \cref{e:bdd:decop}.
			As before,  if $(\lin{\vx})_i>0$
			then  \cref{lem:qp:v}\cref{lem:qp:v:x}
			implies that 
			$({\vy})_i=0$. 
			Applying \cref{cor:intcone}\cref{cor:intcone:ii}\&\cref{cor:intcone:iii} with 
			$C$ replaced by $\left]-\infty,0\right]$,
			$(\x_k)_\knn$
			replaced by  $(( \y_k+k{\vy}+\tau \lin (2\x_{k+1}-\x_k+(k+2){\vx})-\tau \b)_i)_\knn$
			and $\w$ replaced by $(-\lin{\vx})_i<0$ 
			implies that there exists $\hat{K}_i\ge 0$ such that $(\forall {\hat K}\ge \hat{K}_i)$
			$(\forall \knn)$ 
			\begin{align}
				&\qquad ((\y_k+k{\vy}+\tau \lin (2\x_{k+1}-\x_k+(k+2){\vx})-\tau \b-\hat{K}\lin {\vx}-\hat{K}{\vy})_{i})_+
				\nonumber
				\\
				&=((\y_k+k{\vy}+\tau \lin (2\x_{k+1}-\x_k+(k+2){\vx})-\tau \b-\hat{K}\lin {\vx})_{i})_+=0. 
			\end{align}
   
			If $({\vy})_i<0$
			then  \cref{lem:qp:v}\cref{lem:qp:v:xi}
			implies that 
			$(\lin{\vx})_i=0$. 
			Applying \cref{cor:intcone}\cref{cor:intcone:ii}\&\cref{cor:intcone:iii} with 
			$C$ replaced by $\textcolor{black}{\left[0,\infty\right[}$, $(\x_k)_\knn$
			replaced by  $(( \y_k+k{\vy}+\tau \lin (2\x_{k+1}-\x_k+(k+2){\vx})-\tau \b)_i)_\knn$
			and $\w$ replaced by $(-{\vy})_i>0$ 
			implies that there exists $\hat{K}_i\ge 0$
			such that $(\forall \knn)$ $(\forall \hat{K}\ge \hat{K}_i\ge 0)$
			\begin{align}
				&\qquad ((\y_k+k{\vy}+\tau \lin (2\x_{k+1}-\x_k+(k+2){\vx})-\tau \b
    -\hat{K}\lin {\vx}-\hat{K}{\vy})_{i})_+
				\nonumber
				\\
				% &=(\y_k+k{\vy}+\tau \lin (2\x_{k+1}-\x_k+(k+2){\vx})-\tau \b
    % -K{\vy})_{i}
				% \nonumber
				% \\
				&=(\y_k+k{\vy}+\tau \lin (2\x_{k+1}-\x_k)-\tau \b-\hat{K}{\vy})_{i}. 
			\end{align}
   
			We set 
			\begin{equation}
				\overline{K}_2=\max_{i\in I}\{\hat{K}_i\}.	
			\end{equation}	
			This verifies \textsc{Claim~1-b}.
			Finally, we set $K=\max\{ \overline{K}_1, \overline{K}_2\}$.
        			Then  \cref{e:proj:simp} holds, by  \textsc{Claim~1-a}, and 
                    \cref{e:proj:simp2} holds, by  \textsc{Claim~1-b} applied with $\hat{K}$ replaced by $ K$.
			This verifies \textsc{Claim~1}.	
			
			\textsc{Claim~2:}	
			We have $(\forall \knn)$
			\begin{equation}
				\label{e:seq:equiv}
				% z_{k+K}+kv-K v=(v+T)^k(z_K-K v).
				\z_{k+K}+k\v=(\v+T)^k(\z_K).
			\end{equation}	
			To simplify the notation, we set 
			$(\forall \knn)$
			\begin{equation}
				\label{e:seq:equiv:w}
				\w_k:=(\v+T)^k(\z_K)
				:=\begin{pmatrix}
					\p_k\\
					\q_k
				\end{pmatrix}.
			\end{equation}
			Therefore, \cref{e:seq:equiv} reduces to 
			proving 
			\begin{equation}
				\begin{pmatrix}
					(\p_k)_\knn\\
					(\q_k)_\knn\
				\end{pmatrix}
				=\begin{pmatrix}
					(\x_{k+K}+k{\vx})_\knn\\
					(\y_{k+K}+k{\vy})_\knn
				\end{pmatrix}.
			\end{equation}
			We use induction on $k$.
			The base case at $k=0$ is clear.
			Now suppose that for some $k\ge 0$
			\cref{e:seq:equiv} holds.
			We first verify that 
			\begin{equation}
				\label{e:seq:px}
				(\p_k)_\knn=(\x_{k+K}+k{\vx})_\knn=(\x_{k+K}+(k+K){\vx}-K {\vx})_\knn.
			\end{equation}
			Indeed, 
			the inductive hypothesis,
			the linearity of $J_{\sigma H}=(\Id+\sigma H)^{-1}$,
			and \cref{lem:qp:v}\cref{lem:qp:v:vii}\&\cref{lem:qp:v:ix}
			yield
			% \begin{subequations}
				\begin{align}
					\p_{k+1}
					&={\vx}+J_{\sigma H}(\p_k-\sigma \lin^* \q_k-\sigma \c)	
					\\
					&={\vx}+J_{\sigma H}(\x_{k+K}+k{\vx}-\sigma \lin^* (\y_{k+K}+k{\vy})-\sigma \c)
					\\
					&={\vx}+J_{\sigma H}(\x_{k+K}-\sigma \lin^* \y_{k+K}+
					\sigma k\lin^* {\vy}-\sigma \c)	+ kJ_{\sigma H}({\vx})	
					\\
					&
					={\vx}+J_{\sigma H}(\x_{k+K}-\sigma \lin^* \y_{k+K}-\sigma \c)	+ k{\vx}
					\\
					&
					=J_{\sigma H}(\x_{k+K}-\sigma \lin^*\y_{k+K}-\sigma \c)	+(k+1){\vx}
					\\
					&=\x_{k+K+1}+(k+1){\vx}.
				\end{align}	
			% \end{subequations}
			We now verify that 
			\begin{equation}
				(\q_k)_\knn=(\y_{k+K}+k{\vy})_\knn=(\y_{k+K}+(k+K){\vy}-K {\vy})_\knn.
			\end{equation}
			It follows from the
			the inductive hypothesis,
			\cref{e:proj:simp:2}
			and 
			\cref{e:seq:px} that
			% \begin{subequations}
				\begin{align}
					\q_{k+1}
					&={\vy}+(\q_k+\tau \lin (2\p_{k+1}-\p_k)-\tau \b)_{+}	
					\\
					&={\vy}+(\y_{k+K}+(k+K){\vy}+\tau \lin (2(\x_{k+K+1}+(k+K+1){\vx})
					\nonumber
					\\
					&\quad -(\x_{k+K}+(k+K){\vx}))-\tau \b-\tau K \lin {\vx}-K{\vy})_{+}.
					\label{e:se:q}
				\end{align}	
			% \end{subequations}
			Let $i\in \{1,\ldots,m\}$.
			We examine the following cases.
			
			\textsc{Case~1:}  $(\lin {\vx})_i>0$. In this case $({\vy})_{i}=0$.
			On the one hand, 
			\cref{e:proj:simp} implies
			that
			\begin{equation}
				(\y_{k+K+1}+(k+1){\vy})_i=0+0=0.
			\end{equation}
			
			On the other hand, 
			in view of \cref{e:proj:simp:2} 
			and \cref{e:se:q} we have
			\begin{equation}
				(\q_{k+1})_i
				=({\vy})_i+((\q_k+\tau \lin (2\p_{k+1}-\p_k)-\tau \b)_{i})_{+}	
				=0+0=0=(\y_{k+K+1}+(k+1){\vy})_i.	
			\end{equation}	
			
			\textsc{Case~2:}  $({\vy})_i<0$. In this case $(\lin{\vx})_{i}=0$.
			On the one hand, 
			\cref{e:proj:simp} implies
			that
			\begin{equation}
				(\y_{k+K+1})_i=
				(\y_{k+K}+\tau \lin (2\x_{k+K+1}-\x_{k+K})-\tau \b)_{i}.
			\end{equation}

			On the other hand, in view of \textsc{Claim~1-b}
			and \cref{e:se:q} we have
			% \begin{subequations}
				\begin{align}
					(\q_{k+1})_i
					&=({\vy})_i+((\q_k+\tau \lin (2\p_{k+1}-\p_k)-\tau \b)_{i})_{+}	
					\\
					&=({\vy})_i+(\y_{k+K}+(k+K){\vy}+\tau \lin (2\x_{k+K+1}
					-\x_{k+K})-\tau \b-K{\vy})_i
					\\
					&=(\y_{k+K}+(k+1){\vy}+\tau \lin (2\x_{k+K+1}
					-\x_{k+K})-\tau \b)_i
					\\
					&=(\y_{k+K+1})_{i}+(k+1)({\vy})_{i}.
				\end{align}    
			% \end{subequations}
			
			\textsc{Case~3:}  $(\lin{\vx})_i=({\vy})_{i}=0$.	
			In this case, it is straightforward to verify the inductive step
			and the conclusion is obvious.
			
			\textsc{Claim~3:}
			There exists $\alpha\ge 0$ such that the sequence $(\z_k+k\v)_\knn$
			converges to a $\overline{\z}\in \alpha \v+\fix (\v+T)$.
			Indeed, \cref{e:seq:equiv:w} means that $\w_k$ can be obtained by iterating $(\v+T)$ on $\w_0\equiv\z_K$.  Since $T$ is firmly nonexpansive
   (\cref{prop:T:prop}\cref{prop:T:prop:ii}), so is $\v+T$.   By Example 5.18 of \cite{BC2017}, in view of \cref{e:fixnotempty}, 
			we learn that the sequence $(\w_k)_\knn$
			converges to a point in
			$\fix (\v+T)$.
			Now combine with  \cref{e:seq:equiv} to learn
			that $(\z_k+k\v)_\knn$  converges to a point in 
			$Kv+\fix (\v+T)$,
						and the conclusion follows by recalling that
			$\fix (\v+T)=R_{-}\cdot \v+\fix (\v+T)$
			by \cref{fact:fixvT}.		\end{proof}

\subsection{A numerical example}
Consider the linear program \cref{ex:compLP}
(respectively the quadratic program \cref{ex:compQP}) 
\newline 
\begin{minipage}{0.4\linewidth}
\begin{equation}\label{ex:compLP}\tag{LP}
  \begin{alignedat}{2}
    &\text{minimize}\quad   & x_1 - 2x_2 &         \\
    &\text{subject to}\quad & -x_1 + x_2 &\le -2 \\
    && x_1 - x_2   &\le  1 \\
    && -x_1        &\le  0 \\
    && -x_2        &\le  0
  \end{alignedat}
\end{equation}
\end{minipage}%
\hfill
\begin{minipage}{0.55\linewidth}
\begin{equation}\label{ex:compQP}\tag{QP}
  \begin{alignedat}{2}
    &\text{minimize}\quad   & 0.5\,x_1^2 + 0.5\,x_2^2 + x_1 - 2x_2 &        \\
    &\text{subject to}\quad & -x_1 + x_2 &\le -2 \\
    && x_1 - x_2   &\le  1 \\
    && -x_1        &\le  0 \\
    && -x_2        &\le  0
  \end{alignedat}
\end{equation}
\end{minipage}

which was
given in \cref{ex:vcalc}\cref{ex:vcalc:i}
(respectively \cref{ex:vcalc}\cref{ex:vcalc:ii}).
In this section, we provide numerical illustrations of 
\cref{t:QPconvergence} when applied to 
\cref{ex:compLP} and \cref{ex:compQP}.
Additionally, we numerically verify the 
conclusion of \cref{ex:vcalc}.
For both \cref{ex:compLP} and \cref{ex:compQP} we set $\sigma=\tau=0.3$,
$\x_0=(0,0)$,
$\y_0=(0,0,-1,-1)$,
 and 
 $\z_0=(\x_0,\y_0)\tran$.
 Finally, following the notation of
 \cref{t:QPconvergence} we set 
 $(\forall \knn)$
 $\z_k=T^k\z_0$.
 Let $\knn$. We denote the 
 component of $(\z_k+k\v-(\v+T)(\z_k+k\v))$
 corresponding to $\x_k$
 (respectively $\y_k$)
 by $(\z_k+k\v-(\v+T)(\z_k+k\v))_\x$
(respectively $(\z_k+k\v-(\v+T)(\z_k+k\v))_\y$).

\begin{remark}
\label{rem:compareaffine}
Some comments are in order.
\begin{enumerate}
\item 
\label{rem:compareaffine:i}
Let $\w_0\in Z$ and
let $Q\colon Z\to Z$ be an \emph{affine} firmly nonexpansive operator.
We set $(\w_k)_\knn=(Q^k\w_0)_\knn$ and we let $\v_Q$
be the minimal norm vector in $\ran (\Id-Q)$.
The authors in \cite[Theorem~3.2]{BM2015} proved that $\w_k+k\v_Q$
converges to a point in $\fix (\v_Q+Q)$. 
\item
\label{rem:compareaffine:ii}
In view of \cref{rem:compareaffine:i} one wonders
if the limit of $(\z_k+k\v)$ lies in $ \fix (\v+T)$. 
Our numerical experiments provide  
a negative answer to this question, which proves 
the tightness of the conclusion of \cref{t:QPconvergence}.
Indeed,
as the plots in \Cref{fig1} and \Cref{fig2} below show, 
the  sequence $\z_k+k\v-(\v+T)(\z_k+k\v)\to u^*\neq \bzero$.
Recalling \cref{t:QPconvergence},
this in turn implies that $\z_k+k\v\to z^*\not\in \fix (\v+T)$.
\end{enumerate}
\end{remark}
\begin{figure}[h!]
    \centering
    \includegraphics[scale=0.55]{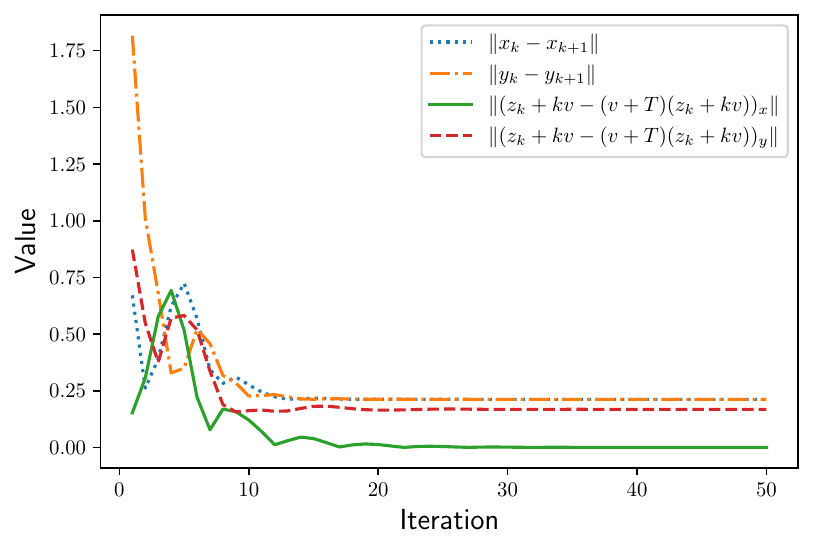}
    \includegraphics[scale=0.55]{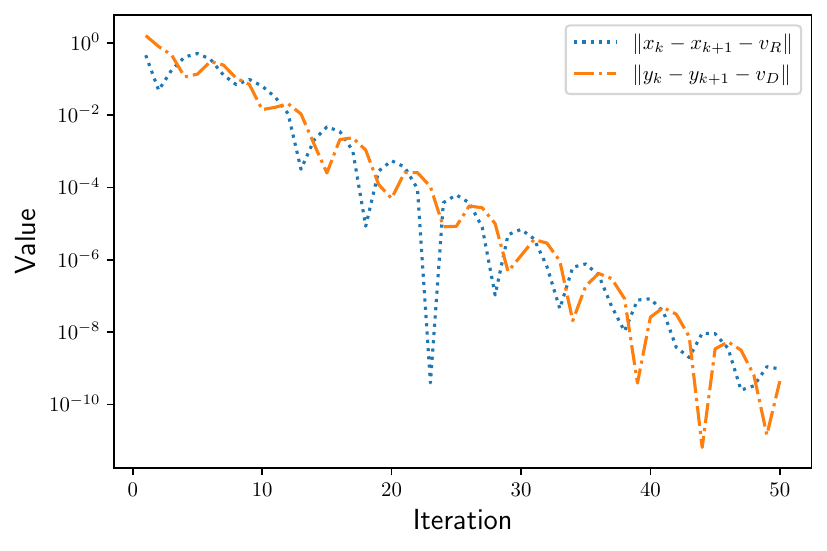}
    \caption{Python plots to illustrate the 
    iterative behaviour of PDHG when applied to solve 
    \cref{ex:compLP}.
    Left: The first $50 $  terms  of the sequences $(\norm{\x_k-\x_{k+1}})_\knn$ (the blue dotted curve)
    and $(\norm{\y_k-\y_{k+1}})_\knn$ (the orange dash-dotted curve) are depicted.
    Also, the first $50 $  terms  of the norms of the $\x-$ and $\y-$ 
    components
    of
    the sequence 
    $(\z_k+k\v-(\v+T)(\z_k+k\v))_\knn$
    are depicted (the solid green curve and the dashed red curve). 
    Right: The first $50 $  terms of the sequences $(\norm{\x_k-\x_{k+1}-\vx})_\knn$ (the  blue dotted curve)
    and $(\norm{\y_k-\y_{k+1}-\vy})_\knn$ (the orange  dotted curve) are depicted.
    }
    \label{fig1}
\end{figure}
\begin{figure}[h!]
    \centering
    \includegraphics[scale=0.55]{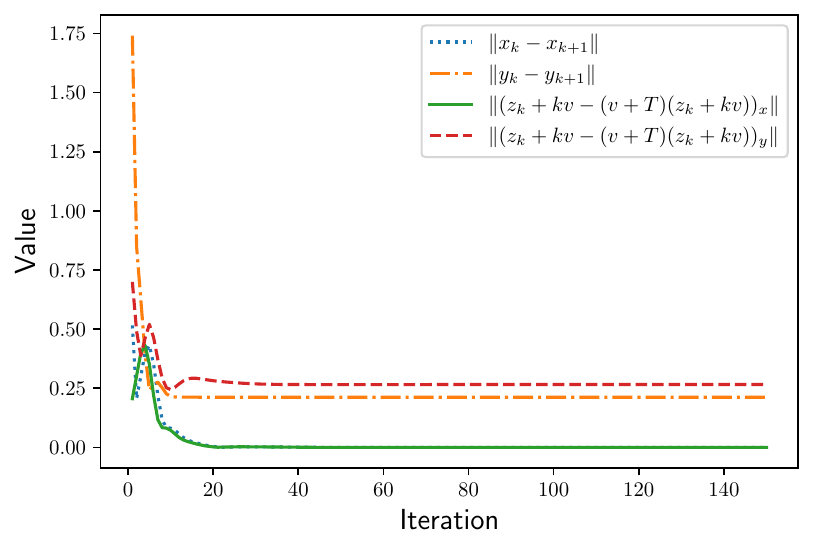}
    \includegraphics[scale=0.55]{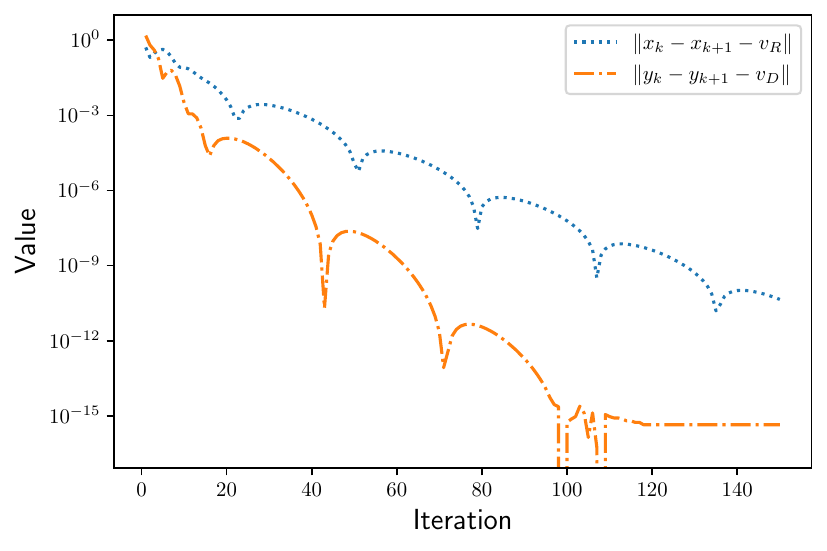}
    \caption{Python plots to illustrate the 
    iterative behaviour of PDHG when applied to solve 
    \cref{ex:compQP}.
    Left: The first $150 $ terms of the sequences $(\norm{\x_k-\x_{k+1}})_\knn$ (the blue dotted curve)
    and $(\norm{\y_k-\y_{k+1}})_\knn$ (the orange dash-dotted curve) are depicted.
    Also, the first $150 $ terms of the norms of the $\x-$ and $\y-$ 
    components of
    the sequence 
    $(\z_k+k\v-(\v+T)(\z_k+k\v))_\knn$
    are depicted (the solid green curve and the dashed red curve). 
    Right: The first $150 $ terms  of the sequences $(\norm{\x_k-\x_{k+1}-\vx})_\knn$ (the  blue dotted curve)
    and $(\norm{\y_k-\y_{k+1}-\vy})_\knn$ (the orange  dotted curve) are depicted.}
  \label{fig2}  
\end{figure}

\subsection{Computing the infimal displacement vector} \label{subsec:constrained}
In this section, we derive a characterization of $\ran(\Id-T)$ for \cref{e:qpform} as the solution to a system of convex constraints in the case that $K$ is polyhedral.    
  A special case is $K=\R^m_-$, i.e., QP \cref{eq:qpprim}. This formula is useful in case one wants to compute with $\ran(\Id-T)$, e.g., the determination of the infimal displacement vector via an interior-point method.

We note that \cref{lem:ran:phd:v}\cref{lem:ran:phd:v:iv} already yields such a characterization.  However, a naive translation of  \cref{lem:ran:phd:v:iv} to a system of constraints introduces four auxiliary vectors. The following result shows that two auxiliary vectors (denoted $\w$ and $\y$ below) suffice.  

            \begin{lemma}\label{lem:quadratic_conic_V}
              In the setting of Problem \cref{e:qpform}, suppose that 
                % $X$ is finite-dimensional, $K = \{\bzero\}$, and 
                $K$ is a polyhedral cone. Define 
            \begin{equation}\label{e:def:qp_V}
                V:=\left \{(\r, \d)\mbox{$\in X\times Y$}~\Bigg|~ \exists(\w,\y)\in X\times Y \;\;\mbox{\rm such that}\begin{array}{rl}
                &\frac{1}{\tau} \d - (\lin \w + \b) \in K,\\
                &\y - \d \in K^\ominus,\\
                &\frac{1}{\sigma}\r - H \r = -H \w + \lin^\ast \y + \c
            \end{array} \right\}.
            \end{equation}
            Then
            \begin{equation}\label{lem:eq:qp_V}
               V = M^{-1}(\ran (\partial F +S))=\ran(\Id-T).
            \end{equation}
            \end{lemma}
            \begin{proof}
             See Appendix~\ref{App:3}.
          \end{proof}
        		
\section{Application to standard conic primal form}
  \label{sec:standconic}
		In this section, we  {consider problems of the form \cref{e:esform}} under the assumptions
		\begin{empheq}[box=\mybluebox]{equation}
			\text{$C$ is a nonempty closed convex cone of $X$,  $\c\in X$, and  {$K=\{\bzero\}$}.}
		\end{empheq}
		
		In other words, the problem under consideration is:
		\begin{equation}
			\begin{array}{rl}
				\label{eq:c0:problem}
				\minimize{x\in C} & \scal{\c}{\x} \\
				% \mbox{\rm subject to} &A\x-\b\in K,
				\mbox{\rm subject to} &\lin \x-\b=\bzero.
			\end{array}
		\end{equation}
Problem~\cref{eq:c0:problem} is commonly known as standard conic primal form since it generalizes linear programming in standard equality form, which takes $C=\R^n_+$.  However, the results in this section extend beyond LP since polyhedrality is not assumed.
{Specializing} \cref{e:impt:xup:es}, the PDHG update to solve 
		\cref{eq:c0:problem} is 
		\begin{equation}
			\label{e:pdhg:CP:update:K0}
			\begin{pmatrix}
				\x^+\\
				\y^+
			\end{pmatrix}=	T	
			\begin{pmatrix}
				\x\\
				\y
			\end{pmatrix}
			:=	\begin{pmatrix}
				P_{C}(\x - \sigma \lin^* \y-\sigma\c)\\
				\y+\tau \lin (2\x^+-\x)-\tau \b
			\end{pmatrix}.
		\end{equation}

\begin{lemma}
\label{lem:C00}
For \cref{e:pdhg:CP:update:K0} we have 
			\begin{equation}
				\cran(\Id-T)=M^{-1}\Big((\overline{ C^\ominus+\ran \lin^*}+\c)\times (\b-{\overline{\lin(C)}})\Big) . 
    \label{eq:lem:c00}
			\end{equation}
		\end{lemma}
		\begin{proof}
			Recalling \cref{eq:ff^*:cone}   and \cref{e:def:fcC}, on the one hand we have 
			\begin{equation}
				\label{e:domsi0}
				\text{$\dom f=C$, $\dom g=\{\b\}$
					and $\dom g^*=Y$.}
			\end{equation}
			On the other hand,
			it follows from 
			\cite[Corollary~16.39~and~Corollary~16.30]{BC2017},
			\cite[Theorem~3.1]{Zara}
    and \cite[Corollary~6.50]{BC2017}
			that 
			% \begin{subequations}
				\begin{align}
					\overline{\dom}\ f^*
					&=\overline{\dom}\ (\scal{\c}{\x}+\iota_C)^*=\overline{\dom}\ \partial (\scal{\c}{\x}+\iota_C)^*
					\\
					&=\overline{\ran}\ \partial(\scal{\c}{\x}+\iota_C)
					=\overline{\ran}\ (\c+N_C)
					\\
					&=\c+\overline{\ran}\ N_C=\c+  C^\ominus . 
                  \label{e:domsii0}
				\end{align}
              
			% \end{subequations}
			Now combine \cref{e:domsi0}, \cref{e:domsii0}
			and \cref{prop:T:prop}\cref{prop:T:prop:vi}.
		\end{proof}

The following lemma, analogous to \cref{lem:quadratic_conic_V}, presents a parsimonious description of $\cran(\Id-T)$ via constraints in the context of the standard conic primal form.
\begin{lemma}
\label{lem:conic_V}
		    In the setting of Problem 
\cref{e:pdhg:CP:update:K0}, define
                $R :=((C^\ominus + \ran \lin^\ast + \c) \times (\b - \lin(C)),$ and
              \begin{equation}\label{e:def:V}
                V:=\left \{(\r, \d)\mbox{$\in X\times Y$}~\Bigg|~ \mbox{$\exists (\w,\y)\in X\times Y$ \;\;{\rm such that}} \begin{array}{rl}
                &\frac{1}{\tau} \d = \lin \w + \b,  \\
                &\r - \w \in C,\\
                &\frac{1}{\sigma}\r - (\lin^\ast \y + \c) \in C^\ominus
            \end{array} \right\}.
            \end{equation}
            Then
            \begin{equation}\label{lem:eq:conic_V}
                V = M^{-1}R. \qquad 
            \end{equation}
		\end{lemma}
          \begin{proof}
Let $(\r, \d) \in V$. Then there exists $ \w, \y$ 
              such that 
              $$
              \begin{array}{rl}
              &\frac{1}{\sigma}\r - \lin^\ast \d = \frac{1}{\sigma}\r  - (\lin^\ast \y + \c) + (\lin^\ast \y + \c) - \lin^\ast \d \in C^\ominus + \ran \lin^\ast + \c,\\
              &\frac{1}{\tau}\d - \lin \r 
              = \lin (\w - \r) + \b \in \b - \lin (C).\\
              \end{array}$$
              The inclusion $V \subseteq M^{-1}R$ follows from the definition of $M$  (see \cref{eq:def:M}).  We now show that 
              $M^{-1}R\subseteq V$. To this end let $\z \in R$ and define $(\r, \d) := M^{-1} \z$.
              Then there exist $\s \in C^\ominus, \t \in Y, \u \in C$ such that
              \begin{equation}\label{e:lem:rform}
                  \z = (\s + \lin^\ast \t + \c, \b - \lin \u) = \left(\tfrac{1}{\sigma}\r - \lin^\ast \d, \tfrac{1}{\tau}\d - \lin \r \right)=M(\r,\d).
              \end{equation}
 Define $\w:=\r - \u$ and $\y:=\d + \t$. 
 Combining this with \cref{e:lem:rform} yields
              $$
              \begin{array}{rl}
                  &\frac{1}{\tau} \d = \lin \r + \b - \lin \u 
                  = \lin \w + \b  \\
                  &\r - \w = \u \in C\\
                  &\frac{1}{\sigma}\r - (\lin^\ast \y + \c) = \lin^\ast \d + \s + \lin^\ast \t + \c - (\lin^\ast \y + \c) = \s + \lin^\ast \y + \c  - (\lin^\ast \y + \c) \in C^\ominus,
              \end{array}
              $$
              which implies that $M^{-1}\z=(\r, \d) \in V$. This completes the proof.
          \end{proof}

\begin{remark}
          Let $R$ and $V$ be defined as in \cref{lem:conic_V}.
          In view of  
          \cref{eq:lem:c00} it is clear that
          $\cran(\Id-T)=\overline{V}$.
\end{remark}

  \subsection{Special case: \texorpdfstring{$\ker\lin\cap C=\{\bzero\}$}{$\ker\lin \cap C = \{0\}$}}
 In this section, we establish that  $\v\in\ran(\Id-T)$ for a subclass of \cref{eq:c0:problem}.  The general version of our result is in \cref{p:ellpsepgenv}, and then the general result is specialized to a  particular problem class in \cref{sec:ellsep}.   We start with the following useful lemma.
		\begin{lemma}
			\label{lem:ncsets}
			Suppose $X$ is finite-dimensional and 
			that  $C_1$ and $C_2$ are nearly convex\footnote{Suppose that $X$ is finite-dimensional.
				A subset $E$ of $X $ is \emph{nearly convex} if there exists a convex set $C\subseteq X$
				such that $C\subseteq E\subseteq \overline{C}$.}  subsets of $X$
			such that 
			$\overline{C_1}=\overline{C_2}$. Then the following hold.
			\begin{enumerate}
				\item 
				\label{lem:ncsets:i}
				$\ri{C_1}=\ri{C_2}$.
				\item 
				\label{lem:ncsets:ii}
				Suppose that 
				$(\exists i\in \{1,2\})$ 
				$C_i=X$.
				Then $C_1=C_2=X$.
			\end{enumerate}	
		\end{lemma}	
		\begin{proof}
			\cref{lem:ncsets:i}:
			This is \cite[Proposition~2.12]{BMWnear}.
			\cref{lem:ncsets:ii}:
			Indeed, without loss of generality suppose that $C_1=X$.
			Observe that  	\cref{lem:ncsets:i} implies
			$X=\ri X=\ri C_1=\ri C_2\subseteq C_2\subseteq X$.
			Hence, $C_2=X$ as claimed. 
		\end{proof}	
		
		We now have the following corollary which will be used in the sequel.
		\begin{corollary}
			Suppose that
			$X$ is finite-dimensional and that 
			$K_1$ and $K_2$
			are closed convex cones of $X$
			such that $K_1\cap K_2=\{\bzero\}$.
			Then 
			$K_1^\ominus +K_2^\ominus=X$. 	
		\end{corollary}	
		\begin{proof}
			Indeed, it follows from 
			\cite[Corollary~16.4.2.]{Rock70}
			that 
			$\overline{K_1^\ominus +K_2^\ominus}=(K_1\cap K_2)^\ominus=\{\bzero\}^\ominus=X$.
			Now combine this with \cref{lem:ncsets}\cref{lem:ncsets:ii}
			applied with $C_1=K_1^\ominus +K_2^\ominus$ and $C_2=X$.	
		\end{proof}

		\begin{lemma}
			\label{lem:Cinf}
 Recalling \cref{e:pdhg:CP:update:K0},
		for Problem~\cref{eq:c0:problem}, the   following hold:
			\begin{enumerate}
				\item
				\label{lem:C0:iv:aa}
				${\vx}\in -C$.
				\item 
				\label{lem:C0:iv:b}
				$\sigma \lin^*{\vy}=P_{C^\ominus}(-{\vx}+\sigma\lin^*{\vy})$.
                % \in C^\ominus$.
				\item 
				\label{lem:C0:iv:b2}
				$-{\vx}=P_{C}(-{\vx}+\sigma\lin^*{\vy})$.
				\item 
				\label{lem:C0:iv:c}
				${\vy}\in  (\lin(C))^\ominus$.
				\item 
				\label{lem:C0:iv:dd}
				Suppose that $\ker \lin\cap C=\{\bzero\}$. Then 
				\begin{enumerate}
					\item 
					\label{lem:C0:iv:a}
					$\vx=\bzero$.
					\item 
					\label{lem:C0:iv:d:a}
					$\overline{C^\ominus+\ran \lin^*}=X$.
				\end{enumerate}
				
			\end{enumerate}
		\end{lemma}	
		\begin{proof}
			Let $\z_0=(\x_0,\y_0)\in X\times Y$.
			Update via $\z_{k+1}=T\z_k$, where $T$
			is defined as in \cref{e:pdhg:CP:update:K0}.
			Then 
			\begin{equation}
				\label{e:zkseq}
				\text{the sequence  $(\x_k)_\knn$ lies in $C$.
				}
			\end{equation}
			\cref{lem:C0:iv:aa}:
			Because $C$ is a cone,  \cref{e:zkseq} implies that 
			$(\x_k/k)_{k\ge 1}$ lies in $C$.
			It follows from 
			\cref{fact:T:v}\cref{fact:T:v:i}
			that 
			\begin{equation}
				\x_k/k\to -{\vx}\in C,
			\end{equation}
			where the inclusion follows from the closedness 
			of $C$.
			\cref{lem:C0:iv:b}:
			It follows from \cref{e:pdhg:CP:update:K0} applied with 
			$\x$ replaced by $\x_k$ 
			and the Moreau decomposition, see, e.g., \cite[Theorem~6.30]{BC2017},
			that
			\begin{equation}
				\x_k-\x_{k+1}-\sigma \lin^*\y_k-\sigma\c
				= \x_k-P_C(\x_k-\sigma \lin^*\y_k-\sigma \c)-\sigma \lin^*\y_k-\sigma\c
				=P_{C^\ominus}(\x_k-\sigma \lin^*\y_k-\sigma \c).
			\end{equation}
			Dividing the above equation by $k\ge 1$,
			using the positive homogeneity of $P_{C^\ominus}$
   (see \cite[Theorem~5.6(7)]{Deutsch})
			and taking the limit as
			$k\to \infty$ in view of 
			\cref{fact:T:v}\cref{fact:T:v:i}\&\cref{fact:T:v:ii} and the continuity of $P_{C^\ominus}$
			we learn that 
			\begin{equation}
				\sigma\lin^*{\vy}\leftarrow\tfrac{\x_k-\x_{k+1}}{k}+\sigma \lin^*\Big(\tfrac{-\y_k}{k}\Big)-\Big(\tfrac{\sigma \c}{k}\Big)
				=P_{C^\ominus}\Big(\tfrac{\x_k}{k}+\sigma\lin^*\big(\tfrac{-\y_k}{k}\big)-\tfrac{\sigma \c}{k}\Big)\to P_{C^\ominus}(-{\vx}+\sigma \lin^*{\vy}).
			\end{equation}
			That is, $\sigma \lin^*{\vy}=P_{C^\ominus}(-{\vx}+\sigma\lin^*{\vy})\in C^\ominus$. 
			\cref{lem:C0:iv:b2}: 
			It follows from \cref{lem:qp:v:K}\cref{lem:qp:K:v:v}
			that $\scal{{\vx}}{\lin^*{\vy}}=\scal{\lin {\vx}}{{\vy}}=0$.
			Now combine this with \cref{lem:C0:iv:aa} and \cref{lem:C0:iv:b}
			in view of, e.g., \cite[Proposition~6.28]{BC2017}.
\cref{lem:C0:iv:c}:
			It follows from
			\cref{lem:C0:iv:b}
			and \cite[Proposition~6.37(ii)]{BC2017}
			that
			${\vy}\in (\lin^*)^{-1}C^\ominus=(\lin (C))^\ominus$.
\cref{lem:C0:iv:a}:
			Indeed,  \cref{lem:C0:iv:aa}
			and \cref{lem:qp:v:K}\cref{lem:C0:ii} yield
			$\vx\in (-C)\cap \ker \lin=-(C\cap \ker \lin)=\{\bzero\}$.
			Hence, $\vx=\bzero$ as claimed.
			\cref{lem:C0:iv:d:a}:
			It follows from \cite[Fact~2.25(iv)~and~Proposition~6.35]{BC2017}
			that
			$\overline{ C^\ominus+\ran \lin^*}
			=\overline{C^\ominus+\cran \lin^*}
			=\overline{C^\ominus+(\ker \lin)^\perp}
			=\overline{C^\ominus+(\ker \lin)^\ominus}
			=(C\cap \ker \lin)^\ominus=\{\bzero\}^\ominus=X$.
   The proof is complete.
		\end{proof}

  		Before we proceed, we recall the following useful 
		fact.
		\begin{fact}
			\label{fact:rock:cl}
			Suppose that $X$ is finite-dimensional 
			and that
			$\ker \lin\cap C=\{\bzero\}$.
			Then $\lin(C)$ is closed.
		\end{fact}
		\begin{proof}
			See \cite[Theorem~9.1]{Rock70}.
		\end{proof}
		
		\begin{lemma}
			\label{lem:C0}
			Suppose that $X$  and $Y$ are finite-dimensional 
   and that $\ker \lin\cap C=\{\bzero\}$.
			Then for Problem~\cref{eq:c0:problem} we  have: 
			\begin{enumerate}
				\item 
				\label{lem:C0:iv:d}
				$\lin(C)$ is a nonempty closed convex cone.
				\item 
				\label{lem:C0:iv:de}
				${C^\ominus+\ran \lin^*}=X$.
				\item 
				\label{lem:C0:iv:e}
				$\cran(\Id-T)=M^{-1}\Big((X \times (\b-{{\lin(C)}})\Big) $.
				\item 
				\label{lem:C0:iv:f}
				Let $\z\in \lin(C)$. Then
				$\scal{{\vy}}{\tfrac{1}{\tau}{\vy}-(\b-\z)}\le 0$.
				\item 
				\label{lem:C0:iv:g}
				$\scal{{\vy}}{\tfrac{1}{\tau}{\vy}-\b}\le 0$. 
				\item 
				\label{lem:C0:iv:h}
				$\tfrac{1}{\tau}{\vy}=P_{\b-\lin(C)}(\bzero)=P_{(\lin (C))^\ominus}(\b)$.
			\end{enumerate}
			Let $\overline{\u}\in C$ be such that $\tfrac{1}{\tau}{\vy}=\b- \lin \overline{\u}$.
			Then we have
			\begin{enumerate}[resume]
				\item 
				\label{lem:C0:iv:k}
				$\scal{\lin \overline{\u}}{{\vy}}=\scal{\overline{\u}}{\lin^*{\vy}}=0$.
				\item 
				\label{lem:C0:iv:l}
				$\tfrac{1}{\tau}\norm{{\vy}}^2=\scal{\b}{{\vy}}$.
				\item 
				\label{lem:C0:iv:m}
				$\overline{\u}=P_{C}(\overline{\u}+\lin^*{\vy})$
				and $\lin^*{\vy}=P_{C^\ominus}(\overline{\u}+\lin^*{\vy})$.
			\end{enumerate}
		\end{lemma}

		\begin{proof}		
			\cref{lem:C0:iv:d}:
			Use \cref{fact:rock:cl}
			to learn that $\lin (C)$ is closed.
			The convexity is clear because $\lin$ is linear and 
			$C$ is convex. The conclusion
			$\lin (C)$ is a cone is straightforward.
			\cref{lem:C0:iv:de}:
			Combine
\cref{lem:Cinf}\cref{lem:C0:iv:d:a}
			and
\cref{lem:ncsets}\cref{lem:ncsets:ii} applied 
			with $C_1=X$ and $C_2=\overline{C^\ominus+\ran \lin^*}$.
			\cref{lem:C0:iv:e}:
			Combine 	\cref{lem:C00}, \cref{lem:C0:iv:de}
			and  \cref{lem:C0:iv:d}.
			\cref{lem:C0:iv:f}:
   Combine \cref{e:domsi0}
    and \cref{p:vobang}\cref{eq:vobang:ii}
    in view of \cref{lem:Cinf}\cref{lem:C0:iv:d:a}.
\cref{lem:C0:iv:g}:
			This is a direct consequence of 
			\cref{lem:C0:iv:f} by setting $\z=\bzero$.
			\cref{lem:C0:iv:h}:
			It follows from 
			\cref{lem:C0:iv:e}
			in view of \cref{e:def:v:Pm} that 
			$(-\lin^*{\vy},\tfrac{1}{\tau}{\vy})=M\v\in X\times ({\b-\lin(C)})$.
			That is, $\tfrac{1}{\tau}{\vy} \in {\b-\lin(C)}$.
			Now combine this with \cref{lem:C0:iv:f}
			and \cref{lem:C0:iv:d} 
			in view of \cite[Theorem~3.16]{BC2017} 
			to learn that $\tfrac{1}{\tau}{\vy}=P_{\b-\lin(C)}(\bzero)$.
			Finally, using, e.g., \cite[Proposition~3.19]{BC2017} we have 
			% \begin{subequations}
				\begin{align}
					\tfrac{1}{\tau}{\vy}
					&=P_{\b-\lin(C)}(\bzero)=\b+P_{-\lin(C)}(-\b)=\b-P_{\lin(C)}(\b)
					\\
					&=(\Id-P_{\lin(C)})(\b)=P_{(\lin(C))^\ominus}(\b).
				\end{align}
			% \end{subequations}            
            
            \cref{lem:C0:iv:k}\&\cref{lem:C0:iv:l}:
			Indeed, it follows from 
			\cref{lem:C0:iv:g}
			and \cref{lem:Cinf}\cref{lem:C0:iv:b}
			that 
			\begin{equation}
				0\le 
				\scal{\b-\tfrac{1}{\tau}{\vy}}{{\vy}}
				=\scal{\lin \overline{\u}}{{\vy}}
				=\scal{\overline{\u}}{\lin^* {\vy}}
				\le 0,
			\end{equation}
			hence $\scal{\b-\tfrac{1}{\tau}{\vy}}{{\vy}}
			=\scal{\lin \overline{\u}}{{\vy}}
			=\scal{\overline{\u}}{\lin^* {\vy}}=0$ and the conclusion follows.
			\cref{lem:C0:iv:m}:
			Combine 
			\cref{lem:C0:iv:k} and \cref{lem:Cinf}\cref{lem:C0:iv:b}
			in view of, e.g., \cite[Proposition~6.28]{BC2017}.
		\end{proof}
  
		We now show a sufficient condition to have $\v\in \ran(\Id-T)$.
		\begin{proposition}
			\label{p:ellpsepgenv}
			Suppose that $X$ and $Y$ are finite-dimensional, 
   that $\ker \lin\cap C=\{\bzero\}$
			and that $\c=\bzero$. Then 
			$\v\in \ran (\Id-T)$.
		\end{proposition}	
		\begin{proof}
			In view of \cref{prop:T:prop}\cref{prop:T:prop:iv}
			and  \cref{lem:Cinf}\cref{lem:C0:iv:d:a}  we have 
			\begin{equation}
				\label{e:attdiff}
				\v\in \ran(\Id-T)\siff M\v
				=\begin{pmatrix}
					-\lin^*{\vy}
					\\
					\tfrac{1}{\tau}{\vy}
				\end{pmatrix}
				\in \ran (\partial F+S),
			\end{equation}
			where 
			$\partial F=N_{C}\times \{\b\}$.
			It follows from \cref{lem:C0}\cref{lem:C0:iv:h}
			that
			$(\exists \overline{\u}\in { C})$
			such that $\tfrac{1}{\tau}{\vy}=\b-\lin\overline{\u}$.
			Now consider the point 
			$\begin{pmatrix}
				\overline{\u}
				\\
				{-\vy}
			\end{pmatrix}
			\in C \times Y=\dom \partial F=\dom (\partial F+S)$.
			We have 
			\begin{equation}
				(\partial F+S)
				\begin{pmatrix}
					\overline{\u}
					\\
					{-\vy}
				\end{pmatrix}
				= \begin{pmatrix}
					N_{C}(\overline{\u})-\lin^*{\vy}
					\\
					\b-\lin\overline{\u}
				\end{pmatrix}
				\ni 
				\begin{pmatrix}
					\bzero-\lin^*{\vy}
					\\
					\b- \lin\overline{\u}
				\end{pmatrix}
				=
				\begin{pmatrix}
					-\lin^*{\vy}
					\\
					\tfrac{1}{\tau}{\vy}
				\end{pmatrix}
				=M\v.
			\end{equation} 
			This completes the proof in view of \cref{e:attdiff}.
		\end{proof}	
		
		\begin{proposition}
			\label{p:ellip:dyn}
			Suppose that $X$ and $Y$ are finite-dimensional, 
   that $\ker \lin\cap C=\{\bzero\}$
			and that $\c=\bzero$. 
			Let $(\x_0,\y_0)\in X\times Y$.
			Update via $(\forall \knn)$
			\begin{equation}
				(\x_{k+1},\y_{k+1})=T(\x_{k},\y_{k}).
			\end{equation}		
			Then the following hold.
			\begin{enumerate}
				\item
				\label{p:ellip:dyn:i}
				The sequence $(\x_k, \y_k+k{\vy})_\knn$ is bounded.
			\end{enumerate}
			Let $\overline{\x}$ be a cluster point of $(\x_k)_\knn$.
			Then we have
			\begin{enumerate}[resume]
				\item 
				\label{p:ellip:dyn:ii}
				$\overline{\x}\in C$.
				\item
				\label{p:ellip:dyn:iii}
				$\tau \lin\overline{\x}=\tau \b-{\vy}$.
				\item 
				\label{p:ellip:dyn:iv}
				$(\overline{\x},\bzero)\in \fix (\v+T)$.
			\end{enumerate}
		\end{proposition}
		\begin{proof}
			\cref{p:ellip:dyn:i}:
			Combine \cref{prop:T:prop}\cref{prop:T:prop:ii},
			and \cref{fact:fejer}
			in view of 
			\cref{lem:Cinf}\cref{lem:C0:iv:a}.
			\cref{p:ellip:dyn:ii}:
			This follows from the fact that $(\x_k)_\knn $
			lies in $C$ and $C$ is closed.
\cref{p:ellip:dyn:iii}:
			Suppose that $\x_{n_k}\to \overline{\x}$.
			It follows from 
			\cref{fact:T:v}\cref{fact:T:v:i}
			in view of 
			\cref{prop:T:prop}\cref{prop:T:prop:ii}
			that 
			$\x_{n_k}-\x_{n_k+1}\to \bzero$, hence 
			$\x_{n_k+1}\to \overline{\x}$.
			Therefore, using this and \cref{e:pdhg:CP:update:K0} applied with $\c=\bzero$
			we have 
			${\vy}\leftarrow \y_{n_k}-\y_{n_k+1}
   = \tau \b-\tau \lin(2\x_{n_k+1}-\x_{n_k})=\tau \b-\tau \lin\overline{\x}$.
			\cref{p:ellip:dyn:iv}:
			Indeed, using 
			\cref{p:ellip:dyn:ii}
			and 
			\cref{p:ellip:dyn:iii}
			we have
			% \begin{subequations}
				\begin{align}
					\v+T(\overline{\x},\bzero)
					&=(\bzero,{\vy})+(P_C(\overline{\x}-\sigma \lin^*\bzero), \bzero+\tau \lin(2P_C(\overline{\x}-\sigma\lin^*\bzero)-\overline{\x})-\tau \b)
					\\
					&=(\overline{\x},{\vy}+ \tau \lin\overline{\x}-\tau \b)=(\overline{\x},\bzero), 
				\end{align}
			% \end{subequations}
			and the conclusion follows.
		\end{proof}
		
		We recall the following fact.
		\begin{fact}
			\label{fact:fmono:ortho}
			Let $(\z_k)_\knn$ be Fej\'er monotone with respect to 
			a nonempty closed convex subset $D$ of $X$.
			Let $\w_1$ and $\w_2$ be two cluster points of $(\z_k)_\knn$.
			Then $\w_1-\w_2\in (D-D)^\perp$.
		\end{fact}
		\begin{proof}
			See, e.g., \cite[Lemma~2.2]{BDM15} or \cite[Theorem~6.2.2(ii)]{B96}.
		\end{proof}
   
In the special case considered in this section, namely $\ker \lin\cap C=\{\bzero\}$, $\c=\bzero$, for which we already know $\v_R=\bzero$ by \cref{lem:Cinf}\cref{lem:C0:iv:a}, we can show that the first component of the sequence $(\x_k,\y_k)_\knn$ converges, and furthermore, we can partly characterize the limit point as follows.
  
		\begin{theorem}\label{thm:conves}
			Suppose that $X$ and $Y$ are finite-dimensional, that $\ker \lin\cap C=\{\bzero\}$
			and that $\c=\bzero$. 
			% $z_0=\begin{psmallmatrix}x_0\\y_0\end{psmallmatrix}\in \RR^n\times \RR^m$.
			Let $(\x_0,\y_0)\in X\times Y$.
			Update via $(\forall \knn)$
			\begin{equation}
				(\x_{k+1},\y_{k+1})=T(\x_{k},\y_{k}).
			\end{equation}		
			Then there exists $\overline{\x}\in C$ such that the following hold.
			\begin{enumerate}
				\item
				\label{t:ellip:i}
				The sequence $(\x_k)_\knn$ converges to $\overline{\x}$.
				\item
				\label{t:ellip:ii}
				$\overline{\x}\in C\cap L$, where $L=\menge{\x\in X}{\lin \x=\b-\tfrac{1}{\tau}{\vy}}$.
			\end{enumerate}
		\end{theorem}
		\begin{proof}
			\cref{t:ellip:i}:
			In view of \cref{p:ellip:dyn}\cref{p:ellip:dyn:i} it suffices to show that 
			$(\x_k)_\knn$ has at most one cluster point.
			To this end suppose that $\overline{\x}$
			and $\hat{\x}$ are two cluster points of 
			$(\x_k)_\knn$, say $\x_{n_k}\to \overline{\x}$
			and $\x_{l_k}\to \hat{\x}$.
			After dropping to a subsubsequence
			and relabelling if needed we can and do assume that 
			$\z_{n_k}+n_k \v\to (\overline{\x},\overline{\y})$
			and 
			$\z_{l_k}+l_k \v\to (\hat{\x},\hat{\y})$.
			On the one hand, it follows from 
			\cref{p:ellip:dyn}\cref{p:ellip:dyn:iv}
			that $(\overline{\x},\bzero)$
			and $(\hat{\x},\bzero)$ lie in $\fix (\v+T)$.
			On the other hand, applying \cref{fact:fmono:ortho}
			with $D$ replaced by $\fix (\v+T)$, $\w_1$
			replaced by $(\overline{\x},\overline{\y})$,
			and $\w_2$ replaced by
			$(\hat{\x},\hat{\y})$ in view of \cref{p:ellip:dyn}\cref{p:ellip:dyn:iii} 
			applied to $\overline{\x}$ and $\hat{\x}$
			yields
			% \begin{subequations}
				\begin{align}
					0&=\scal{(\overline{\x},\overline{\y})-(\hat{\x},\hat{\y})}{(\overline{\x},\bzero)-(\hat{\x},\bzero)}_M
					\\
					&=\scal{(\overline{\x}-\hat{\x},\overline{\y}-\hat{\y})}{(\overline{\x}-\hat{\x},\bzero)}_M
					\\
					&=\scal{(\overline{\x}-\hat{\x},\overline{\y}-\hat{\y})}{(\tfrac{1}{\sigma}(\overline{\x}-\hat{\x})-\lin^*(\bzero),
						\bzero-(\lin\overline{\x}-\lin\hat{\x}))}
					\\
					&= \tfrac{1}{\sigma}\scal{(\overline{\x}-\hat{\x},\overline{\y}-\hat{\y})}{(\overline{\x}-\hat{\x},
						\bzero)}=\tfrac{1}{\sigma} \norm{\overline{\x}-\hat{\x}}^2.
				\end{align}
			% \end{subequations}
			That is $\overline{\x}=\hat{\x}$
			and the conclusion follows. 
			\cref{t:ellip:ii}:
			Combine \cref{t:ellip:i} and 
			\cref{p:ellip:dyn}\cref{p:ellip:dyn:ii}\&\cref{p:ellip:dyn:iii}.
		\end{proof}

		\subsection{Application to the ellipsoid separation problem}
\label{sec:ellsep}
An example of Problem~\cref{eq:c0:problem} in which $\c=\bzero$ and $\ker \lin \cap C=\{\bz\}$ is the ellipsoid separation problem, which we describe in this section.  This problem asks: given two collections of finitely many ellipsoids, say $E_1,\ldots,E_k$ and $E_1',\ldots,E_l'$ all lying in $\R^d$, is there a hyperplane that strictly separates $E_1,\ldots,E_k$ from $E_1',\ldots,E_l'$?  This problem is a robust extension of the classic binary classification problem. ``Robust'' in this context means that the locations of the data points are known only up to an ellipsoid, and that the separating hyperplane should be correct for all possible actual locations of the points.  See, e.g., Shivaswamy et al.
\cite{Shiva}.  We start with a characterization of separators whose proof (omitted) follows directly from the standard hyperplane separation theorem.

\begin{fact}\label{lem:convsep}
Suppose that  $X$  is finite-dimensional.
Let $E_1,\ldots, E_k$ and $E_1',\ldots,E_l'$ be $k+l$ nonempty convex compact bodies lying in $X$.  Then there exists $\bm{a}\in X\smallsetminus\{\bzero\}$, $b\in \R$ such that $\langle\bm{a},\x\rangle<b$ for all $\x\in E_1\cup\cdots\cup E_k$ and $\langle\bm{a},\x\rangle>b$ for all $\x\in E_1'\cup\cdots\cup E_l'$ if and only if 
$\conv(E_1\cup\cdots\cup E_k)\cap \conv(E_1'\cup\cdots\cup E_l')=\emptyset$.
\end{fact}
Let us introduce further notation for the ellipsoids: say that 
\begin{subequations}
\begin{align}
E_i&:=\{\x\in\R^d:\Vert A_i^{-1}(\x-\c_i)\Vert\le 1\}, \quad i\in\{1,\ldots,k\},\\
E_i'&:=\{\x\in\R^d:\Vert B_i^{-1}(\x-\d_i)\Vert\le 1\}, \quad i\in\{1,\ldots, l\}.
\end{align}
\label{eq:Eidef}
\end{subequations}
Here, $A_1,\ldots,A_k,B_1,\ldots,B_l$ are $d\times d$ {invertible matrices}
and $\c_1,\ldots,\c_k,\d_1,\ldots,d_l$ are vectors (centers of the ellipsoids).  
The naive way of writing the problem ``Is $\conv(E_1\cup\cdots\cup E_k)\cap \conv(E_1'\cup\cdots\cup E_l')$ nonempty?" would introduce variables $\v_1,\ldots,\v_k,\w_1,\ldots,\w_l\in\R^d$ constrained to lie in the respective ellipsoids, and nonnegative multipliers $\lambda_1,\ldots,\lambda_k, \mu_1,\ldots,\mu_l$ satisfying $\lambda_1+\cdots+\lambda_k=\mu_1+\cdots+\mu_l=1$ and $\lambda_1\v_1+\cdots+\lambda_k\v_k=\mu_1\w_1+\cdots+\mu_l\w_l$.  However, this formulation is not convex due to the products $\lambda_i\v_i$, $\mu_i\w_i$.  

A standard rescaling trick (see, e.g., Boyd \& Vandenberghe \cite[Exercise~4.56]{BV}  attributed to Parrilo) reformulates the problem of nonemptiness of the intersection of convex hulls as a standard SOCP with variables $\lambda_1,\ldots,\lambda_k,\mu_1,\ldots,\mu_l,\p_1,\ldots,\p_k, \q_1,\ldots,\q_l$:
\[
\begin{array}{rl}
\minimize{} & 0 \\
\mbox{\rm subject to}
&
\lambda_1+\cdots+\lambda_k  = 1, \\
&
\mu_1+\cdots+\mu_l = 1, \\
&
\lambda_1\c_1+A_1\p_1+\cdots+\lambda_k\c_k+A_k\p_k-\mu_1\d_1+B_1\q_1-\cdots-\mu_l\d_l+B_l\q_l = \bz, \\
&
\Vert \p_i\Vert\le \lambda_i\quad\forall i\in\{1,\ldots,k\},\\
&
\Vert \q_i\Vert \le \mu_i\quad \forall
i\in\{1,\ldots, l\}. 
\end{array}
\]
Note that the constraints $\bm{\lambda}\ge \bzero$ and $\bm{\mu}\ge\bzero$ are redundant in this formulation and hence are dropped. The objective ``$\min 0$'' indicates that any feasible solution to the constraints yields a common point in the convex hulls.
  Let $X:=\RR^{d+1}\times\cdots \times \RR^{d+1}$.
We further rewrite this problem in the  form:
%   		\be
% 		\begin{array}{rl}
% 			\mbox{Find} & 		\x:=\left(\begin{array}{c}
% 				\lambda_1 \\
% 				\p_1 \\
% 				\vdots \\
% 				\lambda_k \\
% 				\p_k \\
% 				\mu_1 \\
% 				\q_1 \\
% 				\vdots\\
% 				\mu_l \\
% 				\q_l
% 			\end{array}
% 			\right)  \in X
%    \\
%    \\
% 			\mbox{\rm subject to} &
% 			\mathcal{A}
% 	\x
%    = \left(\begin{array}{c}
% 				1 \\
% 				1 \\
% 				\bz
% 			\end{array}
% 			\right)=:\b, 
%    \\
% 	\\
%    &
% \x \in C_2^{d+1}\times \cdots\times C_2^{d+1}=:\bf{C},
% 		\end{array}
%   \label{eq:ellsep_primal}
% 		\ee

% \tj{Rewrite of the problem above.}
		\be
			\mbox{Find} \;\;	\x:=\left(\begin{array}{c}
				\lambda_1 \\
				\p_1 \\
				\vdots \\
				\lambda_k \\
				\p_k \\
				\mu_1 \\
				\q_1 \\
				\vdots\\
				\mu_l \\
				\q_l
			\end{array}
			\right)  \in X
   \qquad  
			\mbox{\rm subject to} \;\;
			\begin{array}{rl}
			     & \mathcal{A}\x
                = \left(\begin{array}{c}
				1 \\
				1 \\
				\bz
			\end{array}
			\right)=:\b,  \\
			     & \x \in C_2^{d+1}\times \cdots\times C_2^{d+1}=:\bf{C},
			\end{array}
  \label{eq:ellsep_primal}
		\ee

where $C_2^{d+1}$ refers to the second-order cone in $\R^{d+1}$, and where
		\be
		\mathcal{A}=
		\left(
		\begin{array}{cccccccccc}
			1 & \bz^T&\cdots& 1 &\bz^T &0 &\bz^T & \cdots & 0 & \bz^T\\
			0 &\bz^T & \cdots & 0 & \bz^T & 1 & \bz^T&\cdots& 1 &\bz^T  \\
			\c_1 & A_1 & \cdots &\c_k & A_k &
			-\d_1 &B_1 & \cdots & -\d_l & B_l
		\end{array}
		\right).
  \label{eq:ellsepAdef}
		\ee
	This is a convex feasibility problem which can be recast as: 
\begin{equation}
			\label{prob:ellip}
			\min_{\x\in X}\;\;
			\iota_{\bf C}(\x)+\iota_{\{\bzero\}}(\lin \x-\b).
		\end{equation}
		By setting
		$K=\{\bzero\}$, $\c=0$ and $f=\iota_{\bf C}$ in 
		\cref{eq:c:problem}
		and recalling 
		\cref{e:pdhg:CP:update:K0}
		we learn that the PDHG update for the problem becomes 
		\begin{equation}
			\label{e:pdhg:ellipsUp}
			\begin{pmatrix}
				\x^+\\
				\y^+
			\end{pmatrix}		
			:=	
   T	\begin{pmatrix}
				\x\\
				\y
			\end{pmatrix}	
:=
			\begin{pmatrix}
				P_C(\x - \sigma \lin^* \y)\\
				\y+\tau \lin(2\x^+-\x)-\tau \b
			\end{pmatrix} .
		\end{equation}
Thus, the main work for PDHG in this case is multiplication by $\lin$ and $\lin^*$ and projection onto $\mathbf{C}$.
		
		\begin{lemma}
			\label{lem:cqs}
			We have $\bf{C}\cap \ker \lin=(-\bf{C})\cap \ker\lin=\{\bzero\}.$
		\end{lemma}
		\begin{proof}
			It suffices to show $ \bf{C}\cap \ker \lin=\{0\}.$
			Indeed, let 
			$\z= (\lambda_{1},\p_{1}, \ldots, \lambda_{k},\p_{k},\mu_{1},\q_{1}, \ldots,\mu_{l},\q_{l})\in \bf{C}\cap \ker \lin$.
			On the one hand 
			\begin{equation}
				\label{e:accl:i}
				\text{$\z\in \bf{C}$ $ \RA$ $(\forall i\in \{1,\ldots,k\})$
					$\lambda_{i}\ge \norm{\p_{i}}\ge 0$
					and 
					$(\forall j\in \{1,\ldots,l\})$
					$\mu_{j}\ge \norm{\q_{j}}\ge 0$.}
			\end{equation}
			On the other hand 
			\begin{equation}
				\label{e:accl:ii}
				\text{$\z\in \ker \lin$ $ \RA$
					$\sum_{i=1}^k\lambda_{i}=\sum_{j=1}^l\mu_{j} = 0$.}
			\end{equation}
			We learn from \cref{e:accl:i} and \cref{e:accl:ii}
			that 
			$(\forall i\in \{1,\ldots,k\})$
			$\lambda_{i}=0$
			and 
			$(\forall j\in \{1,\ldots,l\})$
			$\mu_{j}= 0$.
			This, together with \cref{e:accl:i}, yields 
			that  $(\forall i\in \{1,\ldots,k\})$
			$\p_{i}=\bzero$
			and 
			$(\forall j\in \{1,\ldots,l\})$
			$\q_{j}= \bzero$.
			That is $\z=\bzero$ as claimed.
			The proof is complete.
		\end{proof}
  
		We have the following two results.
		
		\begin{theorem}
			\label{lem:socp:v:i}
			For Problem~\cref{prob:ellip} we have 
			\begin{enumerate}
				\item
				\label{lem:socp:v:i:i}
				${\vx}=\bzero$.
				\item
				\label{lem:socp:v:i:ii}
				$\v\in \ran(\Id-T)$.
			\end{enumerate}
		\end{theorem}
		\begin{proof}
			\cref{lem:socp:v:i:i}:
			Combine \cref{lem:cqs} and \cref{lem:Cinf}\cref{lem:C0:iv:a}.
			\cref{lem:socp:v:i:ii}:
			Combine \cref{lem:cqs} and \cref{p:ellpsepgenv}.
		\end{proof}
		\begin{lemma}
			Recalling Problem~\cref{prob:ellip}, 
			let ${\vy}=(s,t,\w)\in \RR\times \RR\times \RR^d$
			and set 
			$\lin^*{\vy}=(\lambda_{1},\p_{1}, \ldots, \lambda_{k},\p_{k},\mu_{1},\q_{1}, \ldots,\mu_{l},\q_{l})$.
			Then the following hold:
			\begin{enumerate}
				% \item 
				% \label{lem:socp:v:i}
				% $\lin^*{\vy}\in C^\ominus$.
				% \item 
				% \label{lem:socp:vy:i}
				% $\norm{{\vy}}^2=\scal{b}{{\vy}}$.
				\item 
				\label{lem:socp:vy:ii}
				$\w=\bzero\RA {\vy}=\bzero$.
				\item 
				\label{lem:socp:vy:ii:ii}
				${\vy}\neq \bzero\RA s+t>0$.
				\item 
				\label{lem:socp:vy:iii}
				${\vy}=\bzero\siff \lin^*{\vy}=\bzero$.
				\item 
				\label{lem:socp:vy:iv}
				Suppose that $0\in \{\lambda_1,\ldots, \lambda_k,
				\mu_1,\ldots, \mu_l\}
				$. 
				Then  ${\vy}=\bzero$.
				\item 
				\label{lem:socp:vy:v}
				Suppose that $\bzero\in \{\p_1,\ldots, \p_k,
				\q_1,\ldots, \q_l\}
				$. 
				Then  ${\vy}=\bzero$.
			\end{enumerate}
   \label{lem:vdprops}
		\end{lemma}
		\begin{proof}
  It follows from 
			\cref{lem:C0}\cref{lem:C0:iv:l}
			that 
			\begin{equation}
				\label{e:stw}
				s+t=\tfrac{1}{\tau}(s^2+t^2+\norm{\w}^2)\ge 0.
			\end{equation}
			\cref{lem:socp:vy:ii}:
			Indeed,
			\cref{lem:Cinf}\cref{lem:C0:iv:b} implies that
			$\norm{\p_1}=\norm{A_1^*\w}\le -\lambda_1=-s-\scal{\c}{\w}$
			and similarly
			$\norm{B_1^*\w}\le -t+\scal{d_1}{\w}$.
			Consequently, $\w=\bzero$ $\RA$ $s\le 0$ and $t\le 0$.
			In view of \cref{e:stw} we
			conclude that $s+t=0$
			and hence $s^2+t^2=0$, equivalently, $s=t=0$.
			That is ${\vy}=\bzero$.
			\cref{lem:socp:vy:ii:ii}:
			Combine \cref{lem:socp:vy:ii}
			and \cref{e:stw}.
			\cref{lem:socp:vy:iii}:
			``$\RA$": This is clear.
			``$\LA$": 
			Observe that, because $A_1$ is invertible
			and $A_1^*\w=\bzero$ we must have $\w=\bzero$.
			Now combine this with 
			\cref{lem:socp:vy:ii}.
			\cref{lem:socp:vy:iv}:
			Without loss of generality, we may and do assume
			that $\lambda_1=0$. Then $\p_1=A_1^*\w=\bzero$, hence $\w=\bzero$.
			Now combine this with 
			\cref{lem:socp:vy:ii}.
			\cref{lem:socp:vy:v}:
			Without loss of generality, we may and do assume
			that $\p_1=A_1^*\w=\bzero$.
			Then $\w=\bzero$ and the conclusion follows in view of \cref{lem:socp:vy:ii}. 
		\end{proof}

		\begin{theorem}
  \label{t:ellips}
			Recalling Problem~\cref{prob:ellip}, let $T$ be defined as in \cref{e:pdhg:ellipsUp}.
			Let $\x_0\in \RR^{d+1}\times \ldots\times \RR^{d+1} $ and let $ \y_0\in \RR^{d+2}$.
			Update via $(\forall \knn)$
			\begin{equation}
				(\x_{k+1},\y_{k+1})=T(\x_{k},\y_{k}).
			\end{equation}		
			Then there exists $\overline{\x}\in C$ such that the following hold.
			\begin{enumerate}
				\item
				\label{t:ellips:i}
				The sequence $(\x_k)_\knn$ converges to $\overline{\x}$.
				\item
				\label{t:ellips:ii}
				$\overline{\x}\in {\bf C}\cap L$, where $L=\menge{\x\in \RR^{d+1}\times \ldots \times \RR^{d+1}}{\lin \x=\b-\tfrac{\vy}{\tau}}$.
			\end{enumerate}
		\end{theorem}
		\begin{proof}
			\cref{t:ellip:i}--\cref{t:ellip:ii}:
			Combine \cref{lem:cqs} and \cref{thm:conves}\cref{t:ellips:i}\&\cref{t:ellips:ii}.
		\end{proof}

As indicated by \cref{lem:convsep}, disjointness of the convex hulls, i.e., primal infeasibility of \cref{eq:ellsep_primal}, is certified by a separating hyperplane.  Furthermore, we know from \cref{lem:socp:v:i} that $\vy\ne\bzero$ in the infeasible case. We now argue a nonzero $\v_D$ encodes a separating hyperplane.  We first characterize such a hyperplane with the following lemma.

 \begin{lemma}
 \label{lem:E_in_H}
Given invertible $A\in\R^{n\times n}$, $\c\in\R^n$, $s\in\R$ and $\w\in\R^n\smallsetminus\{\bzero\}$,
     consider  the ellipsoid
       $
       E := \menge{\x\in \RR^n}{\Vert A^{-1}(\x-\c)\Vert\le 1}              
       $
       and the halfspace
       $H := \menge{\x\in \RR^n}{\langle \w,\x\rangle \le s}.
$       
The following hold.
\begin{enumerate}
\item 
\label{lem:E_in_H_i}
$E\subseteq H$ $\siff$ $ s\ge \Vert A^*\w\Vert + \langle\c,\w\rangle.$
\item 
\label{lem:E_in_H_ii}
$E\subseteq \inte H$ $\siff$ $ s> \Vert A^*\w\Vert + \langle\c,\w\rangle.$
\end{enumerate}
   \end{lemma}
   \begin{proof}
Let $\x\in \RR^n$.
Then 
\begin{equation}
\label{eq:ellpsephyp}
 \langle \w,\x\rangle = \langle \w,\x-\c\rangle + \langle \w, \c\rangle 
   = \langle A^*\w, A^{-1}(\x-\c)\rangle + \langle \w,\c\rangle .
\end{equation}
\cref{lem:E_in_H_i}:
$``\LA"$: Let $\x\in E$.
Using \cref{eq:ellpsephyp} and Cauchy--Schwarz
we have $   \langle \w,\x\rangle    \le \Vert A^*\w\Vert\cdot \Vert A^{-1}(\x-\c)\Vert +
   \langle \w,\c \rangle 
   \le \Vert A^*\w\Vert +\langle \w,\c\rangle 
   \le s$.
$``\RA"$: Let $\x=A (A^*\w/\norm{A^*\w})+\c$. Then $\x\in E$, hence $\x\in H$ 
and \cref{eq:ellpsephyp} implies 
$ s\ge\langle \w,\x\rangle 
% = \langle \w,\x-\c\rangle + \langle \w, \c\rangle 
   = \langle A^*\w, A^{-1}(\x-\c)\rangle + \langle \w,\c\rangle
   = \norm{A^*\w}+ \langle \w,\c\rangle.$
   \cref{lem:E_in_H_ii}:
   The proof proceeds similar to the proof of \cref{lem:E_in_H_i}.
\end{proof}

   We now state and prove our main result for the ellipsoid separation problem, which states that, because $\v\in\ran(\Id-T)$ (\cref{lem:socp:v:i}\cref{lem:socp:v:i:ii}), a nonzero $\v$ indicates inconsistency, and furthermore, a nonzero $\v$ encodes a strict separating hyperplane.
\begin{theorem}
\label{thm:esep:cert}
Given $k+l$ ellipsoids specified by \cref{eq:Eidef},
let $\vy=:(s,t,\w)$ and recall \cref{e:pdhg:ellipsUp}. Then
the following are equivalent.
\begin{enumerate}
    \item 
    \label{thm:esep:convempty}
    $\conv(E_1\cup\cdots\cup E_k)\cap
    \conv(E_1'\cup\cdots\cup E'_l)=\emptyset$,
    \item 
    \label{thm:esep:priminfeas}
    SOCP problem \cref{eq:ellsep_primal} is infeasible,
    \item \label{thm:esep:zeronotinrange}
    $\bzero\notin\ran(\Id-T)$, where $T$ is the PDHG operator given by \cref{e:pdhg:ellipsUp},
    \item 
    \label{thm:esep:vnonzero}
    $\v\ne\bzero$, 
    \item     
    \label{thm:esep:vdnonzero}
    $\v_D\ne\bzero$,
    \item 
    \label{thm:esep:sgtnegt}
    $s>-t$.   
    \newcounter{saveenumi}
    \setcounter{saveenumi}{\value{enumi}}
    \end{enumerate}
    Any one of these statements implies:
    \begin{enumerate}
    \setcounter{enumi}{\value{saveenumi}}
    \item 
    \label{thm:esep:sep}
    The hyperplane 
    $\menge{\x}{\langle\w,\x\rangle=s'}$
    strictly separates $E_1,\ldots,E_k$ from $E_1',\ldots,E_l'$, where $s'$ is chosen arbitrarily in $\left]-s,t\right[$.
\end{enumerate}
Conversely, the existence of $(\w,s')$ as in \cref{thm:esep:sep} implies all of \cref{thm:esep:convempty}--\cref{thm:esep:sgtnegt}.
\label{thm:esep}
\end{theorem}

\begin{proof}
{\cref{thm:esep:convempty}$\Leftrightarrow$ \cref{thm:esep:priminfeas}:} This was explained earlier in the formulation of \cref{eq:ellsep_primal}.  
{\cref{thm:esep:priminfeas}$\Leftrightarrow$\cref{thm:esep:zeronotinrange}:} We show the contrapositives.  If \cref{eq:ellsep_primal} has a   solution say $\x^*$, then $(\x^*,\bzero)$ is a fixed point of $T$ defined in
\cref{e:pdhg:ellipsUp}.  Equivalently, $\bzero\in\ran(\Id-T)$.
Conversely, suppose that  $(\x,\y)\in \fix T$.
Then $\x\in\mathbf{C}$ and $\lin \x=\b$, i.e., $\x$ solves \cref{eq:ellsep_primal}.
{\cref{thm:esep:zeronotinrange}$\Leftrightarrow$\cref{thm:esep:vnonzero}:} This follows from
\cref{lem:socp:v:i}\cref{lem:socp:v:i:ii}.
{\cref{thm:esep:vnonzero}$\Leftrightarrow$\cref{thm:esep:vdnonzero}:} This follows from \cref{lem:socp:v:i}\cref{lem:socp:v:i:i}. 
{\cref{thm:esep:vdnonzero}$\Leftrightarrow$\cref{thm:esep:sgtnegt}:}  The forward direction is established by  \cref{lem:vdprops}\cref{lem:socp:vy:ii:ii}, while the reverse direction is trivial.  
{\cref{thm:esep:sgtnegt}$\Rightarrow$\cref{thm:esep:sep}:}  Recalling the form of $\lin$ in \cref{eq:ellsepAdef}, we have
   \[
   \lin^*\v_D=\left(\begin{array}{c}
   s+\langle \c_1,\w\rangle \\
   A_1^*\w \\
   \vdots \\
   s+\langle \c_k,\w\rangle\\
   A_k^*\w \\
   t-\langle \d_1,\w\rangle \\
   B_1^*\w \\
   \vdots\\
   t-\langle \d_l,\w\rangle \\
   B_l^*\w 
   \end{array}
   \right).
   \]
   By \cref{lem:Cinf}\cref{lem:C0:iv:b}, we know that $\lin^*\v_D\in\mathbf{C}^{\ominus}$, in other words,
   \begin{align}
   -s-\langle \c_i,\w\rangle &\ge \Vert A_i^*\w\Vert, \quad i\in \{1,\ldots, k\},\\
   -t+\langle \d_i,\w\rangle&\ge \Vert B_i^*\w\Vert,\quad
   i\in \{1,\ldots, l\}.
   \end{align}
   Since $s>-t$, select an arbitrary $s'$ satisfying $s>-s'>-t$.  Then we obtain the inequalities
   \begin{align}
   s' &> \Vert A_i^*\w\Vert+\langle\c_i,\w\rangle, \quad i\in \{1,\ldots, k\},\\
   -s'&> \Vert B_i^*\w\Vert + \langle \d_i,-\w\rangle, \quad i\in \{1,\ldots, l\}.
   \end{align}
   In view of \cref{lem:E_in_H}, these inequalities show that $E_1,\ldots,E_k$ are strictly on one side of the hyperplane $\menge{\x}{\langle\w,\x\rangle=s'}$ while $E_1',\ldots,E_l'$ are strictly on the other side, thus establishing \cref{thm:esep:sep}. 
Finally, the converse statement at the end of the theorem follows from \cref{lem:convsep}.
   \end{proof}

   \section{Conclusion}
   We have developed a new formula for $\cran(\Id-T)$
   when $T$ is the PDHG operator.  We applied this formula to quadratic programming and the ellipsoid separation problem to show that in both cases, PDHG can diagnose inconsistency by checking the limiting value of $\z_{k}-\z_{k+1}$ as per \cref{fact:T:v}\cref{fact:T:v:ii}.  Both results used the conclusion 
   that $\v\in\ran(\Id-T)$, where $\v$ is the infimal displacement vector. We provided new results on the convergence of PDHG iterates for both problems. 
Many issues remain in understanding the landscape of PDHG for infeasible conic optimization problems.  Lest the reader suspect that \cref{fact:T:v}\cref{fact:T:v:ii} can always diagnose inconsistency, we point out that it is relatively easy to construct small contrived inconsistent problems such that $\bzero\in\cran(\Id-T)\smallsetminus\ran(\Id-T)$, meaning that the test based on \cref{fact:T:v}\cref{fact:T:v:ii} will fail to detect inconsistency.  There are also realistic examples when this occurs, for example, the unbounded case of the min-volume-ellipsoid problem (see, e.g., formulation (12a) in \cite{todd2007khachiyan}), which arises when the data points lie in a low-dimensional affine space.  

\subsection*{Acknowledgements}
The research of WMM and SAV was supported in part by Discovery Grants from the Natural Sciences and Engineering Research Council of Canada (NSERC).

 \small
\bibliography{pdhg}
\bibliographystyle{plain}
\begin{appendices}
\section{Proofs for the results of \cref{Appp:1}} 
\label{App:1}
\begin{myproof}[Proof of \cref{thm:ranform:mm}.]
        
			For simplicity, {we} set 
			$R:=(\ran B_1+\lin^* (\ran B_2)) \times 
			(\dom B_2 -\lin (\dom B_1))$.
			Let $(\u,\v)\in \ran ({\bf B}+S)$.
			Then 
			$(\exists (\x,\y)\in \dom {\bf B}
			=\dom B_1\times \dom B_2^{-1}=\dom B_1\times \ran B_2)$
			such that
			$
			(\u,\v)\in (B_1\x+\lin^*\y)\times (B_2^{-1}\y-\lin\x) 
			\subseteq (\ran B_1 +\lin^*(\ran B_2))
			\times (\dom B_2-\lin (\dom B_1))$.
			This proves that 
			\begin{equation}
				\label{e:forinc:mm}
				\text{$\ran({\bf B}+S)\subseteq R$ and
					hence $\cran{({\bf B}+S)}\subseteq\overline{R}$.}
			\end{equation}
   We now turn to the opposite inclusion. 
			Let $\r\in R$. Then there exists 
			{$\x\in \ran B_1$, $\y\in \dom B_2$,}
			$\w\in \ran B_2$
			and $\z\in \dom B_1$ such that
			% \begin{equation}
				$\r=(\x+\lin^*\w,\y-\lin\z).$
			% \end{equation}
			Recalling \cref{e:mm:sumSB} and 
			Minty's theorem (see, e.g., 
			\cite[Theorem~21.1]{BC2017}), we learn that 
			$(\forall n\ge 1)$
			$\tfrac{1}{n^2}\Id+{\bf B}+S$ is surjective.
			Consequently, $(\forall n\ge 1)$ there exists $\x_n\in \dom \bf B$
			such that  
			% \begin{equation}
   $	\r\in \left(\tfrac{1}{n^2}\Id+{\bB}+S\right)(\x_n) =\tfrac{1}{n^2}\x_n +{\bB} \x_n+S\x_n.$	
			% \end{equation}
			Equivalently, 
			\begin{equation}
				\label{e:seqloc:mm}
				\text{the sequence $\Big(\x_n,\r-\tfrac{1}{n^2}\x_n\Big)_{n\ge 1}$
					lies in  $\gra ({\bB}+S)$.
				}
			\end{equation}
			This in turn implies that 
			\begin{equation}
				\label{e:rloc:mm}
				\text{the sequence $\Big(\r-\tfrac{1}{n^2}\x_n\Big)_{n\ge 1}$
					lies in  $\ran({\bB}+S)$.
				}
			\end{equation}
			We will show that $(\x_n/n)_{n\ge 1}$
			is bounded.
			We set $\s=(\z,\w)$. We  claim that
			there exists $K\in \RR$
			such that 
			\begin{equation}
				\label{e:away:mm}
				\inf_{(\u,\v)\in \gra ({\bB}+S)}   
				\quad\scal{\u-\s}{\v-\r}\ge K.
			\end{equation}
			Indeed, let $(\u,\v)\in \gra ({\bB}+S)$.
			Then $(\exists (\a,\a^*)\in \gra B_1)$
			$(\exists (\b^*,\b)\in \gra B_2^{-1})$
			such that 
$\u=(\a,\b^*)$
					and $\v=(\a^*+\lin^*\b^*,\b-\lin \a)$.
			Now,
			\begin{subequations}
				\begin{align}
					&\qquad\scal{\u-\s}{\v-\r}
					\nonumber
					\\
					&=\scal{(\a,\b^*)-(\z,\w)}{(\a^*+\lin^*\b^*,\b-\lin \a)-(\x+\lin^*\w,\y-\lin \z)}
					\\
					&=\scal{(\a-\z,\b^*-\w)}{(\a^*+\lin^*\b^*-(\x+\lin^*\w),\b-\lin \a-(\y-\lin \z)}
					\\
					&=\scal{\a -\z}{\a^*+\lin^*\b^*-(\x+\lin^*\w)}
					+\scal{\b^*-\w}{\b-\lin \a-(\y-\lin \z)}
					\\
					&=\scal{\a -\z}{\a^*-\x}+\scal{\a -\z}{\lin^*(\b^*-\w)}
					+\scal{\b^*-\w}{\b-\y}-\scal{\b^*-\w}{\lin(\a-\z)}
					\\
					&=\scal{\a -\z}{\a^*-\x}+\scal{\b^*-\w}{\b-\y}+\scal{\b^*-\w}{\lin (\a-\z)}
					-\scal{\b^*-\w}{\lin (\a-\z)}
					\\
					&=\scal{\a -\z}{\a^*-\x}+\scal{\b^*-\w}{\b-\y}\ge K,
					\label{ine:M:mm}
				\end{align}
			\end{subequations}
			for some $K\in \RR$.
			The inequality in \cref{ine:M:mm}
			follows from applying  \cref{e:3*mono:mm} with $C$ replaced by $B_1$ by noting that
			$(\a,\a^*)\in \gra B_1$
			and again applying \cref{e:3*mono:mm} with 
   $X$ replaced by $Y$, and 
   $C$ replaced by $B_2$ by noting that $(\b,\b^*)\in \gra B_2$.
			This proves \cref{e:away:mm}.
			Now \cref{e:seqloc:mm} and \cref{e:away:mm} imply that
			$(\forall n\ge 1)$
			\begin{subequations}
				\begin{align}
					\tfrac{\norm{\x_n}^2}{n^2}
					&=\tfrac{1}{n^2}(\norm{\x_n}^2-\norm{\s}^2)+\tfrac{\norm{\s}^2}{n^2}
					\\
					&\le \tfrac{1}{n^2}(\norm{\x_n}^2-\norm{\s}^2)+\norm{\s}^2
					\\
					&= \tfrac{1}{n^2}(2\norm{\x_n}^2-(\norm{\x_n}^2+\norm{\s}^2))+\norm{\s}^2
					\\
					&\le \tfrac{1}{n^2}(2\norm{\x_n}^2-2\scal{\s}{\x_n})+\norm{\s}^2
					\\
					&=-2\scal{\x_n-\s}{-\tfrac{1}{n^2}\x_n}+\norm{\s}^2
					\\
					& =-2\scal{\x_n-\s}{\r-\tfrac{1}{n^2}\x_n-\r}+\norm{\s}^2
                    \\
                    &\le -2K+\norm{\s}^2.
					\label{se:bdd:mm}
				\end{align}
			\end{subequations}
   	That is, $(\x_n/n)_{n\ge 1}$
			is bounded as claimed.
			Taking the limit in \cref{e:rloc:mm} 
   as $n\to \infty$
   we learn that 
			$\r\in  \cran(\bB+S)$ and hence
			$R\subseteq  \cran(\bB+S)$.
			% This proves the opposite inclusion.
			The proof is complete.
		\end{myproof}

  %       		\begin{example}
		% 	\label{ex:notBHfriend}
		% 	Suppose that $X=Y=\RR$, that $\lin=-\Id$,
		% 	that\footnote{Let $C$
  %  be a nonempty closed convex subset of $X$.
  %  Here and elsewhere we use $N_C$ to denote 
  %  the \emph{normal cone operator}
  %  associated with $C$ defined by $N_C(x):=\menge{u\in X}{\sup\scal{u}{C-x}=0}$, if $x\in C$;
  %  and $N_C(x):=+\infty$, otherwise.}
  %  $B_1=N_{\left[0,+\infty\right[}$
		% 	and that $B_2=N_{\left]-\infty,0\right]}$.
		% 	Then 
		% 	\begin{equation}
		% 		\cran({\bB}+S)=	\left]-\infty,0\right]  \times \RR \subsetneqq \RR^2 =\overline{\ran {\bB}+\ran S}.
		% 	\end{equation}
		% \end{example}

		\begin{myproof}[Proof of \cref{ex:notBHfriend}.]
        
			Observe that $B_1$ and $B_2$ are subdifferential operators, hence $3^*$ monotone 
			by \cref{fact:subd3*}. 
			Moreover,
				\text{ ${\bB}=N_{\RR^2_+}$ and $S=\begin{psmallmatrix}
						0&-1
						\\
						1&0
					\end{psmallmatrix}	$}	
			On the one hand, clearly
				$\ran {\bB}=\RR^2_{-}$ and $\ran S=\RR^2$\; hence $\ran {\bB}+\ran S=\RR^2=\overline{\ran {\bB}+\ran S}$.
			On the other hand,
				$\dom B_1=\ran B_2=\left[0,+\infty\right[$ and  $\dom B_2=\ran B_1=\left]-\infty,0\right]$,
			and \cref{e:ran:form} yields
			$\cran({\bB}+S)=	\left]-\infty,0\right]  \times \RR \subsetneqq \RR^2 =\overline{\ran {\bB}+\ran S}$,
			as claimed.	
		\end{myproof}

		% \begin{example}
		% 	\label{ex:both3*}
		% 	Suppose that $X=\RR$, that $Y=\RR^2$, that 
		% 	$B_1\equiv 0$, that 
		% 	$B_2=\begin{psmallmatrix}
		% 		0&-1
		% 		\\
		% 		1&0
		% 	\end{psmallmatrix}$,
		% 	and that $\lin=\begin{psmallmatrix}
		% 		1
		% 		\\
		% 		0
		% 	\end{psmallmatrix}$.
		% 	Then 
		% 	the following hold:
		% 	\begin{enumerate}
		% 		\item 
		% 		\label{ex:both3*:i}
		% 		$\cran({\bB}+S)=\spn\{(0,1,0)\tran, (-1,0,1)\tran\}  \subsetneqq \RR^3$.
		% 		\item 	
		% 		\label{ex:both3*:ii}
		% 		$\overline{(\ran B_1+\lin^* (\ran B_2))} \times 
		% 		\overline{(\dom B_2 -\lin (\dom B_1))}=\RR^3$.
		% 		\item 
		% 		\label{ex:both3*:iii}
		% 		$	\cran({\bB}+S) \subsetneqq \overline{(\ran B_1+\lin^* (\ran B_2))} \times 
		% 		\overline{(\dom B_2 -\lin (\dom B_1))}.$
		% 	\end{enumerate}
		% \end{example}	
		\begin{myproof}[Proof of \cref{ex:both3*}.]
        
			It follows from \cref{lem:Snot3*} that
			$B_2$ is \emph{not} $3^*$ monotone. 
   Moreover, it is easy to verify that
$B_2^{-1}=-B_2$,
			that $\dom B_1=\RR$, 
			that $\ran B_1=\{0\}$,
			and that $\dom B_2=\ran B_2=\RR^2$.
			\cref{ex:both3*:i}: We have 
			\begin{equation}
				{\bB}=
				\begin{pmatrix}
					0&0&0
					\\
					0&0&1
					\\
					0&-1&0
				\end{pmatrix}		
				\;
				\text{and}
				\;
				S=
				\begin{pmatrix}
					0&1&0
					\\
					-1&0&0
					\\
					0&0&0
				\end{pmatrix}		
				\; \text{hence}\;
				{\bB}+S=
				\begin{pmatrix}
					0&1&0
					\\
					-1&0&1	\\
					0&-1&0
				\end{pmatrix}.
			\end{equation}	
			This proves \cref{ex:both3*:i}.
			\cref{ex:both3*:ii}:
			Indeed, we have 
			$\ran B_1+\lin^* (\ran B_2)=\{0\}+\lin^*(\RR^2)=\RR$
			and 
			$\dom B_2 -\lin (\dom B_1)=\RR^2$.
			\cref{ex:both3*:iii}:
			Combine \cref{ex:both3*:i} and \cref{ex:both3*:ii}.
		\end{myproof}	
		
\section{Proofs for the results of \cref{sec:PDHGrangeIdminusT}	} 
\label{App:2}	

 \begin{myproof}[Proof of \cref{prop:T:op}.]     
 
			\cref{prop:T:op:i}:
			Indeed,
			\begin{subequations}
				\begin{align}
					\cref{e:pdhg:gen:update}
					&\siff 	\x^+=\prox_{\sigma f}(\x-\sigma \lin^* \y),
					\nonumber
					\\
					&\qquad \y^+=\prox_{\tau g^*}(\y+\tau \lin(2\x^+-\x))
     \label{eq:prop:t:op:start}
					\\
					%=============================		
					&\siff 	\x-\sigma \lin^* \y\in \x^++\sigma \partial f(\x^+),
					\nonumber
					\\
					&\qquad \y+\tau \lin(2\x^+-\x)\in \y^++\tau \partial g^*(\y^+)
					\\				%=============================
					&\siff 	\x-\x^+\in \sigma \lin^* \y + \sigma \partial f(\x^+),
					\nonumber\\
					&\qquad \y-\y^+\in -\tau \lin(2\x^+-\x) + \tau \partial g^*(\y^+)\\
					%=============================
					&\siff 	\tfrac{1}{\sigma}(\x-\x^+)\in \lin^* \y + \partial f(\x^+),
					\nonumber\\
					&\qquad \tfrac{1}{\tau}(\y-\y^+)\in -\lin(2\x^+-\x) + \partial g^*(\y^+)\\
					%=============================
					&\siff 	\tfrac{1}{\sigma}(\x-\x^+)-\lin^* (\y-\y^+)\in \lin^* \y +\partial f(\x^+)-\lin^* (\y-\y^+),
					\nonumber\\
					&\qquad \tfrac{1}{\tau}(\y-\y^+)-\lin(\x-\x^+)\in-\lin(2\x^+-\x)+\partial g^*(\y^+)-\lin(\x-\x^+)\\
					%=============================
					&\siff 	\tfrac{1}{\sigma}(\x-\x^+)-\lin^* (\y-\y^+)\in \lin^* \y^+ +\partial f(\x^+),
					\nonumber \label{eq:T:op_1}\\
					&\qquad \tfrac{1}{\tau}(\y-\y^+)-\lin(\x-\x^+)\in \partial g^*(\y^+)-\lin \x^+\\
					%=============================
					&\siff 
					\begin{pmatrix}
						\tfrac{1}{\sigma}\Id_X&-\lin^*\\
						-\lin&\tfrac{1}{\tau}\Id_Y
					\end{pmatrix}
					\begin{pmatrix}
						\x-\x^+\\
						\y-\y^+
					\end{pmatrix}
					\in 
					\begin{pmatrix}
						\partial f(\x^+)\\
						\partial g^*(\y^+)
					\end{pmatrix}
					+
					\begin{pmatrix}
						0&\lin^*\\
						-\lin&0
					\end{pmatrix}
					\begin{pmatrix}
						\x^+\\
						\y^+
					\end{pmatrix}
					\\
					%=============================
					&\siff 
					M
					\begin{pmatrix}
						\x-\x^+\\
						\y-\y^+
					\end{pmatrix}
					\in 
					\begin{pmatrix}
						\partial f(\x^+)\\
						\partial g^*(\y^+)
					\end{pmatrix}
					+
					S
					\begin{pmatrix}
						\x^+\\
						\y^+
					\end{pmatrix}
					\\
					%=============================
					&\siff 
					M
					\begin{pmatrix}
						\x\\
						\y
					\end{pmatrix}
					\in 
					M
					\begin{pmatrix}
						\x^+\\
						\y^+
					\end{pmatrix}
					+
					\begin{pmatrix}
						\partial f(\x^+)\\
						\partial g^*(\y^+)
					\end{pmatrix}
					+
					S
					\begin{pmatrix}
						\x^+\\
						\y^+
					\end{pmatrix}=(M+{\partial F}+S)\begin{pmatrix}
						\x^+\\
						\y^+
					\end{pmatrix}   
     \\
     &\siff\begin{pmatrix}
						\x^+\\
						\y^+
					\end{pmatrix}
					\in(M+{\partial F}+S)^{-1}M\begin{pmatrix}
						\x\\
						\y
					\end{pmatrix}    
      \\
     &\siff\begin{pmatrix}
						\x^+\\
						\y^+
					\end{pmatrix}
					=(M+{\partial F}+S)^{-1}M\begin{pmatrix}
						\x\\
						\y
					\end{pmatrix} .   \label{eq:in_eq}
     \end{align}
     \end{subequations}
     Note that the last step \cref{eq:in_eq} follows because \cref{eq:prop:t:op:start} implies that, given $(\x,\y)$, there is exactly one solution $\z$ to the inclusion $(\z,M(\x,\y))\in\gra(M+\partial F+S)$.	
			The proof of 	\cref{prop:T:op:i} is complete.

			\cref{prop:T:op:ii}:
			Indeed,
			\begin{subequations}
				\begin{align}
					&	\quad\begin{pmatrix}
						\overline{\x}\\
						\overline{\y}
					\end{pmatrix}=		
					T\begin{pmatrix}
						\overline{\x}\\
						\overline{\y}
					\end{pmatrix}
					\\
					&\siff
					\overline{\x}=\prox_{\sigma f}(\overline{\x}-\sigma \lin^* \overline{\y})
					\;\text{  and   }\;  \overline{\y}=\prox_{\tau g^*}(\overline{\y}+\tau \lin \overline{\x})
					\\
					&\siff
					\overline{\x}-\sigma \lin^* \overline{\y}\in \overline{\x}+\sigma \partial f(\overline{\x})
					\;\text{  and   }\; \overline{\y}+\tau \lin \overline{\x}\in \overline{\y}+\tau \partial g^*(\overline{\y})
					\\
					&\siff
					-\sigma \lin^* \overline{\y}\in \sigma \partial f(\overline{\x})\;\text{  and   }\; \tau \lin \overline{\x}\in \tau  \partial g^*(\overline{\y})
					\\
					&\siff
					-\lin^* \overline{\y}\in \partial f(\overline{\x})\;\text{  and   }\; \overline{\y}\in  \partial g (\lin \overline{\x})
					\\
					&\siff\text{$(\overline{\x},\overline{\y})$ is a primal-dual solution of
						\cref{e:p:minc}.}
				\end{align}
			\end{subequations}	
			This completes the proof.
            \end{myproof}

 \begin{myproof}[Proof of \cref{prop:M}.]
 
			\cref{prop:M:0}:
			Indeed, let $(\u,\v)\in X\times Y$.
			Then using Cauchy--Schwarz we have 
			\begin{subequations}
				\begin{align}
					\scal{(\u,\v)}{M(\u,\v)}
					&=\scal{(\u,\v)}{\left(\tfrac{1}{\sigma}\u-\lin^*\v,-\lin \u+\tfrac{1}{\tau}\v\right)}
					\\
					&=\tfrac{1}{\sigma}\norm{\u}^2+\tfrac{1}{\tau}\norm{\v}^2-2\scal{\lin \u}{\v}\ge \tfrac{1}{\sigma}\norm{\u}^2+\tfrac{1}{\tau}\norm{\v}^2-2\norm{\lin}\norm{\u}\norm{\v}
					\\
					% &\ge \norm{u}^2+\norm{v}^2-\norm{A}(\norm{u}+\norm{v})^2=(1-\norm{A})\norm{(u,v)}^2.
					%correct version: &\ge \norm{u}^2+\norm{v}^2-\norm{A}(\norm{u}^2+\norm{v}^2)=(1-\norm{A})\norm{(u,v)}^2.
					&=\tfrac{1}{\sigma}\norm{\u}^2+\tfrac{1}{\tau}\norm{\v}^2-2\left(\tfrac{\sqrt[4]{\sigma \tau}\sqrt{\norm{\lin}}}{\sqrt{\sigma }}\norm{\u}\right)\left(\tfrac{\sqrt[4]{\sigma \tau}\sqrt{\norm{\lin}}}{\sqrt{\tau}}\norm{\v}\right)
					\\
					&\ge \tfrac{1}{\sigma}\norm{\u}^2+\tfrac{1}{\tau}\norm{\v}^2-\tfrac{\sqrt{\sigma \tau}\norm{\lin}}{\sigma}\norm{u}^2 - \tfrac{\sqrt{\sigma \tau}\norm{\lin}}{\tau} \norm{\v}^2
					\\
					&\geq \min \left\{\tfrac{1}{\sigma}, \tfrac{1}{\tau}\right\}(1-\sqrt{\sigma \tau}\norm{\lin})\norm{(\u,\v)}^2.
				\end{align}
			\end{subequations}
			That is, $M$ is $\left( \min \left\{\tfrac{1}{\sigma}, \tfrac{1}{\tau}\right\}(1-\sqrt{\sigma \tau}\norm{\lin})\right)$-strongly monotone.
			Consequently, $M$ is $3^*$ monotone 
   by, e.g., \cite[Example~25.15(iv)]{BC2017}.
			Now combine this with \cite[Proposition~25.16~and~Example~22.7]{BC2017}
			to learn 
			that $M^{-1}$ is strongly monotone.
			\cref{prop:M:i}:	
			Combine \cref{prop:M:0} and \cite[Proposition~22.11(ii)]{BC2017}.
			\cref{prop:M:ii}:	
			This is a direct consequence of \cref{prop:M:0}.
            \cref{prop:M:i:ii}:
            Combine \cref{prop:M:i} and \cref{prop:M:ii}.
			\cref{prop:M:iii}:	
			Using \cref{prop:M:0} we have  $M$ and $M^{-1}$ are monotone.
			Now combine this with \cite[Example~20.34]{BC2017}.
 \end{myproof}

 		\begin{myproof}[Proof of \cref{prop:T:prop}.]
        
			\cref{prop:T:prop:i}:
			Indeed, we have 
			$T=(M+{\partial F}+S)^{-1}M
			=(M^{-1}(M+{\partial F}+S))^{-1}
			=(\Id+M^{-1}({\partial F}+S))^{-1}$.
			\cref{prop:T:prop:ii}:
			Combine \cref{prop:T:prop:i},
			\cref{lem:maxmonosum:op},
			\cref{prop:M}\cref{prop:M:0}
			and, e.g.,
			\cite[Proposition~23.10]{BC2017}.
			\cref{prop:T:prop:iii}:
			It follows from 
			% \cite[Corollary~7.2]{BYW} 
			\cite[Proposition~23.34(iii)]{BC2017}
			in view of \cref{prop:M}\cref{prop:M:0}
			that 
			\begin{subequations}
				\begin{align}
					\Id-T
					&=M^{-1} (\Id+M({\partial F}+S)^{-1})^{-1} M
=M^{-1}(M^{-1}+({\partial F}+S)^{-1})^{-1} 
					\\
					&=(\Id+({\partial F}+S)^{-1}M)^{-1} .
				\end{align}
			\end{subequations}
			\cref{prop:T:prop:iv}:
			Using \cref{prop:T:prop:iii} we have 
			\begin{subequations}
				\begin{align}
					\ran(\Id-T)
					&=\ran (\Id+({\partial F}+S)^{-1}M)^{-1} =\dom (\Id+({\partial F}+S)^{-1}M)
					\\
					&=\dom (({\partial F}+S)^{-1}M) =\ran (({\partial F}+S)^{-1}M)^{-1}
					\\
					&=\ran (M^{-1}({\partial F}+S))=M^{-1}(\ran ({\partial F}+S)),
				\end{align}
			\end{subequations}
			where in the last identity we used \cref{prop:M}\cref{prop:M:i:ii}.
			\cref{prop:T:prop:v}:
			It follows from \cref{prop:T:prop:iv}  that
            $\cran (\Id-T)=\overline{M^{-1}(\ran (\partial F+S))}
            =\overline{M^{-1}(\cran (\partial F+S))}=M^{-1}(\cran (\partial F+S))$.
            Here the second identity follows from applying 
            \cref{eq:lin:cl} with $X$ replaced by $X\times Y$,
            $L$ replaced by $M^{-1}$ and $D$ replaced by $\ran (\partial F+S)$, 
            while the third identity follows from \cref{prop:M}\cref{prop:M:i:ii}.
			\cref{prop:T:prop:vi}:
			Combine \cref{prop:T:prop:v}
			and \cref{thm:ranform}.
		\end{myproof}
        
		\begin{myproof}[Proof of \cref{thm:ranform}.]
        
			Recall that by  \cite[Corollary~16.30]{BC2017} we have
			$(\partial f^*, \partial g^*)=((\partial f)^{-1}, (\partial g)^{-1})$.
			Moreover, by e.g., \cite[Corollary~16.39]{BC2017} we have
			%Proposition~16.27~and~
			\begin{equation}
				\label{e:BronR}
				%	\text{$\dom \partial f \subseteq \dom f$ \;\;and\;\; $\overline{\dom}\  \partial f =\overline{\dom} f$.}    
				\overline{\dom}\  \partial f =\overline{\dom} f.
			\end{equation}
It follows from \cref{eq:lin:cl} applied with $L$ replaced by $\lin$
 and $D$ replaced by $\dom \partial f$
and again by  $D$ replaced by $\dom  f$
in view of \cref{e:BronR} that 
 $\overline{\lin (\dom \partial f)}
 =\overline{\lin (\overline{\dom}\ \partial f)}
  =\overline{\lin (\overline{\dom}\  f)}
 =\overline{\lin (\dom  f)}$.
Similarly, we conclude that
$\overline{\lin^* (\dom \partial g^*)}=\overline{\lin^* (\dom  g^*)}$.
It therefore follows from 
			\cref{thm:ranform:mm}, applied with $(B_1,B_2)$
			replaced by $(\partial f, \partial g)$,
			in view of \cref{fact:subd3*} that
			\begin{subequations}
				\begin{align}
					\cran{(\partial F+S)}
					&= {(\overline{\ran \partial f+\lin^* (\ran (\partial g^*)^{-1})})} \times 
					{(\overline{\dom \partial g -\lin (\dom \partial f)})}\\
					&= {(\overline{\dom \partial f^*+\lin^* (\dom \partial g^*)})} \times 
					{(\overline{\dom \partial g -\lin (\dom \partial f)})}
                \\
                    &={(\overline{\overline{\dom}\  \partial f^*+\overline{\lin^* (\dom \partial g^*)}})} \times 
					{(\overline{\overline{\dom}\ \partial g -\overline{\lin (\dom \partial f)}})}\\
					&= {(\overline{\overline{\dom}\  f^*+\overline{\lin^* (\dom g^*)}})} \times 
					{(\overline{\overline{\dom}\ g -\overline{\lin (\dom f)}})}
					\\
					&={(\overline{\dom  f^*+\lin^* (\dom g^*)}}) \times 
					{(\overline{\dom g - \lin (\dom f)}}).
				\end{align}
			\end{subequations}
			The proof is complete.
		\end{myproof}

		\begin{myproof}[Proof of \cref{p:vobang}.]
        
			Indeed, let $\y\in \cran (\partial F+S)$.
   Recalling 
			\cref{prop:T:prop}\cref{prop:T:prop:v},
			it follows from the Projection Theorem 
			see, e.g., \cite[Theorem~3.16]{BC2017},
			that 
\begin{equation}
	     \overline{\w}=\v
\siff\
			\scal{\overline{\w}}{\overline{\w}-M^{-1}\y}_M\le 0
   \siff
	\scal{\overline{\w}}{M\overline{\w}-MM^{-1}\y}\le 0 
\siff
	   \scal{\overline{\w}}{M\overline{\w}-\y}\le 0. 
	\end{equation}  
      This verifies \cref{eq:vobang}.
Now, let $\r\in \overline{(\dom f^*+\lin^* (\dom g^*))}$
 and let $\d\in \overline{(\dom g -\lin(\dom f))}$.
 Then $\y:=(\r,\d)\in \cran (\partial F+S)$
 by \cref{thm:ranform}.
 This and \cref{eq:vobang} imply that 
 \begin{equation}
 \label{e:vobangi}
\scal{\vx}{\tfrac{1}{\sigma}\vx-\lin^*\vy-\r}+
\scal{\vy}{\tfrac{1}{\tau}\vy-\lin\vx-\d}\le 0.
\end{equation}
Finally, observe that \cref{e:def:v:projm} and 
\cref{prop:T:prop}\cref{prop:T:prop:v}
imply
\begin{equation}
\label{e:vobangii}
M\v=(\tfrac{1}{\sigma}\vx-\lin^*\vy,\tfrac{1}{\tau}\vy-\lin\vx)
\in \cran(\partial F+S).
\end{equation}
\cref{eq:vobang:i}:
Apply \cref{e:vobangi} with $\d$ replaced by $\tfrac{1}{\tau}\vy-\lin\vx$
in view of \cref{e:vobangii}.% and \cref{thm:ranform}.
\cref{eq:vobang:ii}:
Apply \cref{e:vobangi} with $\r$ replaced by $\tfrac{1}{\sigma}\vx-\lin^*\vy$
in view of \cref{e:vobangii}.% and \cref{thm:ranform}.
\end{myproof}

\section{} 
\label{App:3}	

 \begin{myproof}[Proof of \cref{lem:quadratic_conic_V}.]   
 
In view of \cref{rem:polyhedK}, it is sufficient to verify the 
            first identity in \cref{lem:eq:qp_V}.
              Let $(\r, \d) \in V$. Then there exists $ (\w, \y)\in X\times Y$ such that 
              \begin{equation}
              \begin{array}{rl}
              &\frac{1}{\sigma}\r - \lin^\ast \d = H \r -H \w + \lin^\ast \y + \c  - \lin^\ast \d \in \ran H + \lin^\ast(K^\ominus) + \c,\\
              &\frac{1}{\tau}\d - \lin \r = \frac{1}{\tau}\d - (\lin \w + \b) + (\lin \w + \b) - \lin \r \in K + \ran \lin + \b.\\
              \end{array}
              \end{equation}
              Recalling \cref{rem:polyhedK}, for simplicity we set $R :=\ran (\partial F+S)
              =(\ran H + \lin^\ast(K^\ominus) + \c) \times (K + \ran \lin + \b)$.
              The inclusion $V \subseteq M^{-1}R$ follows from the definition of $M$ given in \cref{eq:def:M} and the nonsingularity of $M$ due to the choices of $\tau, \sigma$. We now show that $ M^{-1}R\subseteq V$.
              Indeed let $\z:=(\z_1,\z_2) \in R$.
Then there exist $\x \in X, \s \in K^\ominus, \t \in K, \u \in X$ such that
              \begin{equation}
              \label{e:lem:qp_rform:aux}
                  \z  =(\z_1,\z_2)= (H\x + \lin^\ast \s + \c, \t - \lin \u + \b) .
              \end{equation}
              We claim that there exist $\x^* \in X, \s^* \in K^\ominus, \t^* \in K, \u^* \in X$ such that
              \begin{equation}
              \label{e:lem:qp_rform}
                  \z  = (H\x^* + \lin^\ast \s^* + \c, \t^* - \lin \u^* + \b) 
                  \;\;\text{and}\;\; H\x^*=H\u^*.
              \end{equation}
To this end  consider the problem:
              \begin{equation}
                \begin{array}{rl}
                \minimize{\overline \u\in X}& \tfrac{1}{2} \scal{\overline \u}{H \overline \u}
                - \scal{\z_1-\c}{ \overline \u} \\
                \mbox{\rm subject to}& \lin \overline \u - \b+\z_2 \in K.
                \end{array}
                \label{eq:auxqp}
            \end{equation}
            Standard techniques yield that the Lagrangian dual of \cref{eq:auxqp} is
            \begin{equation}
            \begin{array}{rl}
            \maximize{\bar\s\in Y,\bar\x\in X} & - \frac{1}{2} \scal{\overline \x}{H \overline \x} - \scal{\b-\z_2}{ \overline \s}\\
            \mbox{\rm subject to} & - \lin^\ast \overline \s + \z_1 - \c = H\overline \x\\
            & \overline \s \in K^\ominus.
            \end{array}
            \label{eq:auxqpdual}
            \end{equation}
            It follows from \cref{e:lem:qp_rform:aux}
            that  $\u$ satisfies the primal constraint and  the pair $(\x, \s)$ satisfies the dual constraints.
            Therefore, because  $K$ is a polyhedral cone,
            strong duality holds for the primal-dual problem \cref{eq:auxqp}--\cref{eq:auxqpdual}
            (see, e.g., \cite[Comment~on~Page~227]{BV}). 
            Let $(\u^*, (\x^*, \s^*))$ denote its primal--dual optimal solution. 
            Then there exists $(\t^* , \s^*)\in K\times K^\ominus$ such that 
            \begin{equation}
            \label{eq:pdfeas}
            \text{$\z_1 = H \x^* + \lin^\ast \s^* + \c$\; (dual feasibility) \;and\;   $\z_2 = \t^* - \lin \u^* + \b$\; (primal feasibility).}
            \end{equation}
            Moreover, strong duality and KKT conditions imply that 
$            \bzero = H \u^* - (\z_1 - \c) + \lin^\ast \s^* 
            = H  \u^* - H  \x^*.$
           This proves  \cref{e:lem:qp_rform}.
            Now define 
            $(\r, \d) := M^{-1} \z$, i.e., 
            $(\z_1,\z_2)= M(\r,\d)=\left(\tfrac{1}{\sigma}\r - \lin^\ast \d, \tfrac{1}{\tau}\d - \lin \r \right)$. This and \cref{eq:pdfeas} implies 
\begin{equation}
\label{e:zrd}
\left(\tfrac{1}{\sigma}\r , \tfrac{1}{\tau}\d  \right)=( H \x^* + \lin^\ast \s^* + \c+\lin^\ast \d,\t^* - \lin \u^* + \b+\lin \r).
\end{equation}
Define
            $\w:=\r - \u^*$ and $\y:=\d + \s^*$. In view of  \cref{e:lem:qp_rform}, \cref{e:zrd} and 
            \cref{eq:pdfeas}
            we have 
              $$
              \begin{array}{rl}
                &\frac{1}{\tau} \d - (\lin \w + \b) = \lin \r+ \t^* - \lin \u^* + \b - (\lin \w + \b) = \t^* \in K\\
                &\y - \d = \s^* \in K^\ominus\\
                &\frac{1}{\sigma}\r - H \r = \lin^\ast \d + H\x^* + \lin^\ast \s^* + \c - H \r 
                = -H \w + \lin^\ast \y + \c,
              \end{array}
              $$
              which implies that $(\r, \d) \in V$. 
              The inclusion $M^{-1} R \subseteq V$ follows from the construction $(\r, \d) = M^{-1} \z$.
 \end{myproof}
\end{appendices}
\end{document}